\documentclass[reqno]{amsart}

\usepackage{amssymb,amsmath,amscd,amsthm}

\usepackage{mathrsfs}
\usepackage{latexsym}
\usepackage{graphicx}
\usepackage{breakcites}
\usepackage[utf8]{inputenc}
\usepackage[ngerman,english]{babel}
\usepackage{color}
\usepackage{verbatim}
\usepackage{enumitem}
\usepackage{hyperref}

\newtheoremstyle{myplain}{}{}{\it}
{0pt}{\scshape}{}{ }{\thmname{#1}\thmnumber{ #2}\thmnote{ (#3)}}
\newtheoremstyle{mydefinition}{}{}{}
{0pt}{\scshape}{}{ }{\thmname{#1}\thmnumber{ #2}\thmnote{ (#3)}}

\theoremstyle{myplain}

    \newtheorem{Def}{Definition}[section]
        \newtheorem{Lem}[Def]{Lemma}

        \newtheorem{theo}[Def]{Theorem}
        \newtheorem{prop}[Def]{Proposition}
        \newtheorem{rem}[Def]{Remark}
        \newtheorem{hyp}[Def]{Hypothesis}
        \newtheorem{cor}[Def]{Corollary}
        \newtheorem{claim}[Def]{Claim}

\numberwithin{equation}{section}

 \DeclareMathOperator{\Id}{Id}

\renewcommand{\Im}{\operatorname{Im}}
\newcommand{\im}{\operatorname{im}}
 \newcommand{\N}{\mathbb{N}}
\newcommand{\Z}{\mathbb{Z}} 
\newcommand{\R}{\mathbb{R}} 
\newcommand{\C}{\mathbb{C}} 
 \renewcommand{\P}{\mathbb{P}} 
 \renewcommand{\Id}{\mathbf{Id}}

\newcommand{\Cont}{\mathcal{C}} 
\renewcommand{\O}{\mathcal{O}} 
\newcommand{\T}{\mathbb{T}} 
\newcommand{\s}{\mathtt{s}} 
\newcommand{\Op}{\mathbf{Op}} 
\newcommand{\COp}{\mathbf{COp}}

\renewcommand{\P}{\mathbf{P}} 
\newcommand{\p}{\mathbf{p}}
\newcommand{\U}{\mathbf{U}}
\newcommand{\aac}[1]{\widetilde{#1}}
\newcommand{\trn}[1]{\left\| {#1} \right\|_{\mathtt{tr}}}
\newcommand{\Neig}{\mathcal{N}}
\newcommand{\D}{\mathcal{D}}
\newcommand{\F}{F} 
\newcommand{\Fd}{\mathcal{F}} 
\newcommand{\dom}{\mathsf{dom}}
\newcommand{\jap}[1]{\left\langle #1\right\rangle}
\newcommand{\abs}[1]{\left| #1\right|}
\newcommand{\norm}[1]{\left\| #1 \right\|}
\newcommand{\torn}[1]{\norm{ #1 }_{\L^2(\T^d)}}

\newcommand{\dualp}[2]{\left( #1 , #2 \right)}
\newcommand{\skp}[2]{\left\langle #1 , #2 \right\rangle}
\newcommand{\dz}{\overline{\partial}}

\renewcommand{\l}{\ell}
\renewcommand{\L}{\mathbf{L}}
\newcommand{\ess}{>_{\mathtt{ess}}}
\newcommand{\adjust}[1]{\underline{#1}}
\newcommand{\adjOP}[1]{\underline{#1}}
\newcommand{\I}{\mathbf{1}}
\newcommand{\compw}{\#}
\newcommand{\comp}[1]{\#_{#1}}

\newcommand{\Sy}{\mathtt{S}} 
\renewcommand{\S}{\mathcal{S}} 

\newcommand{\supp}{\operatorname{supp}}

\newcommand{\sign}{\operatorname{sign}}

\newcommand{\tr}{\operatorname{tr}}

\newcommand{\vol}{\operatorname{vol}_\T}
\renewcommand{\div}{\operatorname{div}}

\newcommand{\contr}{\mathbin{\lrcorner}}

\title{Weyl Asymptotics for Pseudodifferential Operators in a
  Discrete Setting}

\author{Markus Klein \and Enrico Reiß \and Elke Rosenberger}

\address{Universität Potsdam\\ Institut für Mathematik \\
  Karl-Liebknecht-Str. 24-25\\ 14476 Potsdam}
\email{mklein@math.uni-potsdam.de, enreiss@uni-potsdam.de,
  elke.rosenberger@uni-potsdam.de}

\date{\today}

\begin{document}
  \begin{abstract}
  We prove a {\em sharp} Weyl estimate for the number of eigenvalues
  belonging  to a fixed interval of energy of a self-adjoint difference operator acting on $\l^2(\epsilon\Z^d)$ if the associated symplectic
  volume of phase space in $\R^d \times \T^d$ accessible for the Hamiltonian flow of the principal symbol is finite. 
 Here $\epsilon$ is a semiclassical parameter. Our proof depends crucially on the construction of a good semiclassical approximation 
 for the time evolution induced by the self-adjoint operator
on $\l^2(\epsilon\Z^d)$. This extends previous semiclassical results to a broad class of difference operators on a scaled lattice.  
  \end{abstract}

\maketitle


\section{Introduction and main results}

Weyl asymptotics express the leading order of the number of eigenvalues in a certain range of energy of a self-adjoint differential or
pseudodifferential operator in terms of the symplectic volume in phase space which is accessible for the associated Hamiltonian flow
induced by the principal symbol of the operator. By phase space we shall always denote a symplectic space.

 Weyl asymptotics  go back to the classical work of Weyl, see \cite{Weyl1911,Weyl1912}, and have since been refined and generalized 
 in  many papers. These asymptotics are always semiclassical in nature, although they both exist in an appropriate high energy 
 version (as in the original work of Weyl) or a purely semiclassical  version containing a small parameter which in physics terms 
 might be identified with Planck's constant $h$.

In this paper we investigate a discrete version of these estimates for a class of self-adjoint difference operators on the  Hilbert space 
$\l^2(\epsilon\Z^d)$.  Here the lattice spacing 
$\epsilon$ plays the role of the semiclassical parameter $h$, similar to previous work of Klein-Rosenberger, 
see \cite{KR08,KR09,KR11,KR12,KR16,KR18} on the asymptotics of individual eigenvalues in the semiclassical limit for such operators.
Such an operator may be both  written as a superposition of translation operators on the scaled lattice or as a discrete type 
of pseudodifferential operator associated to a symbol $a(x,\xi)$ on phase space which is periodic with respect to the momentum 
variable $\xi$, using a discrete quantisation rule of Weyl-type, namely

\begin{align}
  \label{discrweyl}
  \left(\Op_{\epsilon,\frac{1}{2}}^\T a\right)u(x) := \frac{1}{(2\pi)^d} \sum_{y
    \in \epsilon \Z^d} \int_{\T^d} e^{i(y-x)\xi/\epsilon} a\Bigl(\frac{1}{2}(x
  + y), \xi; \epsilon\Bigr) u(y) d\xi \qquad (x \in \epsilon\Z^d).
\end{align}

For more detail  on these pseudodifferential operators 
(including a rigorous definition providing sense to the possibly diverging sum in the above expression)
and the associated  spaces of symbols used in this paper we refer to our 
Appendix A. For the relation of these operators to a superposition of translation operators see \cite{KR08}. We shall, however, 
stick exclusively to the representation of the relevant operator in the form given in equation (\ref{discrweyl}). This is best adapted to 
the microlocal character of  Weyl asymptotics.

We remark  that at least in our opinion it is not a priori clear what the relevant phase space for these operators actually is. The lattice 
does not have a symplectic cotangent bundle, but our symbols $a(x,\xi)$ are assumed to be functions on $\R^d \times \T^d$ which 
we shall sometimes  consider as  functions on $\R^{2d}$, periodic in $\xi$. It has to be proved that $\R^d \times \T^d$ (which 
is isomorphic to the cotangent bundle $T^* \T^d$, switching the space and momentum variables), actually is the relevant phase 
space for operators of the above type giving correct Weyl asymptotics. While the usual Weyl quantisation  
$\Op_{\epsilon,\frac{1}{2}} a$ (see equation
\eqref{eq:OpaNonDiscrDef} in Appendix A)  of our symbols $a$ gives well defined self-adjoint operators in $L^2(\R^d)$, these 
operators with naturally associated phase space $\R^{2d}$ do not under our assumptions below possess discrete spectrum, and 
the associated symplectic volume of phase space is actually infinite, due to periodicity in $\xi$. Thus the identification of  
$\R^d \times \T^d$ as the relevant phase space for the operators considered in this paper already is an important mathematical 
result. We also emphasise that the manifold $T^*\T^d$ in a geometric sense is a mildly more complicated object compared to 
the simplest possible phase space $\R^{2d}$. Thus it will be natural that we shall have to use basic theorems of analysis 
(e.g. the regular value theorem) in a manifold setting, and recalling the basic properties of the Liouville measure for regular 
hypersurfaces in $T^*\T^d$ is conveniently expressed in a slightly more geometric version (using the interior derivative) compared 
to some standard references for $\R^{2d}$.

 Conceptually this phenomenon of an interplay of the discrete lattice with smooth phase space as a manifold is in accordance with 
 the general results on fine semiclassical asymptotics on individual eigenvalues in our previous papers mentioned above.  
 We recall, for instance, that for a broad class of Hamiltonian functions on the phase space $\R^{2d}$  there is a naturally 
 associated Finsler metric on the configuration space $\R^d$, see \cite{KR08}. The associated geodesic Finsler distance then gives 
 the exponential decay rate for eigenfunctions of the associated difference operator on 
 $\l^2(\epsilon\Z^d)$,  and these precise decay rates are crucial for obtaining sharp tunnelling asymptotics for almost 
 degenerate eigenvalues of these difference operators. In some sense, a similar interplay
between classical mechanics on a smooth phase space and spectral properties of a self-adjoint operator  on the discrete 
{\em configuration space} 
$\epsilon \Z^d$ is also present in this paper. However, while not being unexpected, the identification of the correct phase space 
requires proof.
A first crucial result in this direction is contained in Chapter 3 of the present paper, where the symplectic phase space volume in 
$\R^d \times \T^d$ is related to the product of the counting measure on the scaled lattice $\epsilon \Z^d$ and the natural measure 
on the torus $\T^d$. This product measure arises from trace estimates.

We remark that in general there are many ways to obtain the leading order term  in Weyl asymptotics with only a weak estimate on 
the remainder term, and this also applies to the discrete setting of the present paper.  For partial results  in this direction (i.e. the 
lattice Laplacian with an added smooth potential)  see  \cite{Ka23} and our discussion at the end of Chapter 3.  It is, however, the 
main goal of the present paper to obtain 
in a general setting a {\em  sharp }  Weyl estimate, where the estimate on the remainder is improved by a factor $\epsilon$  
(the lattice spacing which is the relevant semiclassical parameter in our context) compared to the volume term in  leading order. 

Lastly, we recall that our interest in refined spectral asymptotics for difference operators originated from a treatment of metastability 
and the study of the spectrum of generators of Markov chains, where the state space is finite but its cardinality goes to infinity, 
see \cite{begk1}, \cite{begk2}  and the book \cite{bov}. In such a slightly different case (where the state space of the Markov chain 
is not necessarily a lattice) the notion of phase space is much less clear and we do not know of a good analog of sharp 
Weyl asymptotics.

For completeness sake, we mention the work of Nakamura and Tadano on long-range scattering for certain difference operators on 
$\ell^2(\Z^d)$, see \cite{na} and \cite{tad}. This work has similarities to the present work in developing analogies to older work 
on Schrödinger operators in $\R^d$ and working on the phase space $\R^d \times \T^d$. 
The theory, however, is not completely semiclassical (the lattice is not scaled by a semiclassical parameter. On a technical level,  
this allows to assume that the symbol of the operators is initially assumed to be a function  on $\Z^d \times \T^d$, which is 
then extended in a fixed and largely arbitrary way to $\R^d \times \T ^d$. 
In a fully semiclassical setting as in the present paper this does not seem to be possible.  However, it has long been known in the 
theory of Schrödinger operators on $\R^d$  that the potential being long-range requires modifications in scattering theory which 
are closely related to classical mechanics in phase space and asymptotics of the associated Hamilton-Jacobi equation and the 
classical flow, see e.g. \cite{ReSi3}. 

These properties are semiclassical in nature, even for an operator which does not explicitly contain a semiclassical parameter. 
See e.g. \cite{hscat} for the existence of wave operators and \cite{hoermander2009analysis} for an exposition of a very general 
scattering theory with long-range perturbations of an elliptic differential operator going back to work of Agmon, see \cite{a}.  
Even in \cite{ak} on radial Schr{\"o}dinger operators in $\R^d$ precise WKB asymptotics on the solution of the radial  
Schr{\"o}dinger equation (uniformly in a complex domain) are crucial.  

Possibly the most microlocal version of this phenomenon is in the work \cite{ik} where an approximation to the wave operator 
is constructed in the form of a special Fourier integral operator which approximately intertwines the two relevant unitary groups; using 
this as a time-independent modifier then gives a well defined scattering theory. This, however, is different from the semiclassical view 
on difference operators used in the present paper. Going  by analogy, it seems reasonable to expect the following results. 
For difference operators on an unscaled lattice $\Z^d $, possibly generalising the operators in \cite{na} and \cite{tad}, 
there are under appropriate conditions sharp Weyl estimates in the high-energy limit. 

On the other side, for an appropriate class of fully semiclassical difference operators with an existing short-range or long-range 
scattering theory a semiclassical version of the Isozaki-Kitada modifiers could be developed which would then give 
semiclassical expansions of the wave operator and possibly the scattering matrix. In spirit this should be close to the work of 
Robert and Tamura, see \cite{rt} on the  Schr{\"o}dinger operator in $\R^d$. 
It is known that these techniques  adapt well to  situations  which are different in a technical sense while being close conceptually; 
e.g. they have been extended in \cite{kmw} to cover the only partially semiclassical case of the Born-Oppenheimer approximation.
 Finally, we recall the recent paper  \cite{nakam} where it is shown that the set of resonances of a Schr{\"o}dinger operator  
 $- \Delta + V(x)$ in $\R^d$ is approximated in the semiclassical limit by the resonances of its discrete counterpart acting on the 
 scaled lattice $\epsilon \Z^d$.

To express our results in more detail, we shall need the following notation.

Using the notation of the  classes of symbols in  our Appendix A a symbol  $a \in \Sy^0(m,\epsilon_0)(\R^d \times \T^d)$  is called
\textit{$(m,\epsilon_0)$-elliptic} if for some $C > 0$
\begin{align}
    \label{eq:DefElliptic}
    \abs{a(x,\xi;\epsilon)} \geq Cm(x,\xi) \qquad (x \in \R^d,\, \xi
    \in \T^d,\, \epsilon \in (0,\epsilon_0]).
\end{align}

For real-valued symbols $a$ and $S \in \R$ we shall write
\begin{align}
  \label{eq:essLarger}
 a \ess S  \qquad \mbox{iff for
some} \quad R > 0 \qquad     
      \inf_{\substack{x \in \R^d,\, |x| > R,\, \xi \in \T^d
          \\ \epsilon \in (0,\epsilon_0]}} a(x,\xi;\epsilon) > S.
    \end{align}

Then the crucial hypothesis on our symbols which ensures self-adjointness
of the associated operator with its  spectrum   being discrete in an appropriate interval of energy is

\begin{hyp}
  \label{hyp:WRef}
  Let $a \in \Sy^0(m,\epsilon_0)(\R^d \times \T^d)$ be real-valued
  where the order function $m$ takes values only in
  $[1,\infty)$. Assume $a$ to satisfy
  \begin{enumerate}
  \item \label{hyp:WRef_EllItem} $a + i$ is $(m,\epsilon_0)$-elliptic,
  \item \label{hyp:WRef_EllItem2} $a \ess \sup J$ where $J \subset \R$ is a bounded open interval,
  \item $a(x,\xi;\epsilon) \sim \sum_{j=0}^\infty \epsilon^j
    a_j(x,\xi)$.
  \end{enumerate}
\end{hyp}

More precisely,
we shall consider the Hilbert space $\l^2(\epsilon\Z^d)$ of
square-summable functions on the $\epsilon$-scaled lattice
$\epsilon\Z^d$, equipped with the inner product
\begin{align}
  \label{eq:skpl2zd}
  \skp{u}{v} := \sum_{x \in \epsilon\Z^d} \overline{u(x)} v(x) \qquad
  (u,v \in \l^2(\epsilon\Z^d)).
\end{align}
We shall show in Proposition \ref{prop:PSelfAdj} that for any symbol
$a \in \Sy^0(m,\epsilon_0)(\R^d \times \T^d)$ satisfying Hypothesis
\ref{hyp:WRef} (\ref{hyp:WRef_EllItem}) and for any $\epsilon$
sufficiently small we can define the self-adjoint operator
  \begin{align}
    \label{eq:PepsDef}
       \P_\epsilon : \l^2(\epsilon\Z^d) \supset \D_\epsilon \rightarrow
      \l^2(\epsilon\Z^d), \quad u \mapsto \left( \Op^\T_{\epsilon,1/2}
        a \right)u
\end{align}
where $\D_\epsilon := \left( \Op_{\epsilon,1/2}^\T (a + i)
\right)^{-1}\left( \l^2(\epsilon\Z^d) \right)$.

Then the main result of this paper is

\begin{theo}
  \label{theo:RefinedWeyl}
  Let the interval $J$, the order function $m$ and the symbol $a$ with
  leading order symbol $a_0$ satisfy Hypothesis \ref{hyp:WRef}. Let
  $\alpha,\beta \in \R$ with $\alpha < \beta$ and $[\alpha,\beta]
  \subset J$. Suppose that $\alpha$ and $\beta$ are non-critical
  values of $a_0$. Denote by $\Neig([\alpha,\beta];\epsilon)$ the
  number of eigenvalues of $\P_\epsilon$ in $[\alpha,\beta]$. Then
  \begin{align}
    \label{eq:MainTheoremNeig}
    \Neig([\alpha,\beta];\epsilon) = \frac{1}{(2\pi \epsilon)^d}
    \left( \vol \left( a_0^{-1}([\alpha,\beta]) \right) +
      \O(\epsilon)\right) \qquad (\epsilon \downarrow 0).
  \end{align}
  Here, for a measurable set $A \subset \R^d \times \T^d$,
  \begin{align}
    \label{eq:phaseSpaceVol}
    \vol (A) := \int_{ A} dx d\xi.
  \end{align}
  denotes the symplectic volume.
\end{theo}

Already here we shall point out that this result, by conjugating with the unitary Fourier series expansion
$\Fd_\epsilon: L^2(\T^d) \to \ell^2(\epsilon \Z^d)$ defined 
  by
    \begin{align}\label{def:disFouTr}
      \Fd_{\epsilon} f (x) := \frac{1}{(2\pi)^{d/2}}
      \int_{\T^d} e^{-ix\xi/\epsilon} f(\xi) d\xi \qquad (x \in \epsilon\Z^d)
    \end{align}
and its inverse,  implies a sharp Weyl law for certain 
self-adjoint operators on $L^2(\T^d)$ which also is new to the best of our knowledge.
We shall amplify  further below.

We recall that to the best of our knowledge all results on Weyl asymptotics with a similarly sharp estimate on the remainder do 
require the construction of a good semiclassical approximation to the unitary time evolution operator
(a semiclassical time parametrix). This construction and its application turn out to also be the main technical result in the proof 
of Theorem \ref{theo:RefinedWeyl}. We emphasize that while the general ideas to construct such a parametrix are well known (and 
we have  chosen to follow the construction in the book \cite{DiSj} of Dimassi and Sj{\"o}strand  to some extent), it is in our case crucial 
to explicitly verify that all functions needed for the time parametrix (phase functions as well as amplitudes) are actually periodic in 
the momentum variable $\xi$, i.e. they are well defined on the phase space $\R^d \times \T^d$. This has forced us to recall the 
construction in detail, thus providing a complete proof which may be readily checked by a critical reader. Wherever possible, however, 
we have simply cited
known  results from the literature, simplifying our exposition. E.g., the functional calculus for our discrete operator 
$\Op_{\epsilon,\frac{1}{2}}^\T a$ discussed in Chapter 2 is not developed from scratch but is instead based on the known 
functional calculus for $\Op_{\epsilon,\frac{1}{2}} a$  as given in \cite{DiSj}.

We shall add a few remarks on the history of the subject and the literature.
In the high energy case, sharp Weyl asymptotics go back to H{\"o}rmander's paper \cite{Hoer1968Spectral} using a representation 
of the time parametrix in form of a Fourier integral operator, in a local form. These results were extended in \cite{chaz74} 
and \cite{DuGu75Spectrum} where it was shown that the global theory of Fourier integral operators gives control on all of the 
singularities of the Fourier transform of the spectral measure. This is  further expanded in the book \cite{ivrii1998} of Ivrii. 
See also \cite{hoermander2009analysis} for a short review.  Shubin's book \cite{shu} and the article \cite{shut} might also be helpful.

The standard semiclassical version of our Theorem \ref{theo:RefinedWeyl} is due to Chazarain, see \cite{Chaz1980}, in the case of 
a Schr{\"o}dinger operator with compact resolvent. For the general case see the paper of Helffer and Robert \cite{hero81} and 
the aforementioned book \cite{ivrii1998}. See also the paper \cite{HELFFER1983246} for a version of the functional calculus based 
on the Mellin transform and the book  \cite{Robert1987AutourDL} for an exposition of  sharp  Weyl asymptotics using the 
functional calculus based on the Mellin transform a semiclassical approximation for the time evolution operator.  
A more recent exposition can be found in \cite{DiSj}, where the construction of a pseudodifferential functional 
 calculus is based on the Helffer-Sj{\"o}strand formula involving the resolvent, thus replacing the use of the Mellin transform. 
 We did follow this exposition in the present paper.
 
 We furthermore remark that many of the references stated above study the influence of the finer structure of closed orbits for 
 the Hamiltonian flow on the distribution of eigenvalues, in particular on the existence of two term asymptotics  if the Liouville 
 measure of the closed orbits is zero in the boundary hypersurfaces in phase space or the phenomenon of clustering. 
 In addition, there is a collection of papers which explicitly focus on such problems taking Weyl asymptotics for granted. 
 For instance, there is a connection between integrability of the Hamiltonian flow (and degeneracy of the geodesic length spectrum)
 versus ergodicity of this flow with the phenomenon of clustering of eigenvalues for the corresponding Hamiltonian. 
 Results on clustering go back e.g. to \cite{DuGu75Spectrum} and \cite{cdv1}, while the properties of eigenfunctions in the 
 ergodic case  for the high energy limit were studied by Shnirelman in \cite{sh}), were extended by Colin de Verdière (see e.g. \cite{cdv2})
 and developed in a semiclassical setting for pseudodifferential operators by Helffer, Martinez and Robert in \cite{hmr}. 
 We expect that a similar relation between spectral properties and the fine structure of the Hamiltonian flow is also present for the kind 
 of difference operators studied in this paper.  Proving this, however, is an open problem. 
 
 For completeness sake we mention the recent work of Ivrii \cite{ivmonster} which contains a plethora of interesting results on 
 various  aspects of Weyl asymptotics and relations between the different  types of limits involved, i.e.  semiclassical, high energy 
 limit  and approaching the ionisation threshold (the infimum of the essential spectrum) from below. While formally different, we 
 consider these types of limit as being semiclassical in nature. This follows old folk wisdom from the physics literature on the validity of 
 the {\em correspondence principle} and is (at least in parts) amplified in \cite{ivmonster}.  To the best of our knowledge, many of 
 these topics have not been analysed in a discrete setting as for a class of operators similar to those considered in this paper.

We shall finally comment on a series of papers by Rushansky et al.,  
see e.g.  \cite{RuzhTuru2010,bkr,bcr}, which 
treat operators on $L^2(\T^d)$, $\Cont(\T^d)$ or on an associated series space on $\Z^d$ or $h \Z^d$ using conjugation with the unitary 
Fourier series transform in \eqref{def:disFouTr}.
There actually is substantial overlap in the calculus of discrete pseudodifferential operators where our calculus, as indicated above, has 
actually been developed earlier,  going back to the thesis \cite{thesis}.

 These authors have opted for a partially discrete "phase space" $ \Z^d \times \T^d$, or $h \Z^d \times \T^d$. Thus their symbols are 
 complex valued functions $\sigma$ on $ \Z^d \times \T^d $ or $\sigma_h$ on $h\Z^d \times \T^d$. In this setting only the $t=0$ and 
 $t=1$ quantisation are a priori well defined (since for $x, y$ belonging to the lattice $t x+(1-t)y$ in general does not), and the authors actually 
 only use one of those.  Using that the discrete Fourier series transform $\Fd_{\epsilon}$ (see \eqref{def:disFouTr}) is implicit in our definition 
 of  $\Op_{\epsilon,t}^\T a$
in equation \eqref{eq:OpTa_u_Def} - most notably for $t=0,1$ - one readily checks that for $t=0,1$
and a symbol $a \in \Sy^0(m,\epsilon_0)(\R^{d} \times \T^d)$ one has
\begin{align}
\label{conj}
\Fd_\epsilon  \left(\Op_{\epsilon,t}^\T a\right) \Fd_\epsilon^{-1} = \left(\COp_{\epsilon,1-t} \sigma_\epsilon \right)
\end{align}
where 
$\COp_{\epsilon,1-t} \sigma_\epsilon$                                                                                                                                                                                                                                                                                                                                                                                                                                                                                                                                                                                                                                                                                                                                                                                                                                                                                                                                                                                                                                                                                                                                                                                                                                                                                                                                                                                                                                         
is an operator of the type considered in \cite{bcr} and 
$\sigma_\epsilon(x,\xi) = a(x, -\xi,\epsilon)$
for $x \in \epsilon \Z^d, \xi \in \T^d$ and any fixed $\epsilon$. 

We emphasise, however, that in the framework of \cite{bcr} there is never 
uniformity with respect to the semiclassical parameter in $\sigma_h$, and as a consequence, there are nowhere semiclassical expansions 
in powers of the semiclassical parameter (with remainder estimates uniform in the semiclassical parameter), neither for the symbolic 
calculus or the adjoint or the transposed operator. 
There are asymptotic expansions, but they are always expansions in symbol classes (of a Hörmander $S^m_{\rho, \delta}$ type, with 
$\delta$ strictly smaller than $\rho$), for any fixed $h$.  
Our change of quantisation formula (a semiclassical expansion for the symbol in powers of $\epsilon$ or $h$) is absent as is the 
intertwining formula with the standard quantisations of  symbols in $T^*\R^d$. As far as we can see there is no discussion of 
self-adjointness which is possibly natural in a context where the Weyl quantisation is not available and thus, in particular, there is 
no pseudodifferential spectral calculus (which in the context of \cite{bcr} needs to be developed from scratch, since the intertwining 
property with operators  in the continuum is absent, which forbids, as  in our paper, to use the known properties of the pseudodifferential 
spectral calculus in $T^*\R^d$  as a convenient input). Furthermore, even the leading volume term in the Weyl asymptotics is not only 
not formally defined for a symbol of type $\sigma_h$, but the absence of any uniform control with respect to the semiclassical parameter 
seems to make it impossible to extract the volume term by a limiting procedure. 

Thus, to the best of our knowledge and understanding, there are no standard results for operators on $L^2(\T^d)$ which would imply 
our Theorem on sharp Weyl asymptotics, nor would such results be simple to obtain in the context of a partially discrete phase space, 
since crucial basic techniques do not seem to be available.

On the other hand, given a symbol $\sigma$ in our class $\Sy^0(m,\epsilon_0)(\R^{d} \times \T^d)$,
one may consider the operator $\COp_{\epsilon,t} \sigma_\epsilon$ for                                                                                                                                                                                                                                                                                                                                                                                                                                                                                                                                                                                                                                                                                                                                                                                                                                                                                                                                                                                                                                                                                                                                                                                                                                                                                                                                                                                                                   
$t=0,1$, initially defined  on smooth functions on the torus. Assume that this operator is essentially self-adjoint.
Then, as described above, its self-adjoint realisation in $L^2(\T^d)$ is unitarily equivalent to an operator
$A=\Op^\T_{\epsilon, 1-t} a$ in $\ell^2(\epsilon \Z^d)$ with, in general,  a non-real symbol 
$a\in \Sy^0(m,\epsilon_0)(\R^{d} \times \T^d)$. 
Using our semiclassical change of quantisation formula given in Prop.\ref{prop:SwitchQuant},  
$A$ can be written as the Weyl-quantisation of a real semiclassical Weyl-symbol 
$a^W\in \Sy^0(m,\epsilon_0)(\R^{d} \times \T^d)$, where the $\epsilon$-principal 
symbols of $a^W$ and $a$ coincide, i.e. $a_0 = a_0^W$.
If $a^W$ satisfies Hypothesis \ref{hyp:WRef}(1),(2), then Theorem \ref{theo:RefinedWeyl}
immediately implies sharp Weyl asymptotics for the corresponding self-adjoint operator in $L^2(\T^d)$. 
Furthermore, using some more technical results from our calculus for discrete pseudodifferential operators, in this
setting it is actually sufficient to impose an analog of Hypothesis \ref{hyp:WRef}(2) only
on the principal symbol $a_0 = a_0^W$, which is directly given by the original symbol 
$\sigma$. At least if $\sigma + i$ is assumed to be $m$-elliptic, the initial assumption of 
self-adjointness then gives a real Weyl symbol $a^W$ and $m$-ellipticity gives control on the lower terms in the
asymptotic expansion of $a^W$. Thus, if $a_0 = a_0^W$ satisfies Hypothesis \ref{hyp:WRef}(2), $a^W$ also does for $\epsilon$ sufficiently small. We leave further details to the reader.
To the best of our knowledge this sharp Weyl estimate for  the operator $\COp_{\epsilon,t} \sigma_\epsilon$ on the torus is a new result.

 The outline of this paper is as follows. In Chapter 2 we treat questions of invertibility for our operators on the lattice based on 
 known results for the operators $\Op_{\epsilon,\frac{1}{2}} a$. Combined with a proof of self-adjointness this gives control of 
 the resolvent and a functional calculus, using the results in \cite{DiSj}. In Chapter 3 we develop the necessary trace estimates, 
 which in our discrete setting turn out to be slightly more direct than the corresponding estimates in \cite{DiSj}. We indicate how 
 these preliminary results could be used for  proving weaker Weyl asymptotics (with much less effort).
 Chapter 4 contains the proof of Theorem \ref{theo:RefinedWeyl}. Here we construct a semiclassical time parametrix in terms of 
 functions (phase function and amplitudes) defined on the relevant phase space $\R^d \times \T^d$.  This is the crucial point. 
 We have tried hard  to give a complete exposition proving all our claims while avoiding being tedious in unnecessarily exposing 
 well known results. The judgement on this, of course, is for the reader.
 Finally, in Appendix A we have for the convenience of the reader collected from previous work
 results on the pseudodifferential calculus for operators $\Op_{\epsilon,\frac{1}{2}}^{\T} a$. 
 These results are not new, but crucial for our exposition. Finally, Appendix B contains the results on Poisson summation which we 
 need to control the continuum approximation of discrete sums on the scaled lattice appearing in our proofs. 
 Here we are indebted to discussions with Giacomo di Gesù, see also \cite{gdg2} and \cite{gdg1}.

\section{Invertibility and functional calculus}

In this section, we shall develop a functional calculus for
pseudodifferential operators in the discrete setting based on the
resolvent and the formula by Helffer and Sjöstrand (\cite[Theorem 8.1]{DiSj}). We shall first treat the problem of invertibility in a
general context in Subsection \ref{sec:Invert}. In Subsection
\ref{sec:FuncCalc}, we then construct the self-adjoint realisation
$\P_\epsilon$ and give a functional calculus for $\P_\epsilon$. As far
as possible, we try to derive our statements from the standard theory
of pseudodifferential operators in the non-discrete setting,
e.g. given in \cite[Chapter 8]{DiSj}.  Otherwise we adapt the proofs
to our setting, using previous results in \cite{KR09, KR18}.

\subsection{Invertibility}
\label{sec:Invert}

In this subsection we shall construct the inverse operator of the discrete $t$-quantisation
$\Op_{\epsilon,t}^\T a$ for a given $(m,\epsilon_0)$-elliptic symbol
$a \in \Sy^0(m,\epsilon_0)(\R^d \times \T^d)$ and show that the
inverse also has a representation as a pseudodifferential operator of
discrete type. This result is stated in Proposition
\ref{prop:discrPseudoInvShort}. The proof is reduced to the
non-discrete setting by using Lemma \ref{Lem:InvSymbPeriodic} where
the symbol of the inverse operator in the non-discrete setting is
identified as the symbol of the inverse operator in the discrete
setting.

\begin{Lem}
  \label{Lem:InvSymbPeriodic}
  Let $a \in \Sy^0(m,\epsilon_0)(\R^d \times \T^d)$ and assume
  $\Op_{\epsilon, t} a$ to be invertible as a map $\S(\R^d)
  \rightarrow \S(\R^d)$ for $\epsilon \in (0,\epsilon_0]$ with
    \begin{align}
      \label{eq:InvSymPeriodicCond}
      \left( \Op_{\epsilon,t} a \right)^{-1} = \Op_{\epsilon,t} b_a
    \end{align}
    for some $b_a \in \Sy^0(m^{-1},\epsilon_0)(\R^d \times \R^d)$. Then
    $b_a$ is periodic with respect to $\xi$ and $\Op_{\epsilon,t}^\T a$
    is invertible as a map $\s(\epsilon\Z^d) \rightarrow
    \s(\epsilon\Z^d)$ for $\epsilon \in (0,\epsilon_0]$ with
      \begin{align}
        \label{eq:discrPseudoInvSymbol2}
        \left(\Op_{\epsilon,t}^\T a\right)^{-1} = \Op_{\epsilon,t}^\T
        b_a.
      \end{align}
\end{Lem}

\begin{proof}
  Let $\gamma \in 2\pi\Z^d$ and define the shifted symbol
  $b_a^\gamma(x,\xi;\epsilon) := b_a (x, \xi +
  \gamma;\epsilon)$. By a straightforward calculation, using the
  definition of the $t$-quantisation in
  formula \eqref{eq:OpaNonDiscrDef},
  one obtains
  \begin{align}
    \Op_{\epsilon,t} b_a^\gamma = M_\gamma \circ
    \Op_{\epsilon,t} b_a \circ M_{-\gamma},
  \end{align}
  where following the usual slight abuse of notation $M_\gamma (x) := e^{i\gamma x /\epsilon}$ denotes the corresponding multiplication operator.
  Using periodicity of the symbol $a$, we get
  \begin{align}
    \Op_{\epsilon,t} a \circ M_\gamma = M_\gamma \circ \Op_{\epsilon,t} a.
  \end{align}
  We therefore conclude
  \begin{align}
    \label{eq:InverseSymbolEquation}
    \Op_{\epsilon,t} a \circ \Op_{\epsilon, t} b_a^\gamma = \Id =
    \Op_{\epsilon,t} b_a^\gamma \circ \Op_{\epsilon, t}a .
  \end{align}
  So, by uniqueness of the inverse operator
  \begin{align}
    \Op_{\epsilon,t} b_a = \Op_{\epsilon,t} b_a^\gamma.
  \end{align}
  From the standard theory (see \cite[Chapter 7]{DiSj},
  considering the Schwartz kernel of a general operator $\S(\R^d)
  \rightarrow \S'(\R^d)$) it follows that $b_a =
  b_a^\gamma$, i.e. $b_a$ is periodic with respect to $\xi$. 

 Thus $b_a\in \Sy^0(m^{-1},\epsilon_0)(\R^d \times \T^d)$ which allows to apply
  the restriction formula \eqref{eq:RestrForm} to
  \eqref{eq:InverseSymbolEquation} to get
  \eqref{eq:discrPseudoInvSymbol2}.
\end{proof}

\begin{prop}
  \label{prop:discrPseudoInvShort}
  Let the symbol $a \in \Sy^0(m,\epsilon_0)(\R^d \times \T^d)$ be
  $(m,\epsilon_0)$-elliptic. There is some $\epsilon_1 \in
  (0,\epsilon_0]$ such that for some neighbourhood $\mathcal{A}$ of
    $a$, the operator $\Op_{\epsilon, t}^\T \tilde{a}$ is invertible as
    a map $\s(\epsilon\Z^d) \rightarrow \s(\epsilon\Z^d)$ for
    $\tilde{a} \in \mathcal{A}$, $\epsilon \in (0,\epsilon_1]$, $t \in
      [0,1]$ with
      \begin{align}
        \label{eq:discrPseudoInvSymbol}
    \left( \Op_{\epsilon,t}^\T \tilde{a} \right)^{-1} =
    \Op_{\epsilon,t}^{\T} b_{\tilde{a}}
  \end{align}
  for some $b_{\tilde{a}} \in \Sy^0(m^{-1},\epsilon_1)(\R^d \times
  \T^d)$. Here, the neighbourhood $\mathcal{A}$ is considered with
  respect to the Fréchet topology induced by the seminorms $\left\|
  \cdot \right\|_\alpha$ defined in \eqref{eq:FrechetSNSym}.
\end{prop}

\begin{proof}
  For the non-discrete setting, it is shown in \cite[Chapter 8]{DiSj}
  that for $\epsilon \in (0,\epsilon_1]$ with $\epsilon_1$
    sufficiently small, the operator $\Op_{\epsilon,t} a$ is
    invertible with
    \begin{align}
      \label{eq:OpaInvOpbCont}
    \left( \Op_{\epsilon,t} a \right)^{-1} = \Op_{\epsilon,t} b_a
  \end{align}
  for some $b_a \in \Sy^0(m^{-1},\epsilon_1)(\R^{2d})$. This is implicitly the punchline of the 
  discussion in \cite[p.100]{DiSj}. The proof is
  based on the Neumann series construction for $1 + \epsilon
  \Op_{\epsilon,t} \rho$ for $\epsilon$ small (using the semiclassical Beals characterisation of pseudodifferential 
  operators to control the symbol of the inverse $\bigl(1 + \epsilon
  \Op_{\epsilon,t} \rho\bigr)^{-1}$), where the symbol $\rho
  \in \Sy^0(1,\epsilon_0)(\R^d \times \R^d)$ is characterised by
  \begin{align} \label{eq:OpaInvOprho}
    \Op_{\epsilon,t} a \circ \Op_{\epsilon,t} 1/a = 1 +
    \epsilon \Op_{\epsilon,t} \rho.
  \end{align}
  Here, $1/a \in \Sy^0(m^{-1},\epsilon_1)(\R^{2d})$ since $a$ is
  $(m,\epsilon_0)$-elliptic. 
  
  To show the invertibility of $\Op_{\epsilon,t} \tilde{a}$ for $\tilde{a}$ in some neighbourhood
  $\mathcal{A}$ of $a$, we first remark that 
  $\mathcal{A}$ can be chosen such that any $\tilde{a} \in
  \mathcal{A}$ is $(m,\epsilon_0)$-elliptic with the same constant:
  Since $a$ is $(m,\epsilon_0)$-elliptic, we have $|a| \geq Cm$ for
  some $C >0$. Assuming that $\left\| a - \tilde{a} \right\|_0 <
  C/2$, we get
  \begin{align}
    |\tilde{a}| \geq |a| - |a - \tilde{a}| \geq Cm - Cm/2 = Cm/2.
  \end{align}
Thus \eqref{eq:OpaInvOprho} holds for $a, \rho$ replaced by $\tilde{a}\in \mathcal{A}$ and some $\tilde{\rho}$.   
  Using uniform estimates for the remainder
  terms in the symbolic calculus and the Theorem of
  Calder{\'o}n-Vaillancourt in the non-discrete setting (see e.g. \cite{Mart}), one easily verifies that there is a constant $C>0$ such that
  $\|\Op_{\epsilon,t}\tilde{\rho}\|\leq C$ for all $\tilde{a}\in \mathcal{A}$. Using the Beals characterisation again, this 
  shows that \eqref{eq:OpaInvOpbCont} holds with $a, b_a$ 
  replaced by $\tilde{a}, b_{\tilde{a}}$ for any $\tilde{a} \in \mathcal{A}$, possibly after
  shrinking $\epsilon_1$. 
  
  To conclude the proof of the statement
  \eqref{eq:discrPseudoInvSymbol}, we apply Lemma
  \ref{Lem:InvSymbPeriodic} with $a$ replaced by $\tilde{a}$.
\end{proof}

\subsection{Functional calculus}
\label{sec:FuncCalc}

For a symbol $a$ satisfying the ellipticity condition in Hypothesis
\ref{hyp:WRef} (\ref{hyp:WRef_EllItem}), we construct the unique
self-adjoint realisation of $\Op_{\epsilon,1/2}^\T a$ in Proposition
\ref{prop:PSelfAdj}. We remark that ellipticity is actually only
needed for fixed $\epsilon$ and not in the uniform sense of
\eqref{eq:DefElliptic}.  We recall from the appendix (see
\eqref{eq:OpSDual}) that $\Op_{\epsilon,1/2}^\T a$ can be extended to
a continuous operator $\s'(\epsilon\Z^d) \rightarrow
\s'(\epsilon\Z^d)$.
  
\begin{prop}
  \label{prop:PSelfAdj}
  Assume the symbol $a$ to satisfy Hypothesis \ref{hyp:WRef}
  (\ref{hyp:WRef_EllItem}).  Then, for $\epsilon > 0$ sufficiently
  small, the operator
      \begin{align}
        \label{eq:PdiscrDef}
    \P_\epsilon : \l^2(\epsilon\Z^d) \supset \D_\epsilon
    \rightarrow \l^2(\epsilon\Z^d), \quad u \mapsto \left(
      \Op_{\epsilon,1/2}^{\T} a \right)u,
  \end{align}
  where $\D_\epsilon := \left( \Op_{\epsilon,1/2}^\T (a + i)
  \right)^{-1}\left( \l^2(\epsilon\Z^d) \right)$, is well-defined and
  self-adjoint.
\end{prop}

\begin{proof}
  Since $a + i$ is $(m,\epsilon_0)$-elliptic, $a-i$ is also
  $(m,\epsilon_0)$-elliptic. Due to Proposition
  \ref{prop:discrPseudoInvShort} we find $\epsilon_1 \in
  (0,\epsilon_0]$ such that there are $b^+, b^- \in
    \Sy^0(m^{-1},\epsilon_1)(\R^d \times \T^d)$ with
    \begin{align}
      \label{eq:SelfAdjInv}
    \left( \Op_{\epsilon,1/2}^\T (a \pm i) \right)^{-1} =
    \Op_{\epsilon,1/2}^\T b^{\pm}
  \end{align}
  for any $\epsilon \in (0,\epsilon_1]$. Since $m \geq 1$, we have $b^+ \in \Sy^{0}(1,\epsilon_1)(\R^d
    \times \T^d)$. So by \eqref{eq:SelfAdjInv} and Proposition
    \ref{prop:discrCaldVail}
    \begin{align}
      \D_\epsilon = \left( \Op_{\epsilon,1/2}^\T b^{+} \right) \left(
      \l^2(\epsilon\Z^d) \right) \subset \l^2(\epsilon\Z^d).
    \end{align}
    The operator $\Op_{\epsilon,1/2}^\T a$ maps $\D_\epsilon$ into
    $s'(\epsilon\Z^d)$. We check that actually
    \begin{align}
      \label{eq:OpDl2}
      \left(\Op_{\epsilon,1/2}^\T a\right) (\D_\epsilon) \subset
      \l^2(\epsilon\Z^d).
    \end{align}
    Let $u$ be an element of the lhs of \eqref{eq:OpDl2}, so 
    \begin{align}
      u = \left( \left(\Op_{\epsilon,1/2}^\T a\right) \circ \left(
      \Op_{\epsilon,1/2}^\T b^{+} \right) \right) v
    \end{align}
    for some $v \in \l^2(\epsilon\Z^d)$. Then by Proposition
    \ref{prop:SharpProd}, setting $\compw := \compw_{\frac{1}{2}}$,
    \begin{align}
      u = \left( \Op_{\epsilon,1/2}^\T a \compw b^+ \right) v,
    \end{align}
    where $a \compw b^+ \in \Sy^{0}(1,\epsilon_1)(\R^d \times
    \T^d)$. So $u \in \l^2(\epsilon\Z^d)$ by Proposition
    \ref{prop:discrCaldVail}. This proves that $\P_\epsilon$ is
    well-defined.

    We shall check that $\P_\epsilon$ is self-adjoint by applying the
    basic criterion of self-adjointness (\cite[Theorem VIII.3]{ReSi}):
    $\P_\epsilon$ is self-adjoint if $\P_\epsilon$ is symmetric and
    \begin{align}
      \label{eq:PSelfAdjCrit}
      \left(\P_\epsilon \pm i\right)(\D_\epsilon) = \l^2(\epsilon\Z^d).
    \end{align}

    Using $a = \overline{a}$, one shows by a straightforward
    computation similar to \eqref{eq:OpExtSDualPrep} that the Weyl
    quantisation $\Op_{\epsilon,1/2}^\T a$ is symmetric on
    $\s(\epsilon\Z^d)$, i.e.
    \begin{align}
      \label{eq:PSymmetry}
      \skp{\left(\Op_{\epsilon,1/2}^\T a\right) u}{v} =
      \skp{u}{\left(\Op_{\epsilon,1/2}^\T a\right)v} \qquad \mbox{for
      } u,v \in \s(\epsilon\Z^d).
    \end{align}
    We claim that $\s(\epsilon\Z^d)$ is dense in $\D_\epsilon$ with
    respect to the graph norm induced by $\Op_{\epsilon,1/2}^\T a$. As
    a consequence, the relation \eqref{eq:PSymmetry} extends to any
    $u, v \in \D_\epsilon$, which shows that $\P_\epsilon$ is
    symmetric.  We prove the claim. For this purpose let $u \in
    \D_\epsilon$. We shall construct a sequence of functions $u_j \in
    \s(\epsilon\Z^d)$ with
    \begin{align}
      \label{eq:GraphNormLimits}
      u_j \rightarrow u \quad \mbox{ and } \quad
      \left(\Op_{\epsilon,1/2}^\T a\right) u_j \rightarrow
      \left(\Op_{\epsilon,1/2}^\T a \right) u \quad \mbox{ in }
      \l^2(\epsilon\Z^d).
    \end{align}
    By the definition of $\D_\epsilon$ and $b^+$, there is $w \in \l^2(\epsilon\Z^d)$ with $u = \left( \Op_{\epsilon,1/2}^\T b^+
    \right) w$. Since $s(\epsilon\Z^d)$ is dense in
    $\l^2(\epsilon\Z^d)$, there is a sequence of functions $w_j \in
    \s(\epsilon\Z^d)$ with $w_j \rightarrow w$ in
    $\l^2(\epsilon\Z^d)$. Defining $u_j := \left(
    \Op_{\epsilon,1/2}^\T b^+ \right) w_j \in \s(\epsilon\Z^d)$, we
    have $u_j \rightarrow u$ in $\l(\epsilon\Z^d)$ by Proposition
    \ref{prop:discrCaldVail}. Furthermore, by Proposition
    \ref{prop:SharpProd}, we have $\left(\Op_{\epsilon,1/2}^\T
    a\right)u_j = \left( \Op_{\epsilon,1/2}^\T a \compw b^+ \right)
    w_j$ with $a \compw b^+ \in \Sy^0(1,\epsilon_1)(\R^d \times
    \T^d)$. Applying again Proposition \ref{prop:discrCaldVail}, this
    implies the second limit statement in \eqref{eq:GraphNormLimits}.

    It remains to check \eqref{eq:PSelfAdjCrit}. We claim that
    \begin{align}
      \label{eq:OpaPiOpaMi}
      \left(\Op_{\epsilon,1/2}^\T (a+i)\right)^{-1}
      \left(\l^2(\epsilon\Z^d)\right) = \left(\Op_{\epsilon,1/2}^\T
      (a-i)\right)^{-1} \left(\l^2(\epsilon\Z^d)\right).
    \end{align}
    As a consequence, \eqref{eq:PSelfAdjCrit} follows immediately from
    the definition of $\D_\epsilon$. Note that
    $\Op_{\epsilon,1/2}^\T(a-i)$ is invertible for $\epsilon \in
    (0,\epsilon_1]$ due to the choice of $\epsilon_1$ in
      \eqref{eq:SelfAdjInv}. In order to check \eqref{eq:OpaPiOpaMi},
      we write
      \begin{align}
        \label{eq:OpaPiOpaMiTransf}
        \left( \Op_{\epsilon,1/2}^\T (a+i) \right)^{-1} = \left(
        \Op_{\epsilon,1/2}^\T (a-i) \right)^{-1} \circ
        \mathbf{Q}_\epsilon \qquad \mbox{on } \s(\epsilon\Z^d)
      \end{align}
      with
      \begin{align}
        \mathbf{Q}_\epsilon := \left( \Op_{\epsilon,1/2}^\T (a-i)
        \right) \circ \left( \Op_{\epsilon,1/2}^\T (a+i) \right)^{-1}.
      \end{align}
      Applying Proposition \ref{prop:SharpProd} to $\mathbf{Q}_\epsilon$, we have
      \begin{align}
        \mathbf{Q}_\epsilon = \Op_{\epsilon,1/2}^\T \left( (a-i)
        \compw b^+ \right), \qquad \mbox{where } (a-i) \compw b^+ \in
        \Sy^0(1,\epsilon_1)(\R^d \times \T^d).
      \end{align}
      By Proposition \ref{prop:discrCaldVail}, $\mathbf{Q}_\epsilon$
      has a continuous extension onto $\l^2(\epsilon\Z^d)$. Since the
      operator $\mathbf{Q}_\epsilon$ is bijective on
      $\s(\epsilon\Z^d)$, its extension is bijective on
      $\l^2(\epsilon\Z^d)$. Combined with \eqref{eq:OpaPiOpaMiTransf},
      this gives \eqref{eq:OpaPiOpaMi}.
\end{proof}

We remark that the arguments in the proof of Proposition
\ref{prop:PSelfAdj} also show that the operator $\Op_{\epsilon,1/2}^\T
a$ on $\s(\epsilon\Z^d)$ is essentially self-adjoint.

Given $f \in \Cont_0^\infty (\R)$, we shall call a function $\aac{f}
\in \Cont_0^\infty (\C)$ an \textit{almost analytic extension} of $f$
if $\left. \aac{f} \right|_{\R} = f$ and if there are constants $C_N >
0$ such that
\begin{align}
  \label{eq:aacf}
  \abs{\dz \aac{f}}(z) \leq C_N |\Im z|^N \qquad (z \in \C)
\end{align}
for any $N \in \N$, where $\dz = (\partial_x + i\partial_y)/2$. The
function $\aac{f}$ can be constructed using an adaptation of the Borel
construction (see \cite{hoermander1968LecNotes}) or the Fourier
transform (see \cite{mather1971}). The almost analytic extension is
needed for the Helffer-Sjöstrand formula cited in Theorem
\ref{theo:FuncCalc}.

\begin{theo}
  \label{theo:FuncCalc}
  Let $\mathbf{A}$ be a self-adjoint operator on a Hilbert space
  $\mathcal{H}$. Let $f \in \Cont_0^\infty(\R)$ and let $\aac{f} \in
  \Cont_0^\infty(\C)$ be an almost analytic extension of $f$. Then
  \begin{align}
    \label{eq:FuncCalc}
    f(\mathbf{A}) = - \frac{1}{\pi} \int_\C \dz \aac{f} (z)
    (z-\mathbf{A})^{-1} L(dz)
  \end{align}
  ($L(dz)$ denoting the Lebesgue measure on $\C$).
\end{theo}

\begin{rem}
  \label{rem:IntegralDef}
  Using the estimates $\norm{(z - \mathbf{A})^{-1}} \leq |\Im z|^{-1}$
  and \eqref{eq:aacf}, the integrand in \eqref{eq:FuncCalc} may be
  considered as a compactly supported continuous function on $\C$ with
  values in the Banach space of bounded operators on
  $\mathcal{H}$. The integral in \eqref{eq:FuncCalc} then exists as a
  version of a Banach space valued Riemann integral. For the special
  case of integrands on $\R$, this is treated in \cite[p. 11]{ReSi}
  and \cite{dieu2008}. The integral in \eqref{eq:FuncCalc} also exists
  as a weak Lebesgue integral and as a Bochner integral (see
  e.g. \cite{Yos78, AmEsch09}).

  Taking the first or the last point of view, we remark that for a
  compactly supported continuous integrand on $\C$ with values in any
  Banach space $\mathcal{B}$, the integral exists in
  $\mathcal{B}$. While here we take $\mathcal{B}$ as the space of
  bounded operators, we shall in the next section also apply this
  statement with $\mathcal{B}$ being the space of trace class
  operators.
\end{rem}

In Theorem \ref{theo:FuncCalcNonDiscr}, we collect in the form of a
condensed theorem some useful results from \cite[Chapter 8]{DiSj},
which are statements on the functional calculus for
$\tilde{\P}_\epsilon$ based on the Helffer-Sjöstrand formula. We
remark that Hypothesis \ref{hyp:WRef} (\ref{hyp:WRef_EllItem}),
assumed in Theorem \ref{theo:FuncCalcNonDiscr}, is a special case of
the Hypothesis in \cite[Chapter 8]{DiSj} for the validity of the
functional calculus since we additionally assume periodicity for the
symbol $a$.

\begin{theo}
  \label{theo:FuncCalcNonDiscr}
  Assume the symbol $a$ to satisfy Hypothesis \ref{hyp:WRef}
  (\ref{hyp:WRef_EllItem}). For some $\epsilon_1$ sufficiently small and $f \in
  \Cont_0^\infty(\R)$, we have for any $\epsilon \in (0,\epsilon_1]$:
  \begin{enumerate}
  \item One can define the self-adjoint operator
    \begin{align}
      \label{eq:PTildeDef}
      \tilde{\P}_\epsilon : \L^2(\R^d) \supset \tilde{\D}_\epsilon
      \rightarrow \L^2(\R^d), \quad u \mapsto \left(
      \Op_{\epsilon,1/2} a \right)u,
    \end{align}
    where $\tilde{\D}_\epsilon :=
      \left( \Op_{\epsilon,1/2} (a + i) \right)^{-1}\left( \L^2(\R^d)
      \right)$.
  \item \label{item:FuncCalcNonDiscr1} For each $z \in \C$ with $\Im z
    \neq 0$ there is a unique symbol $b_{z-a} \in
    \Sy^0(m^{-1},\epsilon_1)(\R^{2d})$ with
    \begin{align}
      \label{eq:bzCharact}
      \left( z - \tilde{\P}_\epsilon \right)^{-1} = \Op_{\epsilon,1/2}
      b_{z-a}.
    \end{align}
    For some $C > 0$ there are constants $C_{\alpha,\beta} > 0$
    ($\alpha,\beta \in \N^d$) such that for all $z \in \C$ with $|z|
    \leq C$ and $\Im z \neq 0$
    \begin{align}
      \label{eq:bzEstim}
      |\partial_x^\alpha \partial_\xi^\beta b_{z-a}(x,\xi;\epsilon)| \leq
      C_{\alpha,\beta} \max \left( 1, \frac{\epsilon^{1/2}}{|\Im z|}
      \right)^{2d+1} |\Im z|^{-(|\alpha| + |\beta|)-1}.
    \end{align}
  \item \label{item:FuncCalcNonDiscr2} $f(\tilde{\P}_\epsilon) = \Op_{\epsilon,1/2} c$ where $c$
    given by
    \begin{align}
      \label{eq:FuncCalcSymb}
      c(x,\xi;\epsilon) = -\frac{1}{\pi} \int_\C \dz \aac{f} (z) b_{z-a}
      (x,\xi; \epsilon) L(dz) \qquad (x,\xi \in \R^d)
    \end{align}
    is an element of $\Sy^0(m^{-k},\epsilon_1)(\R^{2d})$ for
    any $k \in \N_0$ and $\aac{f}$ is an almost analytic extension of
    $f$.
  \item \label{item:FuncCalcNonDiscr3} If $a(x,\xi;\epsilon) \sim \sum_{j=0}^\infty \epsilon^j
    a_j(x,\xi)$, then the symbol $c$ has an asymptotic expansion
    $c(x,\xi;\epsilon) \sim \sum_{j=0}^\infty \epsilon^j c_j(x,\xi)$
    where the symbols $c_j \in \Sy^0(m^{-1},\epsilon_1)(\R^{2d})$ can
    be chosen as
    \begin{align}
      \label{eq:FuncCalcAsympcj}
      c_j(x,\xi) = \frac{1}{(2j)!} \left. \partial_t^{2j}
      (q_j(x,\xi,t)f(t)) \right|_{t = a_0(x,\xi)} \qquad (x,\xi \in
      \R^d)
    \end{align}
    where $q_j$ are polynomials of the form
    \begin{align}
      \label{eq:FuncCalcAsympcj2}
      q_j(x,\xi,z) = \sum_{k=0}^{2j} q_{j,k}(x,\xi) z^k \qquad (x,\xi
      \in \R^d)
    \end{align}
    where $q_{j,k} \in \Cont^\infty(\R^d \times \R^d)$. In particular,
    $c_0 = f \circ a_0$ and  $c_1 = (f' \circ a_0) a_1$. The Fréchet
    seminorms of the remainder terms associated with the asymptotic
    expansion of $c$ only depend linearly on finitely many derivatives
    of $f$ and on $a$.
  \end{enumerate}
\end{theo}

In the following, for any $\epsilon > 0$ sufficiently small, we let
the operators $\P_\epsilon$, $\tilde{\P}_\epsilon$ be as in Theorem
\ref{theo:FuncCalcNonDiscr} and Proposition \ref{prop:PSelfAdj}.

From Theorem \ref{theo:FuncCalcNonDiscr}, we derive a functional
calculus for $\P_\epsilon$.

\begin{cor}\label{cor:cperiodic}
  Assume the symbol $a$ to satisfy Hypothesis \ref{hyp:WRef}
  (\ref{hyp:WRef_EllItem}). Let further $f \in
  \Cont_0^\infty(\R)$. Then the symbol $c$ in \eqref{eq:FuncCalcSymb}
  and the functions $c_j, q_j, q_{j,k}$ in \eqref{eq:FuncCalcAsympcj}
  and \eqref{eq:FuncCalcAsympcj2} are periodic with respect to $\xi$
  and for any $\epsilon > 0$ sufficiently small, we have
    \begin{align}
      \label{eq:FuncCalcDiscr}
      f(\P_\epsilon) = \Op_{\epsilon,1/2}^\T c.
    \end{align}
\end{cor}

\begin{proof}
  Since $a$ is periodic with respect to $\xi$, we conclude from Lemma
  \ref{Lem:InvSymbPeriodic} that the symbol $b_{z-a}$ characterised by
  \eqref{eq:bzCharact} is periodic with respect to $\xi$ and that
  \begin{align}
    \label{eq:zPInvbzDiscr}
    \left( z - \P_\epsilon \right)^{-1} = \Op_{\epsilon,1/2}^\T 
    b_{z-a}
    \qquad (\Im z \neq 0).
  \end{align}
  Since the functions $c, c_j, q_j, q_{j,k}$, defined in Theorem
  \ref{theo:FuncCalcNonDiscr}, are induced by $b_{z-a}$, they are all
  periodic with respect to $\xi$. Using the restriction mapping
  $r_\epsilon$ and the restriction formula from Proposition
  \ref{prop:RestrForm}, we derive from \eqref{eq:zPInvbzDiscr} that
  \begin{align}
    \label{eq:rzP}
    r_\epsilon \circ \left( z - \tilde{\P}_\epsilon \right)^{-1} =
    \left( z - \P_\epsilon \right)^{-1} \circ r_\epsilon \qquad
    \mbox{on } \S(\R^d).
  \end{align}
  Combining the identity \eqref{eq:rzP} with Theorem
  \ref{theo:FuncCalc} and using Theorem \ref{theo:FuncCalcNonDiscr}
  (\ref{item:FuncCalcNonDiscr2}) and the restriction formula, we
  obtain
  \begin{align}
    f(\P_\epsilon) \circ r_\epsilon = r_\epsilon \circ
    f(\tilde{\P}_\epsilon) = r_\epsilon \circ \Op_{\epsilon,1/2} c =
    \Op_{\epsilon,1/2}^\T c \circ r_\epsilon \qquad \mbox{on }
    \S(\R^d).
  \end{align}
  This proves \eqref{eq:FuncCalcDiscr}.
\end{proof}

If $a \in \Sy^0(m,\epsilon_0)(\R^d \times \T^d)$ is real-valued with
$a \ess S$ for some $S \in \R$, then by \eqref{eq:essLarger} we can
always modify $a$ to a real-valued symbol $\adjust{a} \in
\Sy^0(m,\epsilon_0)(\R^d \times \T^d)$ with
\begin{align}
  \label{eq:adjustDef}
      \inf_{\substack{x \in \R^d,\, \xi \in \T^d \\ \epsilon \in
          (0,\epsilon_0]}} \adjust{a}(x,\xi;\epsilon) > S
\end{align}
by changing it on some ball $|x| < const$ not depending on $\epsilon$
for $\epsilon \in (0,\epsilon_0]$. We call $\adjust{a}$ an
  $S$-adjustment of $a$.

\begin{Lem}
  \label{lem:FuncCalcAdjust} Assume the symbol $a$ to satisfy Hypothesis \ref{hyp:WRef} (1) and (2). Let
  $\adjust{a}$ be a $\left(\sup J \right)$-adjustment of $a$. Then for
  any $\epsilon > 0$ sufficiently small the operator
  \begin{align}
    \adjOP{\P}_\epsilon : \l^2(\epsilon\Z^d) \supset \D_\epsilon
    \rightarrow \l^2(\epsilon\Z^d), \quad u \mapsto \left(
      \Op_{\epsilon,1/2}^\T \adjust{a} \right)u,
  \end{align}
  is well-defined and self-adjoint, where $\D_\epsilon := \left( \Op_{\epsilon,1/2}^\T (a + i)
  \right)^{-1}\left( \l^2(\epsilon\Z^d) \right)$ 
  coincides with the domain of $\P_\epsilon $ given in Proposition \ref{prop:PSelfAdj}. Moreover, for $f \in
  \Cont_0^\infty(J)$ we have
  \begin{align}
    \label{eq:FuncCalcAdjust}
    f(\P_\epsilon) = -\frac{1}{\pi} \int_\C \dz \aac{f}(z) (z -
    \P_\epsilon)^{-1} (\P_\epsilon - \adjOP{\P}_\epsilon) (z -
    \adjOP{\P}_\epsilon)^{-1} L(dz)
  \end{align}
  where $\aac{f}$ is an almost analytic extension of $f$.
\end{Lem}

\begin{proof}
  Since $\adjust{a}$ differs from $a$ only on some ball $|x| < const$,
  we may apply Proposition \ref{prop:PSelfAdj} with $a$ replaced by
  $\adjust{a}$ to get that for any $\epsilon > 0$ sufficiently small
  the operator
  \begin{align}
    \adjOP{\P}_\epsilon : \l^2(\epsilon\Z^d) \supset
    \adjOP{\D}_\epsilon \rightarrow \l^2(\epsilon\Z^d), \quad u
    \mapsto \left( \Op_{\epsilon,1/2}^\T \adjust{a} \right)u,
  \end{align}
  where $\adjOP{\D}_\epsilon := \left( \Op_{\epsilon,1/2}^\T
  (\adjust{a} + i) \right)^{-1} \left( \l^2(\epsilon\Z^d) \right)$, is
  well-defined and self-adjoint. Since the difference
  \begin{align}
    \Op^\T_{\epsilon,1/2} (a+i) - \Op^\T_{\epsilon,1/2} (\adjust{a}+i)
    =\Op^\T_{\epsilon,1/2} (a - \adjust{a})
  \end{align}
  is a bounded operator $\l^2(\epsilon\Z^d) \rightarrow
  \l^2(\epsilon\Z^d)$ (see Proposition \ref{prop:discrCaldVail}), it
  actually follows that $\D_\epsilon = \adjOP{\D}_\epsilon$.

  We now prove that, for any $\epsilon > 0$ sufficiently small, the
  operator $\adjOP{\P}_\epsilon$ has no spectrum in $\overline{J}$.

      We first check that $\lambda - \adjust{a}$ is
      $(m,\epsilon_0)$-elliptic for $\lambda \in \overline{J}$. Since
      $\adjust{a} + i$ is $(m,\epsilon_0)$-elliptic, there is some $C
      > 0$ such that $|\adjust{a} + i| \geq Cm$. Since $J$ is bounded,
      there is some $K > 0$ with $|\lambda - i| \leq K$ for $\lambda
      \in \overline{J}$. So, on the set of points $(x,\xi)$ where $Cm(x,\xi)
      \geq 2K$, we have
      \begin{align}
        \label{eq:lambdaEllip1}
        |\lambda - \adjust{a}| \geq |\adjust{a} + i| - |\lambda - i|
        \geq Cm - K \geq Cm/2.
      \end{align}
      Since $\adjust{a}$ is a $(\sup J)$-adjustment of $a$, there is
      some $\delta > 0$ with $|\lambda - \adjust{a}| \geq \delta$ for
      $\lambda \in \overline{J}$. So, for $(x,\xi)$ with $Cm(x,\xi) <
      2K$, we have
      \begin{align}
        \label{eq:lambdaEllip2}
        |\lambda - \adjust{a}| > \frac{\delta Cm}{2K}.
      \end{align}
      Ellipticity of $\lambda - \adjust{a}$ now follows from
      \eqref{eq:lambdaEllip1} and \eqref{eq:lambdaEllip2}.

      For each $\lambda \in \overline{J}$, we may now apply
      Proposition \ref{prop:discrPseudoInvShort} to get some
      $\epsilon(\lambda) > 0$ and some neighbourhood
      $\mathcal{U}_\lambda$ of $\lambda$ such that
      $\Op^\T_{\epsilon,1/2} ( \tilde{\lambda} - \adjust{a} )$ is
      invertible for any $\tilde{\lambda} \in \mathcal{U}_\lambda$ and
      $\epsilon \in (0,\epsilon(\lambda)]$. Furthermore,
        \begin{align}
          \label{eq:adjPlambda}
        \left( \Op^\T_{\epsilon,1/2} ( \tilde{\lambda} - \adjust{a} )
        \right)^{-1} = \Op_{\epsilon,1/2}^\T b_{\tilde{\lambda} -
          \adjust{a}}
      \end{align}
      with the symbol $b_{\tilde{\lambda}- \adjust{a}} \in
      \Sy^0(m^{-1},\epsilon(\lambda))(\R^d \times \T^d)$ characterised
      by \eqref{eq:discrPseudoInvSymbol}. Since $\overline{J}$ is
      compact, we find some $\epsilon_1 \in (0,\epsilon_0]$, such that
        the operators $\Op_{\epsilon,1/2}^\T (\lambda - \adjust{a})$
        are invertible for any $\lambda \in \overline{J}$ and any
        $\epsilon \in (0,\epsilon_1]$ with inverse operators given by
          \eqref{eq:adjPlambda}. Since the operator
          $\Op_{\epsilon,1/2}^\T b_{\lambda-\adjust{a}} :
          \l^2(\epsilon\Z^d) \rightarrow \D_\epsilon$ is the resolvent
          of $\adjOP{\P}_\epsilon$ at $\lambda \in \overline{J}$,
          $\lambda$ does not belong to the spectrum of
          $\adjOP{\P}_\epsilon$ for $\epsilon \in (0,\epsilon_1]$.

  Finally, we prove the representation formula
  \eqref{eq:FuncCalcAdjust}. For $z \in \C$ with $|\Im z| > 0$ and
  $\epsilon > 0$ sufficiently small, we have the resolvent equation
  \begin{align}
    \left(z - \P_\epsilon\right)^{-1} = \left(z -
      \adjOP{\P}_\epsilon\right)^{-1} + \left(z-
      \P_\epsilon\right)^{-1} \left(\P_\epsilon -
      \adjOP{\P}_\epsilon\right) \left(z-
      \adjOP{\P}_\epsilon\right)^{-1},
  \end{align}
  which combined with Theorem \ref{theo:FuncCalc} yields
  \begin{align}
    f(\P_\epsilon) = f(\adjOP{\P}_\epsilon) - \frac{1}{\pi} \int_\C
    \dz \aac{f}(z) (z-\P_\epsilon)^{-1} (\P_\epsilon -
    \adjOP{\P}_\epsilon) (z-\adjOP{\P}_\epsilon)^{-1} L(dz)
  \end{align}
  for $f \in \Cont_0^\infty(\R)$ and an almost analytic extension
  $\aac{f}$ of $f$. If $\supp f \subset J$, then
  $f(\adjOP{\P}_\epsilon) = 0$ since $\adjOP{\P}_\epsilon$ has no
  spectrum in $\overline{J}$ for any $\epsilon > 0$ sufficiently
  small.
\end{proof}

\section{Trace estimates}
\label{sec:TraceEst}

In the first subsection of this section we state and prove a general trace class criterion for integral operators which gives a 
convenient criterion for difference operators of the type considered in this paper to be trace class. 
In the second subsection we then perform a localisation in energy via functional calculus of difference operators satisfying our Hypothesis 
\ref{hyp:WRef} and obtain trace asymptotics for appropriate functions of these operators. 
This is an important first step to extract the leading order term - the Weyl term - of eigenvalue asymptotics. 
In the final section we apply these trace asymptotics to obtain some rough Weyl asymptotics 
before developing the more advanced theory in Section 4 using a good semiclassical time parametrix.

\subsection{General trace class criteria}

We recall that for a compact operator $A$ on a separable Hilbert space
$\mathcal{H}$, the singular values $s_i(A)$ of $A$ are defined to be
the eigenvalues of the positive operator $(AA^*)^{1/2}$, where $A^*$
denotes the adjoint operator of $A$. $A$ is called to be of trace
class if the trace norm
\begin{align}
  \trn{A} := \sum_{i} s_i(A)
\end{align}
is finite. The space of trace class operators is complete with respect
to the trace norm and forms a two-sided ideal in the space of bounded
operators on $\mathcal{H}$. In particular,
\begin{align}
  \label{eq:TraceIdealEstim}
  \trn{A_1 A_2} &\leq \trn{A_1} \norm{A_2}, & \trn{A_2 A_1} &\leq \trn{A_1} \norm{A_2} 
\end{align}
if $A_1$ is of trace class and $A_2$ is bounded. If $A$ is of trace
class, the trace of $A$, defined by
\begin{align}\label{traceA}
  \tr{A} := \sum_{i} \skp{e_i}{Ae_i} \qquad \mbox{ for any orthonormal
    basis $(e_i)$ of } \mathcal{H},
\end{align}
is absolutely convergent and does not depend on the choice of the
orthonormal basis $(e_i)$. For more information on trace class
operators, see e.g. \cite{GohK69,SimonTI,ReSi,GeVi64} and the short summary
in \cite[chapter 9]{DiSj}.

From \cite{Stinespring1958} we recall that, given an orthonormal basis
$(e_i)$ of $\mathcal{H}$, a bounded operator $A$ is of trace class if
\begin{align}\label{traceconditionA}
  \sum_{i} \norm{Ae_i} < \infty.
\end{align}
In this case
\begin{align}\label{traceestimateA}
  \trn{A} \leq \sum_{i} \norm{Ae_i}.
\end{align}

In our context, we shall consider operators on the separable Hilbert
space $\mathcal{H} = \l^2(\epsilon\Z^d)$, equipped with the norm
$\norm{\cdot}$ induced by the inner product \eqref{eq:skpl2zd}. Here,
an orthonormal basis is given by the canonical basis $(e_x)_{x \in
  \epsilon\Z^d}$ with $e_x(y) = \delta_{xy}$ being the Kronecker delta
for $x, y \in \epsilon\Z^d$.

Define formally 
\begin{align}
    \label{eq:IntOpAeps}
    (\mathbf{A}_\epsilon u)(x) = \sum_{y \in
      \epsilon\Z^d} k_\epsilon(x,y) u(y),\qquad u\in \l^2(\epsilon\Z^d),
\end{align}
with kernel $k_\epsilon \in \l^2(\epsilon\Z^d \times
\epsilon\Z^d)$ (note that  $\l^2(\epsilon\Z^d \times
\epsilon\Z^d)$  is isomorphic to $\l^2(\epsilon\Z^d) \otimes
\l^2(\epsilon\Z^d)$). Then $\mathbf{A}_\epsilon$ is a Hilbert-Schmidt operator 
and in particular it is a bounded operator on $\l^2 (\epsilon \Z^d)$.

Applying \eqref{traceA}, \eqref{traceconditionA} and \eqref{traceestimateA} to $\mathbf{A}_\epsilon$
 we get
\begin{prop}
  \label{prop:TraClCri}
  The integral operator $\mathbf{A}_\epsilon$ defined in \eqref{eq:IntOpAeps}
  is of trace class if
  \begin{align}
    \label{eq:TraClCond}
    \sum_{y \in \epsilon\Z^d} \norm{k_\epsilon(\cdot, y)} < \infty.
  \end{align}
  In this case
  \begin{align}
    \label{eq:TraClEstimateGeneral}
    \trn{\mathbf{A}_\epsilon} \leq \sum_{y \in \epsilon\Z^d}
    \norm{k_\epsilon(\cdot,y)} \qquad \mbox{and} \qquad \tr
    \mathbf{A}_\epsilon = \sum_{x \in \epsilon\Z^d} k_\epsilon(x,x).
  \end{align}
\end{prop}
\noindent Using $\norm{k_\epsilon(\cdot,y)} \leq
\norm{k_\epsilon(\cdot,y)}_{\l^1(\epsilon\Z^d)}$, we see that, in
particular, the condition \eqref{eq:TraClCond} is fulfilled if
\begin{align}
  \label{eq:TraClCondSpec}
  |k_\epsilon (x,y)| \leq C \jap{(x,y)}^{-2d-\delta} \qquad \mbox{ for
    some $C, \delta > 0$ and any $x, y \in \epsilon\Z^d$}.
\end{align}

We use Proposition \ref{prop:TraClCri} to derive a trace class
criterion, an estimate for the trace norm and a trace formula for pseudodifferential operators of discrete type. Here,
due to the discrete version of the Theorem of Calder{\'o}n-Vaillancourt
(Proposition \ref{prop:discrCaldVail}), for a bounded symbol $a$ we
consider $\Op^\T_{\epsilon,t} a$, defined in \eqref{eq:OpTa_u_Def}, as
a bounded operator on $\l^2(\epsilon\Z^d)$.

\begin{prop}
  \label{prop:TraClCriPDO}
  Let $a \in \Sy^0(m,\epsilon_0)(\R^d \times \T^d)$. The operator
  $\Op_{\epsilon,t}^\T a$ is of trace class for $\epsilon \in
  (0,\epsilon_0]$, $t \in [0,1]$ if
    \begin{align}
      \label{eq:TraClCondPDO}
      m(x,\xi) = \jap{x}^{-d-\delta} \qquad \mbox{ for some } \delta > 0. 
    \end{align}
  In this case
  \begin{align}
      \label{eq:TraClEstPDO}
      \trn{\Op_{\epsilon,t}^\T a} \leq \frac{1}{(2\pi)^{d/2}}\sum_{x
        \in \epsilon\Z^d} \torn{a_0(x,\cdot;\epsilon)} < \infty,    
  \end{align}
  where $a_0$ is the symbol of $\Op_{\epsilon,t}^\T a$ in
  $(t=0)$-quantisation (see Proposition \ref{prop:SwitchQuant}), and
  \begin{align}
    \label{eq:TracePDO}
    \tr \left( \Op_{\epsilon,t}^\T a \right) = \frac{1}{(2\pi)^d}
    \sum_{x \in \epsilon\Z^d} \int_{\T^d} a(x,\xi;\epsilon)
    d\xi.
  \end{align}
\end{prop}

\begin{proof}
  Let $t \in [0,1]$. Due to Proposition \ref{prop:SwitchQuant}, for
  any $s \in [0,1]$ there is a symbol $a_s \in
  \Sy^0(m,\epsilon_0)(\R^d \times \T^d)$ such that
  $\Op_{\epsilon,t}^\T a = \Op_{\epsilon,s}^\T a_s$. So, using the
  pointwise definition \eqref{eq:OpTa_u_Def}, we may write
  \begin{align}
    \label{eq:ChangeQuantKernel}
    \left( \Op_{\epsilon,t}^\T a \right) u(x) = \sum_{y \in \epsilon\Z^d}
    k_{\epsilon,s}(x,y) u(y) \qquad (u \in \l^2(\epsilon\Z^d),\, x \in
    \epsilon\Z^d,\, s \in [0,1])
  \end{align}
  with kernel
  \begin{align}
    k_{\epsilon,s}(x,y) &= \frac{1}{(2\pi)^d} \int_{\T^d}
    e^{i(y-x)\xi/\epsilon} a_s(sx + (1-s)y,\xi;\epsilon) d\xi.
  \end{align}
 
  For verifying the trace class condition \eqref{eq:TraClCond}, we
  choose $s = 0$, which seems most easy.
  
  Clearly, $k_{\epsilon,0} \in
  \l^2(\epsilon\Z^d) \otimes \l^2(\epsilon\Z^d)$: Square summability
  of $k_{\epsilon,0}$ for fixed $\epsilon > 0$ on $x \neq y$ follows
  by integration by parts. For the diagonal $x = y$, we use the decay
  of $a_0$ according to condition \eqref{eq:TraClCondPDO}.

  Using the discrete Fourier transform $\Fd_\epsilon : \L^2(\T^d)
  \rightarrow \l^2(\epsilon\Z^d)$, defined 
  in \eqref{def:disFouTr}
  we have
  \begin{align}
    k_{\epsilon,0}(x,y) &=
    \frac{1}{(2\pi)^{d/2}} \Fd_{\epsilon}
    (a_0(y,\cdot;\epsilon))(x-y).
  \end{align}
  $\Fd_{\epsilon}$ is isometric, so
  \begin{align}
    \label{eq:FdIsom}
    \norm{k_{\epsilon,0}(\cdot,y)} = \frac{1}{(2\pi)^{d/2}}
    \torn{a_0(y,\cdot;\epsilon)} \qquad (y \in \l^2(\epsilon\Z^d)).
  \end{align}
  Combining \eqref{eq:FdIsom} with \eqref{eq:TraClCondPDO}, we see
  that condition \eqref{eq:TraClCond} is fulfilled. Applying
  Proposition \ref{prop:TraClCri} proves that $\Op_{\epsilon,t}^\T a$
  is of trace class with the trace norm estimate in
  \eqref{eq:TraClEstPDO}. The trace formula in \eqref{eq:TracePDO}
  follows from the trace formula in \eqref{eq:TraClEstimateGeneral}
  using the representation \eqref{eq:ChangeQuantKernel} with $s = t$.
\end{proof}

\subsection{Trace asymptotics}

 In this subsection we apply the trace class criterion and the asymptotic expansion of the trace in Proposition \ref{prop:TraClCriPDO}  
 to the operator  $f(\P_\epsilon)$ considered in this paper. The crucial result is Theorem \ref{theo:TraceAsympExp} below. 
 For this we need some  preparation.

\begin{prop}
  \label{prop:fPTraceClass}
  Assume Hypothesis \ref{hyp:WRef} (1) and (2) and let $\P_\epsilon$
  be as in Proposition \ref{prop:PSelfAdj}. Then, for any $\epsilon >
  0$ sufficiently small, the operator $f(\P_\epsilon)$ is of trace
  class for any $f \in \Cont_0^\infty(J)$.
\end{prop}

\begin{proof}
  Choose $\adjust{a}$ and $\adjOP{\P}_\epsilon$ as in Lemma
  \ref{lem:FuncCalcAdjust}.  Then, for any $\epsilon > 0$ sufficiently
  small
  \begin{align}
    \label{eq:TraceClassFuncCalc}
    f(\P_\epsilon) = -\frac{1}{\pi} \int_\C \dz \aac{f}(z)
    (z-\P_\epsilon)^{-1} \left( \Op_{\epsilon,1/2}^\T (a-\adjust{a})
    \right) (z- \adjOP{\P}_\epsilon)^{-1} L(dz)
  \end{align}
  where $\aac{f}$ is an almost analytic extension of $f$. Since $a -
  \adjust{a}$ has compact support with respect to $x$, we know by
  Proposition \ref{prop:TraClCriPDO} that $\Op_{\epsilon,1/2}^\T (a -
  \adjust{a})$ is of trace class. A priori, the integrand in
  \eqref{eq:TraceClassFuncCalc} is only defined for $\Im z \neq
  0$. But, using the general trace estimate
  \eqref{eq:TraceIdealEstim}, the resolvent estimates
  \begin{align}
    \label{eq:resolvEstimP}
    \norm{((z-\P_\epsilon)^{-1})} &\leq |\Im z|^{-1}, &
    \norm{(z-\adjust{\P}_\epsilon)^{-1}} &\leq |\Im z|^{-1}
  \end{align}
  and the estimate \eqref{eq:aacf} for $\aac{f}$, one verifies that
  the integrand in \eqref{eq:TraceClassFuncCalc} can be extended to a
  continuous compactly supported function on $\C$ with values in the
  space of trace class operators. Since the space of trace class
  operators is complete with respect to the trace norm, the integral
  \eqref{eq:TraceClassFuncCalc} is also of trace class (see Remark
  \ref{rem:IntegralDef}).
\end{proof}

Next we also need some more technical preparations. In principle,
pseudo-differential operators are nonlocal; in particular, the composition of operators with symbols of disjoint support is non-zero. 
But in the semiclassical
limit such non-locality corrections are small. This is well known from the standard semiclassical symbolic calculus. Similarly, if only 
one of the symbols involved in a  composition of operators is compactly supported, the composition of operators has fast decay 
with respect to any polynomial weight on phase space.   We show in Lemma
\ref{lem:DecrOutsComp} below that in the discrete setting also these corrections
are of order $\O(\epsilon^\infty)$ and the decay mentioned above holds.

We recall that we use the notation $\compw := \compw_{\frac{1}{2}}$ as introduced in Proposition \ref{prop:SharpProd} in the appendix.

\begin{Lem}
  \label{lem:DecrOutsComp}
  Let $n \in \N$. For $j \in \left\{1, \dots, n\right\}$ and
  $a_j \in \Sy^0(m_j,\epsilon_0)(\R^d \times \T^d)$. Suppose
  \begin{equation}\label{ass1:DecrOutsComp} 
  \bigcap_{j=1}^n \bigcup_{\epsilon \in (0,\epsilon_0]} \supp
    a_j(\cdot,\cdot;\epsilon) = \emptyset 
    \end{equation} 
    and $\bigcup_{\epsilon \in
      (0,\epsilon_0]}\supp a_i(\cdot, \cdot; \epsilon) \subset K
      \times \T^d$ for some $i \in \left\{ 1, \dots, n \right\}$ and
      some compact $K \subset \R^d$. 
      
      Then for any $N \in \N$
      \begin{align}
        \label{eq:TwistProdSmall}
    a_1 \compw \cdots \compw a_n \in
    \Sy^N(m^{-N},\epsilon_0)(\R^d \times \T^d)
  \end{align}
  for $m(x,\xi) := \jap{x}$ and 
  \begin{align}
    \label{eq:ProdOperTraceSmall}
    \trn{ \left( \Op_{\epsilon,1/2}^\T a_1 \right) \circ \cdots \circ
      \left( \Op_{\epsilon,1/2}^\T a_n \right) } = \O(\epsilon^N)
    \qquad (\epsilon \downarrow 0).
  \end{align}
\end{Lem}

\begin{proof}
  Let $N \in \N$. We check \eqref{eq:TwistProdSmall} first. Since
  $a_i$ has compact support in $x$ (uniformly in $\xi, \epsilon$), we
  have $a_i \in \Sy^0\left( m^{-N'}, \epsilon_0 \right)(\R^d \times
  \T^d)$ for any $N' \in \N$. Then, by Proposition
  \ref{prop:SharpProd}, $q := a_1 \compw \cdots \compw a_n \in
  \Sy^0(m^{-N},\epsilon_0)(\R^d \times \T^d)$. Using the asymptotic
  expansion \eqref{eq:SharpProdAsExp} for $t=\frac{1}{2}$, we find $q \sim \sum_{k =
    0}^\infty \epsilon^k q_k(x,\xi;\epsilon)$ with symbols $q_k$ of
  the form
  \begin{align}
    \label{eq:nonlocalAsympExp}
    q_k (x,\xi;\epsilon) = \sum_{\substack{\alpha_1, \dots, \alpha_n
        \in \N^{2d} \\ \sum_j |\alpha_j| =
        2k}} C_{k\alpha_1\dots\alpha_n} (\partial^{\alpha_1}a_1
    \cdots \partial^{\alpha_n}a_n) (x,\xi;\epsilon).
  \end{align}
  Indeed, for each $k \in \N$ the symbol $q_k$  vanishes under the assumption \eqref{ass1:DecrOutsComp}. So, using the
  remainder estimates of Proposition \ref{prop:SharpProd}, we get $q
  \in \Sy^{N+d}(m^{-N},\epsilon_0)(\R^d \times \T^d)$.

  The relation \eqref{eq:ProdOperTraceSmall} follows as a consequence
  of \eqref{eq:TwistProdSmall}: Firstly, due to Proposition
  \ref{prop:SharpProd}
  \begin{align}
    \label{eq:nonlocalChangeQuant2}
    \trn{ \left( \Op_{\epsilon,1/2}^\T a_1 \right) \circ \cdots \circ
      \left( \Op_{\epsilon,1/2}^\T a_n \right) } = \trn{
      \Op_{\epsilon,1/2}^\T q }.
  \end{align}
  Secondly, due to Proposition \ref{prop:SwitchQuant} there is a
  unique symbol $q_0 \in \Sy^{N+d}(m^{-N},\epsilon_0)(\R^d \times
  \T^d)$ with $\Op_{\epsilon,1/2}^\T q = \Op_{\epsilon,0}^\T
  q_0$. Since for our choice $m(x,\xi) = \jap{x}$ the order function
  $m^{-N}$ satisfies \eqref{eq:TraClCondPDO} for $N$ sufficiently
  large, we have by Proposition \ref{prop:TraClCriPDO} that the rhs of
  \eqref{eq:nonlocalChangeQuant2} can be bounded by
  \begin{align}
    \label{eq:nonlocalChangeQuant}
    \trn{ \Op_{\epsilon,1/2}^\T q } \leq \frac{1}{(2\pi)^{d/2}}\sum_{x
      \in \epsilon\Z^d} \torn{q_0(x,\cdot;\epsilon)} =
    \O(\epsilon^{N}).
  \end{align}
\end{proof}

The next Lemma provides a similar estimate on non-local corrections
for the operator $f(\P_\epsilon)$ where $f \in \Cont_0^\infty(J)$ and
$J$ fulfils Hypothesis \ref{hyp:WRef} (\ref{hyp:WRef_EllItem2}). One
expects that the symbol of $f(\P_\epsilon)$ is mainly supported in the
set $K \times \T^d$ characterised by \eqref{eq:chiAK}. Thus the
product of $f(\P_\epsilon)$ and $\Op_{\epsilon,1/2}^\T (1-\chi)$,
where $1 - \chi$ is supported outside of $K \times \T^d$, is small in
trace norm. The proof of Lemma \ref{lem:fPChiTrace} uses the integral
representation \eqref{eq:FuncCalcAdjust} based on the
resolvent. Therefore a version of Lemma \ref{lem:DecrOutsComp}
depending on the resolvent parameter $z$ is crucial. To obtain this
more refined version, we shall reconsider the idea of the proof of
Lemma \ref{lem:DecrOutsComp}.

\begin{Lem}
  \label{lem:fPChiTrace}
  Assume the symbol $a$ to satisfy Hypothesis \ref{hyp:WRef} (1) and
  (2). Then there is a $(\sup J)$-adjustment $\adjust{a}$ of $a$ such
  that for some compact subset $K \subset \R^d$
  \begin{align}
    \label{eq:chiAK}
     \mbox{some neighbourhood of } \; G := \overline{\bigcup_{\epsilon
         \in (0,\epsilon_0]}\supp \left(
     (a-\adjust{a})(\cdot,\cdot;\epsilon) \right)} \; \mbox{ is
         contained in } \; K \times \T^d .
  \end{align}
  Let $\P_\epsilon$ be as in Proposition \ref{prop:PSelfAdj}. Then for any
  $\adjust{a}$ and $K$ fulfilling \eqref{eq:chiAK}, any bounded $\chi
  \in \Cont^\infty (\R^d \times \T^d)$ with $\left.\chi \right|_{K
    \times \T^d}\equiv 1$ and any $f \in \Cont_0^\infty(J)$, setting
  $\mathbf{X}_\epsilon := \Op_{\epsilon,1/2}^\T (1-\chi)$ we have
  \begin{align}
    \label{eq:fxxftrn}
    \trn{f(\P_\epsilon) \mathbf{X}_\epsilon} +
    \trn{\mathbf{X}_\epsilon f(\P_\epsilon)} = \O(\epsilon^\infty).
  \end{align}
\end{Lem}

\begin{proof}
  By Hypothesis \ref{hyp:WRef} (2) and the discussion preceding Lemma
  \ref{lem:FuncCalcAdjust}, in particular using the uniformity in
  $\epsilon$, there are $\adjust{a}$ and $K$ fulfilling
  \eqref{eq:chiAK}.

  Let $\chi \in \Cont^\infty (\R^d \times \T^d)$ be bounded with
  $\left.\chi \right|_{K \times \T^d}\equiv 1$ and $f \in
  \Cont_0^\infty(J)$. Let further $\adjOP{\P}_\epsilon$ and $\aac{f}$
  be as in Lemma \ref{lem:FuncCalcAdjust}. Then, for any $\epsilon >
  0$ sufficiently small, multiplying \eqref{eq:TraceClassFuncCalc} by
  $\mathbf{X}_\epsilon$ from the right, we have
  \begin{align}
    \label{eq:fXTrace}
    f(\P_\epsilon) \mathbf{X}_\epsilon = \int_\C
    \mathbf{W}_{\epsilon}(z)  L(dz)
  \end{align}
  where for $\Im z \neq 0$
  \begin{align}
    \label{eq:fXTraceW}
    \mathbf{W}_{\epsilon}(z) &:= \mathbf{W}_{1,\epsilon}(z)
    \mathbf{W}_{2,\epsilon}(z), \\ \label{eq:fXTraceW1}
    \mathbf{W}_{1,\epsilon}(z) &:= -\frac{1}{\pi} \dz\aac{f}(z) (z -
    \P_\epsilon)^{-1}, \\ \label{eq:fXTraceW2}
    \mathbf{W}_{2,\epsilon}(z) &:= \left(\Op_{\epsilon,1/2}^\T (a -
    \adjust{a}) \right) (z - \adjOP{\P}_\epsilon)^{-1}
    \mathbf{X}_\epsilon.
  \end{align}
  By the arguments given in the proof of Proposition \ref{prop:fPTraceClass}, $\mathbf{W}_{\epsilon}(z)$ can be continuously
  extended by $0$ to $\Im z = 0$.
  
  In \eqref{eq:fXTrace}, for any $\epsilon > 0$ sufficiently small,
  $f(\P_\epsilon)\mathbf{X}_\epsilon$ is of trace class since
  $f(\P_\epsilon)$ is of trace class (Proposition
  \ref{prop:fPTraceClass}) and $\mathbf{X}_\epsilon$ is bounded
  (Proposition \ref{prop:discrCaldVail}). For each $z \in \C$ with
  $\Im z \neq 0$, the operator $\mathbf{W}_{1,\epsilon}(z)$ is bounded
  and, by Proposition \ref{prop:TraClCriPDO} and the ideal property of
  the trace class operators, the operator $\mathbf{W}_{2,\epsilon}(z)$
  is of trace class. Thus $\mathbf{W}_{\epsilon}(z)$ is of trace class
  and, by \eqref{eq:TraceIdealEstim}, we have
  \begin{align}
    \label{eq:fxxfTrnW}
    \trn{\mathbf{W}_\epsilon(z)} \leq
    \norm{\mathbf{W}_{1,\epsilon}(z)} \trn{\mathbf{W}_{2,\epsilon}(z)}
    \qquad (z \in \C,\; \Im z \neq 0).
  \end{align}
  We now prove that $\mathbf{W}_\epsilon(z)$ is of order
  $\O(\epsilon^\infty)$ in trace norm, uniformly for $z \in \C$.

  The parameter $z$ appears in both the function $\aac{f}$ and the
  symbols of the resolvents in \eqref{eq:fXTraceW1} and
  \eqref{eq:fXTraceW2}. Since $\aac{f}$ has compact support, it is
  sufficient to verify uniformity on some compact subset of $\C$.
  Using the resolvent estimate \eqref{eq:resolvEstimP} and the
  estimate \eqref{eq:aacf} for the almost analytic extension, we have
  for any $M \in \N$
  \begin{align}
    \label{eq:fxxfTrnW1}
    \norm{\mathbf{W}_{1,\epsilon}(z)} \leq \frac{1}{\pi} \left|
    \dz\aac{f}(z) \right| |\Im z|^{-1} \leq C_M |\Im z|^{M} \qquad
    \mbox{for } z \in \C \mbox{ with } \Im z \neq 0
  \end{align}
  where $C_M$ is a constant not depending on $z$.  We now verify that
  for any $N \in \N$ there is some $M \in \N$ such that
  \begin{align}
    \label{eq:fxxfTrnW2}
    \trn{\mathbf{W}_{2,\epsilon}(z)} = |\Im z|^{-M} \O(\epsilon^N)
    \qquad \mbox{for } \epsilon \downarrow 0, \mbox{ uniformly in } z
    \mbox{ with } \Im z \neq 0,
  \end{align}
  by justifying and applying a parameter dependent version of Lemma
  \ref{lem:DecrOutsComp}. Granted \eqref{eq:fxxfTrnW2}, we may combine
  \eqref{eq:fxxfTrnW}, \eqref{eq:fxxfTrnW1} and \eqref{eq:fxxfTrnW2}
  to see that the integrand $\mathbf{W}_\epsilon(z)$ is of order
  $\O(\epsilon^\infty)$ in trace norm, uniformly for $z$. Thus, using
  the compact support of $\mathbf{W}_\epsilon(z)$, the integral
  \eqref{eq:fXTrace} is of order $\O(\epsilon^\infty)$ in trace norm
  and therefore $\trn{f(\P_\epsilon)\mathbf{X}_\epsilon} =
  \O(\epsilon^\infty)$. An analogue argument proves $\trn{
    \mathbf{X}_\epsilon f(\P_\epsilon)} = \O(\epsilon^\infty)$.

  It remains to prove \eqref{eq:fxxfTrnW2}. Let $N \in \N$. Due to
  Theorem \ref{theo:FuncCalcNonDiscr}, for some $\epsilon_1 >0$
  sufficiently small, the symbol $b_{z-\adjust{a}}$ of $(z -
  \adjOP{\P}_\epsilon)^{-1}$ is an element of
  $\Sy^{0}(m^{-1},\epsilon_1)(\R^d \times \T^d)$. We shall write
  $\adjust{b}_z := b_{z-\adjust{a}}$. According to Proposition
  \ref{prop:SharpProd}, we have $\mathbf{W}_{2,\epsilon}(z) =
  \Op_{\epsilon,1/2}^\T q(z)$ with
  \begin{align}
    q(z) := (a-\adjust{a}) \compw \adjust{b}_z \compw (1-\chi) \qquad
    \mbox{for } z \in \C \mbox{ with } \Im z \neq 0.
  \end{align}
  Since $a-\adjust{a}$ and $1-\chi$ have disjoint support and $a -
  \adjust{a}$ has compact support, we may apply Lemma
  \ref{lem:DecrOutsComp} to get $q(z) \in
  \Sy^0(\jap{x}^{-N},\epsilon_1)(\R^d \times \T^d)$. We may now trace
  the dependence on $z$ in the proof of Lemma \ref{lem:DecrOutsComp}
  with $q = q(z)$. Due to Proposition \ref{prop:SharpProd}, $q(z)$ has
  an asymptotic expansion $q(z) \sim \sum_{j=0}^\infty \epsilon^j
  q_j(z)$ with the expansion terms $q_j(z)$ given by
  \eqref{eq:nonlocalAsympExp}. In detail we have
  \begin{align}
    q_j(z) (x,\xi;\epsilon) = \sum_{\substack{\alpha_1, \alpha_2,
        \alpha_3 \in \N^{2d} \\ |\alpha_1 + \alpha_2 + \alpha_3| =
        2j}} C_{j \alpha_1 \alpha_2 \alpha_3} \left(
    \partial^{\alpha_1} (a - \adjust{a}) \partial^{\alpha_2}
    \adjust{b}_z \partial^{\alpha_3} (1-\chi) \right)(x,\xi;\epsilon)
  \end{align}
  with suitable constants $C_{j \alpha_1 \alpha_2 \alpha_3}$ not
  depending on $z$. In fact $q_j(z) = 0$ for each $j$ since
  $a-\adjust{a}$ and $1-\chi$ have disjoint support. Thus for the
  remainder term $R_{N+d}(z) := q(z) - \sum_{j=0}^{N+d-1} \epsilon^j
  q_j(z)$, we have
  \begin{align}
    \label{eq:remainderEqualSymb}
    R_{N+d}(z) = q(z) \in \Sy^{N+d}(\jap{x}^{-N},\epsilon_1)(\R^d \times
    \T^d).
  \end{align}
  By Proposition \ref{prop:SharpProd}, the mapping $(a-\adjust{a},
  \adjust{b}_z, 1-\chi) \mapsto R_{N+d}(z)$ is continuous in the
  Fréchet topology. By Proposition \ref{prop:SwitchQuant} (with $s =
  1/2,\, t=0$), the change of quantisation $q(z) \mapsto q_0(z)$,
  where $q_0(z)$ is in the same space as $q(z)$, is also
  continuous. Thus, using the identity in
  \eqref{eq:remainderEqualSymb}, the mapping $(a-\adjust{a},
  \adjust{b}_z, 1-\chi) \mapsto q_0(z)$ is continuous in the Fréchet
  topology. This, expressed in terms of the Fréchet seminorms
  $\norm{\cdot}_\alpha$ introduced in \eqref{eq:FrechetSNSym}, means
   \begin{align}
    \label{eq:q0zCont}
    \epsilon^{-(N+d)}\jap{x}^N |q_0(z)(x,\xi;\epsilon)| \leq
    \norm{q_0(z)}_0 \leq \sum_{\alpha_1, \alpha_2, \alpha_3 \in
      \N^{2d}} D_{\alpha_1 \alpha_2 \alpha_3} \norm{a -
      \adjust{a}}_{\alpha_1} \norm{\adjust{b}_{z}}_{\alpha_2}
    \norm{1-\chi}_{\alpha_3}
  \end{align}
  with suitable constants $D_{\alpha_1 \alpha_2 \alpha_3}$ vanishing
  for sufficiently large multi-indices $\alpha_1, \alpha_2, \alpha_3$
  and not depending on $z$. Since $\adjust{a}$ satisfies Hypothesis \ref{hyp:WRef} (1), the estimate \eqref{eq:bzEstim} holds for
  $\adjust{b}_z$. It thus follows from \eqref{eq:q0zCont} that for some $M \in \N$
  \begin{align}
    \label{eq:q0Estimxzeps}
    |q_0(z)(x,\xi;\epsilon)| = \jap{x}^{-N} |\Im z|^{-M}
    \O(\epsilon^{N+d}) \; \, \mbox{uniformly for } x \in \R^d,\, \xi
    \in \T^d \mbox{ and } z \in \C \mbox{ with } \Im z \neq 0.
  \end{align}
  Combining \eqref{eq:q0Estimxzeps} with the $z$-dependent version of
  \eqref{eq:nonlocalChangeQuant} for $N$ sufficiently large proves
  \eqref{eq:fxxfTrnW2}. This completes the proof of Lemma
  \ref{lem:fPChiTrace}.
\end{proof}

We are now ready to prove the main result of this section.

\begin{theo}
  \label{theo:TraceAsympExp}
  Assume Hypothesis \ref{hyp:WRef} and let $\P_\epsilon$ be as in
  Proposition \ref{prop:PSelfAdj}. Then for any $\epsilon > 0$
  sufficiently small and for any $f \in \Cont_0^\infty(J)$, the
  operators $f(\P_\epsilon)$ and $\Op_{\epsilon,1/2}^\T c_j$ for $j
  \in \N$ are of trace class. Here the functions $c_j \in
  \Cont^\infty(\R^d \times \T^d)$ - given by
  \eqref{eq:FuncCalcAsympcj} - form the asymptotic expansion of the
  Weyl symbol $c$ of $f(\P_\epsilon)$. Moreover,
    \begin{align}
      \label{eq:TraceAsympExpNorm}
    \trn{f(\P_\epsilon) - \sum_{j=0}^{N-1}\epsilon^j
      \Op_{\epsilon,1/2}^\T c_j} = \O(\epsilon^{N-d}) \qquad (\epsilon
    \downarrow 0) \qquad \mbox{for } N \in \N^*
  \end{align}
  and
  \begin{align}
    \label{eq:TraceAsympExp}
    \tr (f(\P_\epsilon)) \sim \frac{1}{(2\pi\epsilon)^d} \sum_{j=0}^\infty
    \epsilon^j \int_{\R^d} \int_{\T^d}
    c_j(x,\xi) d\xi dx.
  \end{align}
  In particular,
  \begin{align}
    \label{eq:TraceAsympExpPart}
    \tr (f(\P_\epsilon)) = \frac{1}{(2\pi\epsilon)^d} \left(
    \int_{\R^d} \int_{\T^d} f(a_0(x,\xi)) d\xi dx + R_1(\epsilon)
    \right),
  \end{align}
  where $|R_1(\epsilon)| \leq C \epsilon$ with some constant $C$ only
  depending linearly on finitely many derivatives of $f$ and $a$.
\end{theo}

\begin{proof}
  $f(\P_\epsilon)$ is of trace class for $\epsilon$ sufficiently small
  due to Proposition \ref{prop:fPTraceClass}.
  
  We claim that the functions $c_j$ have compact support. As a
  consequence, the operators $\Op_{\epsilon,1/2}^\T c_j$ are of trace
  class according to Proposition \ref{prop:TraClCriPDO}. To prove the
  claim, we choose some $(\sup J)$-adjustment $\adjust{a} \in
  \Sy^0(m,\epsilon_0)(\R^d \times \T^d)$ of $a$ and let $G$ and some
  compact $K$ be as in \eqref{eq:chiAK}. In particular, we shall show
  \begin{align}
    \label{eq:suppcjG}
    \supp c_j \subset G.
  \end{align}
  Since $f$ is supported in $J$, we have by formula
  \eqref{eq:FuncCalcAsympcj} that $\supp c_j \subset a_0^{-1} (J)$. It
  is therefore sufficient to show that $a_0^{-1}(J) \subset G$. For
  this let $(x,\xi) \in a_0^{-1} (J)$. Since $\adjust{a} > \sup J$,
  there is $\delta > 0$ such that
    \begin{align}
      \label{eq:TraceAsympExpProof1}
      \adjust{a} (x,\xi; \epsilon) - a_0(x,\xi) > \delta \qquad
      \mbox{for any } \epsilon \in (0,\epsilon_0].
    \end{align}
    According to the definition of the asymptotic expansion we have
    \begin{align}
      \label{eq:TraceAsympExpProof2}  
      a_0(x,\xi) - a(x,\xi;\epsilon) > - \frac{\delta}{2}
    \end{align}
    for any $\epsilon > 0$ sufficiently small.
    \eqref{eq:TraceAsympExpProof1} and \eqref{eq:TraceAsympExpProof2}
    then yield
    \begin{align}
      (\adjust{a} - a)(x,\xi;\epsilon) = \adjust{a} (x,\xi; \epsilon)
      - a_0(x,\xi) + a_0(x,\xi) - a(x,\xi;\epsilon) > \frac{\delta}{2}
      > 0,
    \end{align}
    so $(x,\xi) \in \supp (a - \adjust{a})$ for small
    $\epsilon$. Therefore $(x,\xi) \in G$. This proves \eqref{eq:suppcjG}. Thus the
    functions $c_j$ have compact support.

  Now define
  \begin{align}
    \mathbf{R}_N := f(\P_\epsilon) - \sum_{j=0}^{N-1} \epsilon^j
    \Op_{\epsilon,1/2}^\T c_j
  \end{align}
  and let $\chi \in \Cont_0^\infty(\R^d \times \T^d)$ satisfy
  $\left. \chi \right|_{K \times \T^d} \equiv 1$. Then
  \begin{align}
    \label{eq:TraceAsympExpProof5}
    \trn{\mathbf{R}_N} \leq \trn{\left(\Op_{\epsilon,1/2}^\T\chi \right)
      \mathbf{R}_N} + \trn{\left(\Op_{\epsilon,1/2}^\T(1-\chi) \right)
      \mathbf{R}_N}.
  \end{align}
  It follows from \eqref{eq:suppcjG} and the definition of $\chi$ that
  the supports of $c_j$ and $1-\chi$ are disjoint. Therefore, we
  conclude from Lemma \ref{lem:DecrOutsComp} and Lemma
  \ref{lem:fPChiTrace} that
  \begin{align}
    \label{eq:TraceAsympExpProof3}
      \trn{\left(\Op_{\epsilon,1/2}^\T (1- \chi)
      \right) \mathbf{R}_N} = \O(\epsilon^\infty) \qquad (\epsilon
    \downarrow 0).
  \end{align}
  Furthermore, due to the general trace norm estimate
  \eqref{eq:TraceIdealEstim}, the trace norm estimate
  \eqref{eq:TraClEstPDO} for discrete-type pseudo-differential
  operators and the discrete version of the Theorem of
  Calder{\'o}n-Vailloncourt (Proposition \ref{prop:discrCaldVail}), we get
  \begin{align}
    \label{eq:TraceAsympExpProof4}
    \trn{\left( \Op_{\epsilon,1/2}^\T \chi \right) \mathbf{R}_N} \leq
    \trn{\Op_{\epsilon,1/2}^\T \chi} \norm{\Op_{\epsilon,1/2}^\T
      \left( c - \sum_{j=0}^{N-1} \epsilon^j c_j \right)} =
    \O(\epsilon^{N-d}).
  \end{align}
  Combining \eqref{eq:TraceAsympExpProof5},
  \eqref{eq:TraceAsympExpProof3} and \eqref{eq:TraceAsympExpProof4},
  we get $\trn{\mathbf{R}_N} = \O(\epsilon^{N-d})$.

  Using the trace formula \eqref{eq:TracePDO}, we have
  \begin{align}
    \label{eq:TraceOpcj}
    \tr \left( \Op_{\epsilon,1/2}^\T c_j \right) = \frac{1}{(2\pi)^d}
    \sum_{x \in \epsilon\Z^d} \int_{\T^d} c_j(x,\xi) d\xi.
  \end{align}
  Since the trace is bounded by the trace norm,
  \eqref{eq:TraceAsympExp} is a consequence of
  \eqref{eq:TraceAsympExpNorm} and \eqref{eq:TraceOpcj}, using that
  due to the compact support of $c_j$ we may choose $\varphi \equiv 0$
  in Proposition \ref{prop:PoisApp} in order to approximate the sum
  $\sum_{x \in \epsilon\Z^d}$ by an integral
  $\epsilon^{-d}\int_{\R^d}dx$ with remainder of order
  $\O(\epsilon^\infty)$ for $\epsilon \downarrow 0$.

  The statement \eqref{eq:TraceAsympExpPart} is a consequence of
  \eqref{eq:TraceAsympExp}, using $c_0 = f \circ a_0$ and the
  statement on the remainder terms in Theorem
  \ref{theo:FuncCalcNonDiscr} (\ref{item:FuncCalcNonDiscr3}).
\end{proof}

\subsection{Rough Weyl asymptotics}

As a direct consequence of Theorem \ref{theo:TraceAsympExp}, we get
the following rough Weyl asymptotics for the number of eigenvalues.
Sharpening
the remainder estimate in the next section will be the main result of this paper. 
\begin{cor}
  \label{cor:WeylRough}
  Assume the symbol $a$ and the interval $J$ to satisfy Hypothesis
  \ref{hyp:WRef}, in particular, $a \sim \sum_{j=0}^\infty \epsilon^j
  a_j$. Let $\P_\epsilon$ be the self-adjoint operator as in
  Proposition \ref{prop:PSelfAdj}. Let $\alpha,\beta \in \R$ with
  $\alpha < \beta$ and $[\alpha, \beta] \subset J$ and denote by
  $\Neig([\alpha,\beta];\epsilon)$ the number of eigenvalues of
  $\P_\epsilon$ in $[\alpha,\beta]$. Defining upper and lower phase
  space volume with respect to $a_0$ and $[\alpha, \beta]$ by
  \begin{align}
    \label{eq:upperLowerVolDef}
    \overline{V}([\alpha,\beta]) &= \lim_{\delta \downarrow 0} \vol
    \left( a_0^{-1}([\alpha - \delta, \beta + \delta]) \right), \quad
    \underline{V}([\alpha,\beta]) = \lim_{\delta \downarrow 0} \vol
    \left( a_0^{-1}([\alpha + \delta, \beta - \delta])\right)
  \end{align}
  with $\vol$ given by \eqref{eq:phaseSpaceVol}, one has
  \begin{align}
    \label{eq:NeigAsympRough}
    \frac{1}{(2\pi \epsilon)^d} \left( \underline{V}([\alpha,\beta]) +
      o(1) \right) \leq \Neig([\alpha,\beta];\epsilon) \leq
    \frac{1}{(2\pi \epsilon)^d} \left( \overline{V}([\alpha,\beta]) +
      o(1) \right) \qquad (\epsilon \downarrow 0).
  \end{align}
  Furthermore, if $\alpha$ and $\beta$ are both non-critical values of
  $a_0$, then the lower and upper phase space volume in
  \eqref{eq:NeigAsympRough} coincide and we have
  \begin{align}
    \label{eq:NeigAsympRoughAdd}
    \Neig([\alpha,\beta];\epsilon) = \frac{1}{(2\pi \epsilon)^d}
    \left( \vol \left( a_0^{-1}([\alpha,\beta]) \right) +
    O(\epsilon^\nu) \right) \qquad (\epsilon \downarrow 0) \qquad
    \mbox{for some } \nu > 0.
  \end{align}
\end{cor}

\begin{proof}[Sketch of the proof]
   Since $J$ is open, we can choose for any $\delta > 0$ sufficiently
   small, functions $\underline{f}_\delta,
   \overline{f}_\delta \in \Cont_0^\infty(J, [0,1])$ such that
   \begin{align}
     \label{eq:fupperlower}
    \I_{[\alpha+\delta,\beta-\delta]} \leq \underline{f}_\delta \leq
    \I_{[\alpha,\beta]} \leq \overline{f}_\delta \leq
    \I_{[\alpha-\delta,\beta+\delta]}
   \end{align}
   and
   \begin{align}
     \label{eq:fupperlower2}
     \sup_{\lambda \in \R} \left(
     |\underline{f}_\delta^{(k)}(\lambda)| +
     |\overline{f}_\delta^{(k)}(\lambda)| \right) = \O(\delta^{-k})
     \qquad (\delta \downarrow 0) \qquad \mbox{for any } k \in \N.
   \end{align}
  We recall that by Theorem \ref{theo:TraceAsympExp} the operators
  $\underline{f}_\delta(\P_\epsilon)$ and
  $\overline{f}_\delta(\P_\epsilon)$ are of trace class for any
  $\epsilon, \delta$ sufficiently small. Thus, using
  \eqref{eq:fupperlower} and the spectral theorem,
  \begin{align}
    \label{eq:WeylRoughEigenvalueEstim}
    \tr \underline{f}_\delta (\P_\epsilon) \leq
    \Neig([\alpha,\beta];\epsilon) \leq \tr \overline{f}_\delta
    (\P_\epsilon) \qquad \mbox{for } \epsilon, \delta \mbox{
      sufficiently small}.
  \end{align}
  Firstly, we obviously have by the definition of the limit
  \begin{align}
    \label{eq:underVo1}
    \underline{V}([\alpha,\beta]) + o(1) = \int_{\R^d} \int_{\T^d} (\I_{[\alpha +
        \delta,\beta - \delta]} \circ a_0)(x,\xi) dxd\xi \qquad (\delta \downarrow 0). 
  \end{align}
  Secondly, using \eqref{eq:fupperlower} and \eqref{eq:TraceAsympExpPart} with
  $f = \underline{f}_\delta$ and $c_0 = \underline{f}_\delta \circ
  a_0$, we get
  \begin{align}
    \frac{1}{(2\pi \epsilon)^d} \int_{\R^d} \int_{\T^d} (\I_{[\alpha +
        \delta,\beta - \delta]} \circ a_0)(x,\xi) dxd\xi &\leq
    \frac{1}{(2\pi \epsilon)^d} \int_{\R^d} \int_{\T^d}
    \underline{f}_\delta (a_0(x,\xi)) dx d\xi
    \nonumber\\[1mm] \label{eq:WeylRoughProof} &
    = \tr \underline{f}_\delta
    (\P_\epsilon) + \epsilon^{-d}R(\delta,\epsilon)
  \end{align}
  where, due to \eqref{eq:fupperlower2} and the statement 
  on $R_1(\epsilon)$
  in \eqref{eq:TraceAsympExpPart}, there is some $k > 0$ such that
  \begin{align}
    \label{eq:RdeltaepsRoughWeyl}
   R(\delta,\epsilon) = \delta^{-k}O(\epsilon) \qquad \mbox{for }
   \epsilon \downarrow 0, \mbox{ uniformly for } \delta.
  \end{align}
  We choose $\delta(\epsilon)$ to fulfil $\delta(\epsilon) = o(1)$
  and $1/\delta(\epsilon) = o(\epsilon^{-1/k})$. Then
  $R(\delta(\epsilon),\epsilon) = o(1)$ and therefore, combining
  \eqref{eq:underVo1}, \eqref{eq:WeylRoughProof} and
  \eqref{eq:WeylRoughEigenvalueEstim},
  \begin{align}
    \label{eq:NeigBoundBelowRough}
    \frac{1}{(2\pi \epsilon)^d} \left( \underline{V}([\alpha,\beta]) +
    o(1) \right) \leq \tr
    \underline{f}_{\delta(\epsilon)}(\P_\epsilon) \leq
    \Neig([\alpha,\beta];\epsilon) \qquad (\epsilon \downarrow 0).
  \end{align}
  This proves the first inequality in \eqref{eq:NeigAsympRough}. The
  second inequality can be derived analogously. We note that the rough
  estimate in \eqref{eq:NeigAsympRough} is due to the rough estimate
  in \eqref{eq:underVo1}.

  We shall now prove \eqref{eq:NeigAsympRoughAdd} under the additional
  assumption that $\alpha$ and $\beta$ are non-critical values of
  $a_0$. In this case, the rough estimate in \eqref{eq:underVo1} can
  be improved. For some neighbourhood $U$ of the regular values
  $\alpha$ and $\beta$, we may construct in $a_0^{-1}(U)$ the
  Liouville form $L$ introduced in \eqref{eq:LiouvilleFormDef} to
  represent the symplectic volume form as $d vol = da_0 \wedge L$. We
  then write
  \begin{align}
    \label{eq:PhaseSpaceVolDiff1}
    \vol a_0^{-1}([\alpha - \delta, \beta + \delta]) - \vol
    a_0^{-1}([\alpha,\beta]) &= \int_{a_0 \in [\alpha-\delta, \alpha]
      \cup [\beta,\beta+\delta]} da_0 \wedge L, \\ \label{eq:PhaseSpaceVolDiff2} \vol
    a_0^{-1}([\alpha,\beta]) - \vol a_0^{-1}([\alpha + \delta, \beta -
      \delta]) &= \int_{a_0 \in [\alpha, \alpha+\delta] \cup
      [\beta-\delta,\beta]} da_0 \wedge L.
  \end{align}
  Using that 
  \begin{align}
    \int_{[\alpha-\delta, \alpha] \cup [\beta, \beta+\delta]} da_0 =
    \O(\delta) \qquad \mbox{and} \qquad \int_{[\alpha, \alpha+\delta]
      \cup [\beta-\delta, \beta]} da_0 = \O(\delta) \qquad \mbox{for }
    \delta \downarrow 0,
  \end{align}
  we obtain from \eqref{eq:PhaseSpaceVolDiff1} and
  \eqref{eq:PhaseSpaceVolDiff2} for the upper and lower phase space
  volume defined in \eqref{eq:upperLowerVolDef}
  \begin{align}
    \label{eq:allVolSame}
    \overline{V}([\alpha,\beta]) = \underline{V}([\alpha,\beta]) =
    \vol a_0^{-1}([\alpha,\beta]) = \vol a_0^{-1}([\alpha \mp \delta,
      \beta \pm \delta]) + \O(\delta) \qquad (\delta \downarrow 0).
  \end{align}
  Therefore, for the case of non-critical values $\alpha$ and $\beta$,
  the remainder of order $o(1)$ in \eqref{eq:underVo1} is actually of
  order $\O(\delta)$. Choosing $\delta(\epsilon) = \epsilon^{k+1}$ for
  $k$ given in \eqref{eq:RdeltaepsRoughWeyl}, we have
  $R(\delta(\epsilon),\epsilon) = \O(\epsilon^{1/(k+1)})$ for
  $\epsilon \downarrow 0$. Thus, reconsidering in this special case
  the arguments around \eqref{eq:NeigBoundBelowRough}, we get the
  improvement in the remainder estimate stated in
  \eqref{eq:NeigAsympRoughAdd}.
\end{proof}

We remark that, similarly to the setting in Corollary
\ref{cor:WeylRough}, a Hamiltonian given by a discrete Laplacian plus
$\Cont^\infty$-potential without the additional assumption on
regularity in $\alpha$ and $\beta$ has been treated in
\cite{Ka23}. This is a special case of a symbol which is analytic in a
group of variables. In this case, one obtains the first equalities in
\eqref{eq:allVolSame} but not the last. We shall not investigate
improved error estimates for such kinds of symbols.

It is the content of Theorem \ref{theo:RefinedWeyl} that $\nu$ in
\eqref{eq:NeigAsympRoughAdd} for the setting of non-critical values
$\alpha$ and $\beta$ can actually be chosen as $\nu = 1$. The proof of
this statement is the main focus of our work and shall be given in the
next section. It requires additional techniques such as the
semiclassical approximation of the time evolution of $f(\P_\epsilon)$
given in Theorem \ref{theo:PseudoFourIntApprox} below.

\section{Proof of Theorem \ref{theo:RefinedWeyl}}
\label{sec:ProofOfTheorem}

In order to prove Theorem \ref{theo:RefinedWeyl}, we shall follow the strategy
of \cite[Chapter 10]{DiSj} for the non-discrete setting. First, we see
in Subsection \ref{subsec:MainProof1} that the proof of Theorem
\ref{theo:RefinedWeyl} can be reduced to the proof of Proposition
\ref{prop:f1f3asymp}, where the neighbourhoods
of the interval boundaries $\alpha$ and $\beta$ are analysed. The interior of the
interval can already be treated by means of the trace asymptotics in Theorem
\ref{theo:TraceAsympExp}. 

The main tool for proving
Proposition \ref{prop:f1f3asymp} is a semiclassical approximation of
the time evolution of $\P_\epsilon$ with respect to the trace
norm. This construction is given in Subsection \ref{subsec:MainProof2}
and follows standard ideas, however an additional analysis addressing
the periodicity with respect to the momentum variable is needed. 

We
complete the proof of Proposition \ref{prop:f1f3asymp} in Subsection
\ref{subsec:MainProof3}, where we relate the Fourier transform of the
time evolution to the density of eigenvalues near the non-critical
points $\alpha$ and $\beta$ up to an error of order
$\O(\epsilon)$. The approximation of the
time evolution as a Fourier integral operator of discrete type
induced by a certain kernel and a phase function allows to apply trace
estimates (Section \ref{sec:TraceEst}), Poisson summation techniques
(Appendix \ref{sec:appPoisson}) and the method of stationary phase,
which are the essential techniques here. 

Throughout this section we
shall assume Hypothesis \ref{hyp:WRef} to ensure that all occurring
phase space volumes are finite.

\subsection{Reducing the proof to Proposition \ref{prop:f1f3asymp}}
\label{subsec:MainProof1}

Let the interval $J$ and the symbol $a$
with leading order symbol $a_0$ satisfy Hypothesis \ref{hyp:WRef}. Let $\P_\epsilon$ be the 
self-adjoint operator associated to $a$ and
let $[\alpha, \beta] \subset J$ where $\alpha$ and $\beta$ are
non-critical values of $a_0$. 

We follow \cite[chapter 10]{DiSj} and
choose $f_1, f_2, f_3 \in \Cont_0^\infty(\R)$ with supports in $J$
such that
\begin{align}
 \label{eq:sumf1f2f3}
 f_1 + f_2 + f_3 &= 1 \quad \mbox{on } [\alpha,\beta],
\end{align}
$\supp f_2 \subset (\alpha,\beta)$ and that $f_1$ and $f_3$ have
supports in neighbourhoods of $\alpha$ and $\beta$, respectively, only
consisting of non-critical values of $a_0$.

Supposing that these
neighbourhoods are chosen sufficiently small and denoting by $\Neig([\alpha,\beta];\epsilon)$ the
number of eigenvalues of $\P_\epsilon$ in $[\alpha,\beta]$ as in Theorem \ref{theo:RefinedWeyl}, we have the
decomposition
\begin{align}
  \notag
  \Neig([\alpha,\beta];\epsilon) &= \sum_{\alpha \leq \lambda_j \leq
    \beta} 1  = \sum_{\alpha \leq \lambda_j \leq \beta} (f_1 + f_2
  + f_3)(\lambda_j)\\
  \label{eq:UnitDecomp}
  &= \sum_{\lambda_j \geq \alpha} f_1(\lambda_j) + \sum_{\lambda_j}
  f_2(\lambda_j) + \sum_{\lambda_j \leq \beta} f_3(\lambda_j),
\end{align}
where we sum over eigenvalues $\lambda_j$ of $\P_\epsilon$ counted
with multiplicity. The sum in the middle is the trace of
$f_2(\P_\epsilon)$, which can be expanded asymptotically according to
Theorem \ref{theo:TraceAsympExp}. Neglecting higher order terms we
thus get
\begin{align}
  \label{eq:f2asymp}
  \sum_{\lambda_j} f_2(\lambda_j) = \frac{1}{(2\pi \epsilon)^d} \left(
  \int_{\substack{x \in \R^d,\, \xi \in \T^d}}
  f_2(a_0(x,\xi))  d\xi dx + \O(\epsilon) \right) \qquad (\epsilon
  \downarrow 0)
\end{align}
The remaining sums in \eqref{eq:UnitDecomp} require substantially
different arguments (the cut-offs $\lambda_j \geq \alpha$ and
$\lambda_j \leq \beta$ are not smooth). As in the well known
pseudodifferential setting, we shall use a semiclassical parametrix
for the unitary group induced by $\P_\epsilon$ to obtain

\begin{prop}
  \label{prop:f1f3asymp}
  Assume the symbol $a$ and the interval $J$ to satisfy Hypothesis
  \ref{hyp:WRef}, in particular, $a \sim \sum_{j=0}^\infty \epsilon^j
  a_j$. Let $\P_\epsilon$ be the self-adjoint operator as in
  Proposition \ref{prop:PSelfAdj}. Furthermore, let $[\alpha, \beta] \subset J$ where
  $\alpha$ and $\beta$ are non-critical values of $a_0$. Then for
  $f_1$ and $f_3$ from \eqref{eq:sumf1f2f3}, we have
  \begin{align}
    \label{eq:f1asymp}
    \sum_{\lambda_j \geq \alpha} f_1(\lambda_j) &= \frac{1}{(2\pi
      \epsilon)^d} \left( \int_{\substack{x \in \R^d,\, \xi \in
        \T^d \\ \alpha \leq a_0(x,\xi)}} f_1(a_0(x,\xi)) d\xi
    dx + \O(\epsilon) \right), \\
  \label{eq:f3asymp}
  \sum_{\lambda_j \leq \beta} f_3(\lambda_j) &= \frac{1}{(2\pi
    \epsilon)^d} \left( \int_{\substack{x \in \R^d,\, \xi \in
      \T^d \\ a_0(x,\xi) \leq \beta}} f_3(a_0(x,\xi)) d\xi dx
  + \O(\epsilon) \right)
  \end{align}
  for $\epsilon \downarrow 0$, where we sum over eigenvalues
  $\lambda_j$ of $\P_\epsilon$.
\end{prop}

With the help of \eqref{eq:f2asymp}, \eqref{eq:f1asymp} and
\eqref{eq:f3asymp} we may then replace the sums in
\eqref{eq:UnitDecomp} by their asymptotics and use the property
\eqref{eq:sumf1f2f3} to get
\begin{align*}
  \Neig([\alpha,\beta];\epsilon) &= \frac{1}{(2\pi \epsilon)^d} \left(
  \int_{\substack{x \in \R^d,\, \xi \in \T^d \\ \alpha \leq
      a_0(x,\xi) \leq \beta}} (f_1 + f_2 + f_3)(a_0(x,\xi)) d\xi dx +
  \O(\epsilon) \right) \\ &=
  \frac{1}{(2\pi \epsilon)^d} \left( \int_{\substack{x \in \R^d,\, \xi
      \in \T^d \\ \alpha \leq a_0(x,\xi) \leq \beta}} d\xi dx
  + \O(\epsilon) \right) \\
  &= \frac{1}{(2\pi \epsilon)^d}
    \left( \vol \left( a_0^{-1}([\alpha,\beta]) \right) +
      \O(\epsilon)\right)
\end{align*}
for $\epsilon \downarrow 0$. Granted Proposition \ref{prop:f1f3asymp}
this proves Theorem \ref{theo:RefinedWeyl}.

\subsection{Semiclassical approximation of the time evolution}
\label{subsec:MainProof2}

In this subsection we shall, for given $f \in \Cont_0^\infty(J)$,
construct Fourier integral operators $\U^{(f)}_\epsilon(t)$ of discrete
type, which approximate the time evolution $e^{it\P_\epsilon/\epsilon}
f(\P_\epsilon)$ to any order $\O(\epsilon^N)$ with respect to the
trace norm in a small neighbourhood of $t = 0$.
The use of $f(\P_\epsilon)$ introduces a localisation in energy. 
Our construction is summarised in 
Theorem \ref{theo:PseudoFourIntApprox}.

First we introduce a suitable class of Fourier integral operators
$\U_\epsilon(t)$ of discrete type (mapping $\l^2 (\epsilon \Z^d)$ to $\l^2 (\epsilon \Z^d)$), to which
our approximate time evolution $\U_\epsilon^{(f)}(t)$ belongs.  
In contrast to the
non-discrete setting, we later need both the kernel function to be periodic with
respect to the momentum variable and the phase function to fulfil an appropriate periodicity condition 
as specified in \eqref{equ:HamJacPeriod}.

At least formally, for $\epsilon \in (0,\epsilon_0]$ and
  $t \in \R$, the operator $\U_\epsilon(t)$ belonging to our class  
is induced by a kernel function $\mu$ and a Hamiltonian
  $H \in \Cont^\infty(\R^d \times \T^d)$ via the formula
  \begin{align}
    \label{eq:FIODef}
      \U_\epsilon(t) u(x) =
      \frac{1}{(2\pi)^d} \sum_{y \in \epsilon\Z^d} \int_{[-\pi,\pi]^d}
      e^{i(y\xi - \phi(t,x,\xi))/\epsilon} \mu(t,x,y,\xi;\epsilon) u(y)
      d\xi
\end{align}
for $u \in \l^2(\epsilon\Z^d),\, x \in \epsilon\Z^d$, where for some
numbers $T, L > 0$
\begin{enumerate}
\item $\mu$ is a symbol in $\Sy^0(1,\epsilon_0)(\R
  \times \R^d \times \R^d \times \T^d)$ and has support in $(-T,T)
  \times (-L,L)^d \times (-L,L)^d \times \T^d \times (0,\epsilon_0]$
    and
\item $\phi : \R \times \R^d \times \R^d
  \rightarrow \R$ is smooth in $(-T,T) \times (-L,L)^d \times
  \R^d$ and solves the Hamilton-Jacobi equation
  \begin{align}
    \label{eq:HamJacEq}
    \partial_t \phi(t,x,\xi) + H(x,\nabla_x \phi(t,x,\xi)) = 0, \quad
    \phi(0,x,\xi) = x\xi
  \end{align}
  in $(-T,T) \times (-L,L)^d \times \R^d$.
\end{enumerate}

We first observe that the relevant Hamilton-Jacobi equation \eqref{eq:HamJacEq} actually 
possesses smooth solutions of a certain periodicity type.

\begin{Lem}
  \label{lem:HamJacPeriod}
  For any smooth Hamiltonian which is $(2\pi\Z^d)$-periodic with
  respect to $\xi$, i.e. $H \in \Cont^\infty(\R^d \times \T^d)$, and
  for any compact set $K \subset \R^d$ there is a time $T > 0$ such
  that the associated Hamilton-Jacobi equation \eqref{eq:HamJacEq}
  with the specified initial condition has a unique smooth solution
  $\phi$ in the domain $(-T,T) \times K \times \R^d$. Furthermore,
  $\phi$ can be represented as
  \begin{align}\label{equ:HamJacPeriod}
    \phi(t,x,\xi) = x\xi + \phi_\T(t,x,\xi) \qquad (t \in (-T,T),\,
    x \in K,\, \xi \in \R^d)
  \end{align}
  where the function $\phi_\T$ is $(2\pi\Z^d)$-periodic with respect to $\xi$.
\end{Lem}

\begin{proof}
  Let $V$ be an open bounded subset of $\R^d$ containing $[-\pi,\pi]^d$. It is
  known (see \cite{DiSj,Robert1987AutourDL,hoermander2009analysis})
  that there is $T>0$ such that the Hamilton-Jacobi equation
  \eqref{eq:HamJacEq} has a unique smooth solution $\phi$ in $(-T,T)
  \times K \times V$. Let $\phi_\T \in \Cont^\infty((-T,T) \times K \times V)$ with
  \begin{align}
    \phi(t,x,\xi) = x\xi + \phi_\T(t,x,\xi) \qquad (t \in (-T,T),\,
    x \in K,\, \xi \in V).
  \end{align}
  Denote by $(\gamma_i)$ all elements of $2\pi\Z^d$ with norm $2\pi$.
  We fix $i \in \left\{ 1, \dots, 2d \right\}$ and set $\gamma :=
  \gamma_i$.  We have by the initial assumption on $V$ that $V \cap (V
  - \gamma)$ is an open neighbourhood of the $i$-th face of
  $[-\pi,\pi]^d$. Define then the function
  \begin{align}
    \tilde{\phi}(t,x,\xi) := x\xi + \phi_\T(t,x,\xi + \gamma) \qquad
    (t \in (-T,T),\, x \in K,\, \xi \in V - \gamma).
  \end{align}
  By checking that $\tilde{\phi}$ fulfils \eqref{eq:HamJacEq} in
  $(-T,T) \times K \times (V - \gamma)$, we conclude by the uniqueness
  of the solution of the Hamilton-Jacobi equation that $\phi (t,x,\xi)
  = \tilde{\phi} (t,x,\xi)$ for $t \in (-T,T)$, $x \in K$, $\xi \in V
  \cap (V - \gamma)$, so $\phi_\T (t,x,\xi + \gamma) = \phi_\T
  (t,x,\xi)$. Since $\gamma$ was chosen as any $\gamma_i$, we will get
  \begin{align}
    \label{eq:phiTPeriod}
    \phi_\T (t,x,\xi + \gamma_i) = \phi_\T (t,x,\xi) \qquad (t
    \in (-T,T),\, x \in K)
  \end{align}
  for any $\gamma_i$ and $\xi \in V \cap (V - \gamma_i)$.

  We first check the initial condition. Let $x \in K$. Since
  $\phi(0,x,\xi) = x\xi$, we have $\phi_\T(0,x,\xi) = 0$ for $\xi \in
  V$. So $\tilde{\phi}(0,x,\xi) = x\xi$ for $\xi \in V - \gamma$.

  We check now that $\tilde{\phi}$ fulfils the Hamilton-Jacobi
  differential equation. Using periodicity of $H$ and that $\phi$
  solves the Hamilton-Jacobi equation, we have for $t \in (-T,T)$, $x
  \in K$, $\xi \in V-\gamma$
  \begin{align}
    \notag
    \partial_t \tilde{\phi}(t,x,\xi) + H(x,\nabla_x \tilde{\phi}(t,x,\xi)) &= \partial_t \phi_\T(t,x,\xi + \gamma) + 
    H(x, \xi + \nabla_x \phi_\T(t,x,\xi + \gamma)) \\
    \notag
    &= \partial_t \phi_\T(t,x,\xi + \gamma) + H(x, \xi + \gamma + \nabla_x \phi_\T(t,x,\xi + \gamma)) \\
    \notag
    &= \partial_t \phi(t,x,\xi + \gamma) + H(x, \nabla_x \phi(t,x,\xi + \gamma)) \\
    \label{eq:HamiltonJac_phiTilde} &= 0.
  \end{align}
  Therefore the periodicity statement \eqref{eq:phiTPeriod} is
  valid. Since $(\gamma_i)$ generates $2\pi\Z^d$, we can now extend
  $\phi_\T$ uniquely to a periodic function on the domain $(-T,T)
  \times K \times \R^d$. This extension in turn is used to extend
  $\phi$ onto the same domain by defining
  \begin{align}
    \phi(t,x,\xi) := x\xi + \phi_\T(t,x,\xi) \qquad (t \in (-T,T),\, x
    \in K,\, \xi \in \R^d).
  \end{align}
  Arguing as in \eqref{eq:HamiltonJac_phiTilde} we see that $\phi$
  fulfils the Hamilton-Jacobi equation \eqref{eq:HamJacEq} on $(-T,T)
  \times K \times \R^d$.
\end{proof}

Observe also that due to Proposition
\ref{prop:TraClCri}, since $\mu$ is
compactly supported with respect to $x$ and $y$, we have

\begin{Lem}
  \label{lem:UepsTrCl}
  $\U_\epsilon(t)$ is of trace class, in particular it is bounded from $\l^2 (\epsilon \Z^d)$ to
  $\l^2 (\epsilon \Z^d)$. 
\end{Lem}

In fact, $\U_\epsilon (t)$ is even a finite rank operator due to the compact support of $\mu$ and as such clearly trace class.

As a consequence of Lemma \ref{lem:HamJacPeriod} and \ref{lem:UepsTrCl}, our class of operators formally 
given in \eqref{eq:FIODef} is non-empty as
a class of trace class operators. 

We are now ready to construct $\U^{(f)}_\epsilon(t)$.

\begin{theo}
  \label{theo:PseudoFourIntApprox}
  Assume the symbol $a$ with leading order $a_0$ and the interval $J$ to satisfy Hypothesis
  \ref{hyp:WRef} and let $\P_\epsilon$ be the associated self-adjoint operator as in
  Proposition \ref{prop:PSelfAdj}. Let $f \in \Cont_0^\infty(\R)$
  with $\supp f \subset J$ and let $\chi \in \Cont_0^\infty(\R^d)$
  with $\chi \equiv 1$ near some compact set $K \subset \R^d$ where
  $a_0^{-1}(J) \subset K \times \T^d$. Then there is a family
  $(\U^{(f)}_\epsilon(t))_{\epsilon,t}$ of operators of the form
  \eqref{eq:FIODef} induced by some kernel function $\mu$ and the
  Hamiltonian $H = a_0$ such that
  \begin{align}
    \label{eq:PseudoFIOApprox}
    \sup_{|t| < T} \trn{ \U^{(f)}_\epsilon(t) -
      e^{it\P_\epsilon/\epsilon}f(\P_\epsilon)} = \O(\epsilon^\infty)
    \qquad (\epsilon \downarrow 0)
  \end{align}
  for some number $T > 0$ (possibly shrinking the number $T$ from
  \eqref{eq:FIODef}) and
  \begin{align}
    \label{eq:FIOKernelInit}
    \mu(0,x,y,\xi;\epsilon) = \chi(x)\chi(y) c((x+y)/2,\xi;\epsilon)
  \end{align}
  for $x,y \in \R^d, \xi\in \T^d, \epsilon \in (0,\epsilon_0]$ where $c \in
    \Sy^0(1,\epsilon_0)(\R^d \times \T^d)$ is the Weyl symbol of
    $f(\P_\epsilon)$ (characterised by \eqref{eq:FuncCalcSymb} and Corollary \ref{cor:cperiodic}).
\end{theo}

\begin{proof} Fixing $f\in \Cont_0^\infty(\R)$ with $\supp f \subset J$ we write for simplicity 
$\U_\epsilon:= \U_\epsilon^{(f)}$ and we shall explicitly show that $\U_\epsilon$ can be constructed in the class of 
operators satisfying \eqref{eq:FIODef}.
  We follow the strategy of proof in \cite[chapter 10]{DiSj} for
  analogous statements in the non-discrete setting. We define  the operator
   \begin{align}
    \mathbf{W}_\epsilon(t) := \U_\epsilon(t) -
    e^{it\P_\epsilon/\epsilon} f(\P_\epsilon)\, ,
  \end{align}
  which is of trace class for $\epsilon$ sufficiently small, since $f(\P_\epsilon)$ is of
  trace class for small $\epsilon$ (Proposition
  \ref{prop:fPTraceClass}) and $\U_\epsilon(t)$ is of trace class
  (Lemma \ref{lem:UepsTrCl}).
  By the fundamental
  theorem of calculus $\mathbf{W}_\epsilon (t)$ can be represented as
  \begin{align}
    \label{eq:WtDiffRep}
    i \epsilon e^{-it\P_\epsilon/\epsilon} \mathbf{W}_\epsilon(t) =
    \int_0^t i\epsilon \partial_\tau \left( e^{-i\tau\P_\epsilon/\epsilon}
      \mathbf{W}_\epsilon(\tau) \right) d\tau + i \epsilon
    \mathbf{W}_\epsilon(0).
  \end{align}
  Using that $\trn{A_b A_{tr}} \leq \norm{A_b} \trn{A_{tr}}$ for a
  bounded operator $A_b$ and a trace class operator $A_{tr}$,
  we can take the trace norm in \eqref{eq:WtDiffRep} to find that for
  any number $T > 0$ the lhs of \eqref{eq:PseudoFIOApprox} can be bounded
  by
  \begin{align}
    \label{eq:supTrW}
    \sup_{|t| < T} \trn{ \mathbf{W}_\epsilon(t) } \leq
    \frac{T}{\epsilon} \sup_{|t| < T} \trn{ i\epsilon \partial_t
    \left( e^{-it\P_\epsilon/\epsilon}
    \mathbf{W}_\epsilon(t)\right) } + \trn{\mathbf{W}_\epsilon(0)}.
  \end{align}
  We shall see that $\U_\epsilon(t)$ can be constructed such that all
  the appearing trace norms are finite. The time derivative that
  appears in \eqref{eq:supTrW} can be written as
  \begin{align}
    \label{eq:partW}
    i\epsilon\partial_t \left( e^{-it\P_\epsilon/\epsilon}
    \mathbf{W}_\epsilon(t) \right) =  i\epsilon\partial_t \left( e^{-it\P_\epsilon/\epsilon}
    \U_\epsilon(t) \right) =
    e^{-it\P_\epsilon/\epsilon} (i\epsilon\partial_t + \P_\epsilon)
    \U_\epsilon(t).
  \end{align}
  By \eqref{eq:supTrW} and \eqref{eq:partW}, the estimate
  \eqref{eq:PseudoFIOApprox} holds if $\U_\epsilon(t)$ is chosen such
  that
  \begin{align}
    \label{eq:CondForFourierAppr}
    \trn{\mathbf{W}_\epsilon(0)} + \trn{(i\epsilon
  \partial_t + \P_\epsilon)\U_\epsilon(t)} = \O(\epsilon^\infty),
  \end{align}
  uniformly in $t \in (-T,T)$. We will construct operators
  $\U_\epsilon(t)$ which satisfy \eqref{eq:CondForFourierAppr} for
  some number $T > 0$.

  Let $L > 0$ with $\supp \chi \subset (-L,L)^d$. According to Lemma
  \ref{lem:HamJacPeriod} (applied with $K = [-L,L]^d$) we find $T' >
  0$ such that the Hamilton-Jacobi equation \eqref{eq:HamJacEq} with
  $H = a_0$ can be solved by some $\phi \in \Cont^\infty(\R \times
  \R^d \times \R^d)$ in the domain $(-T',T') \times (-L,L)^d \times
  \R^d$. Moreover we may assume $\phi$ to fulfil the condition
  \eqref{eq:gradPhiAss} of Proposition \ref{prop:OpUnitConj}, which will be used later. In the
  following we take this $\phi$ as phase function in $\U_\epsilon(t)$.
  
  We seek to control each summand in \eqref{eq:CondForFourierAppr}. We
  handle $\trn{\mathbf{W}_\epsilon(0)}$ first by defining
  $\U_\epsilon(t)$ for $t=0$. We have $\phi(0,x,\xi) = x\xi$ due to
  \eqref{eq:HamJacEq} and we choose the kernel $\mu$ of
  $\U_\epsilon(0)$ to fullfill \eqref{eq:FIOKernelInit}.
  $\mu(0,\cdot,\cdot,\cdot;\epsilon)$ then clearly has support in
  $(-L,L)^d \times (-L,L)^d \times \T^d$. 
  $\U_\epsilon(0)$ now takes the form
  \begin{align}
    \notag
     \U_\epsilon(0) u(x) &= \frac{1}{(2\pi)^d} \sum_{y \in
       \epsilon\Z^d} \int_{[-\pi,\pi]^d} e^{i(y - x)\xi/\epsilon}
     \chi(x)\chi(y) c((x+y)/2,\xi;\epsilon) u(y) d\xi \\
     \label{eq:UchifPchi} &= \left(
     \Op_{\epsilon,1/2}^\T \chi \right) f(\P_\epsilon) \left(
     \Op_{\epsilon,1/2}^\T \chi \right) u(x)
  \end{align}
  for $u \in \l^2(\epsilon\Z^d),\, x \in \epsilon\Z^d$. Thus, using Lemma
  \ref{lem:fPChiTrace}, we have
  \begin{align}
    \trn{\mathbf{W}_\epsilon(0)} = \trn{\U_\epsilon(0) -
      f(\P_\epsilon)} = \O(\epsilon^\infty).
  \end{align}
    Note that Lemma \ref{lem:fPChiTrace} is applicable because $\chi
    \equiv 1$ near some compact set $K$ where $a_0^{-1}(J) \subset K
    \times \T^d$. This means that for any open set $O \supset K$ we
    have $a(\cdot,\cdot;\epsilon)^{-1}(J) \subset O \times \T^d$ for
    $\epsilon$ small enough and that we can even choose $O$ such that
    there are a $(\sup J)$-adjustment $\adjust{a}$ of $a$ and a
    compact set $K' \supset O$ with $\left. \chi \right|_{K'} \equiv
    1$ and $({K'}^c \times \T^d) \cap G = \emptyset$ for $G =
    \bigcup_{\epsilon \in (0,\epsilon_0]}
      \supp((a-\adjust{a})(\cdot,\cdot;\epsilon))$.

  We shall now handle the second summand $\trn{(i\epsilon \partial_t +
    \P_\epsilon)\U_\epsilon(t)}$ in \eqref{eq:CondForFourierAppr} by
  constructing the operators $\U_\epsilon(t)$ for $t\neq 0$. This
  construction shall be reduced to well known results of the
  non-discrete setting, where defining the amplitude $\mu$ will be
  based on solutions of transport equations. Here, in contrast to the
  standard setting, it is absolutely crucial to check additional
  periodicity properties of $\mu$ with respect to $\xi$.

  To cover the discrete case,
  we use the restriction mapping $r_\epsilon$ combined with the
  restriction formula \eqref{eq:RestrForm}. Observe that, setting
  \begin{align}
    \eta_{t,y,\xi;\epsilon}(x) := e^{-i\phi(t,x,\xi)/\epsilon}
    \mu(t,x,y,\xi;\epsilon) \qquad (x,y \in \R^d,\, \xi \in \T^d,\, t \in
    (-T',T'),\, \epsilon \in (0,\epsilon_0]),
  \end{align}
  with yet undetermined $\mu(t,x,y,\xi;\epsilon)$,
  the operator $(i\epsilon \partial_t + \P_\epsilon)\U_\epsilon(t)$
  can be formally represented as
  \begin{align}
    \notag
    (i\epsilon \partial_t + \P_\epsilon)\U_\epsilon(t) u(x) &=
    \frac{1}{(2\pi)^d} \sum_{y\in \epsilon\Z^d} \int_{[-\pi,\pi]^d}
    ((i\epsilon\partial_t + \P_\epsilon) \circ r_\epsilon)
    \eta_{t,y,\xi;\epsilon} (x) e^{iy\xi/\epsilon} u(y) d\xi \\ \label{eq:iedPCont} &=
    \frac{1}{(2\pi)^d} \sum_{y\in \epsilon\Z^d} \int_{[-\pi,\pi]^d}
    \left( r_\epsilon \circ \left( i\epsilon\partial_t +
    \Op_{\epsilon,1/2} a \right) \right) \eta_{t,y,\xi;\epsilon} (x)
    e^{iy\xi/\epsilon} u(y) d\xi
  \end{align}
  for $u \in \l^2(\epsilon\Z^d),\, x \in \epsilon\Z^d$. We remark that
  we shall construct $\mu$ such that $\mu(t,\cdot,\cdot,\xi;\epsilon)$
  has compact support. Thus $\Op_{\epsilon, 1/2} a$ can actually be
  applied to $ \eta_{t,y,\xi;\epsilon}$. Let
  \begin{align}
    \notag \tilde{\eta}(t) &: (x,y,\xi,\epsilon) \mapsto
    e^{i\phi(t,x,\xi)/\epsilon}\left( i\epsilon\partial_t +
    \Op_{\epsilon,1/2} a \right) \eta_{t,y,\xi;\epsilon}
    (x)\\ \label{eq:tildeEtaT} &\phantom{: (x,y,\xi,\epsilon) \mapsto}
    = \left[ e^{i\phi(t,\cdot,\xi)/\epsilon} \left(
      i\epsilon\partial_t + \Op_{\epsilon,1/2} a \right)
      e^{-i\phi(t,\cdot,\xi)/\epsilon} \mu(t, \cdot, y,\xi;\epsilon)
      \right](x)
  \end{align}
  for $t \in (-T',T')$. We identify $\tilde{\eta}(\cdot)$ with a
  function on $(-T',T')$.

  \begin{claim}
    \label{claim:AppropKernelFIO}
    A function $\mu(t,x,y,\xi;\epsilon)$, compactly supported with respect to $x$ and $y$ and periodic in $\xi$, 
    can be constructed such that
    for some $\tilde{T} \in (0,T')$ the function $\tilde{\eta}$
    satisfies
    \begin{align}
    \label{eq:tildeEtaC0SN}
    \tilde{\eta} \in \Cont^\infty \left( (-\tilde{T},\tilde{T}),
    \Sy^{n_1}(m_0^{-n_2},\epsilon_0)(\R^d \times \R^d \times \T^d)
    \right)
    \end{align}
    for $m_0(x,y,\xi) = \jap{(x,y)}$ and any $n_1, n_2 \in \N$.
  \end{claim}
  
  Assuming Claim \ref{claim:AppropKernelFIO}, which is proven below,
  the number $T$ in \eqref{eq:PseudoFIOApprox} is then chosen to be
  some $T \in (0,\tilde{T})$. We finally choose the kernel function
  for $\U_\epsilon(t)$ in Theorem \ref{theo:PseudoFourIntApprox} to be
  $\mu$ multiplied by a cut-off function being equal to $1$ on
  $(-T,T)$.  Using \eqref{eq:iedPCont} and choosing $n_2 > 2d + 1$ in
  \eqref{eq:tildeEtaC0SN}, we simultaneously see that the operator $(i\epsilon
  \partial_t + \P_\epsilon)\U_\epsilon(t)$ is of the form
  \eqref{eq:IntOpAeps} with kernel
  \begin{align}
    k_{t,\epsilon}(x,y) = \frac{1}{(2\pi)^d} \int_{\T^d} e^{i(y\xi - \phi(t,x,\xi))/\epsilon}
    \tilde{\eta}(t)(x,y,\xi,\epsilon) d\xi
  \end{align}
  and fulfils the condition \eqref{eq:TraClCondSpec} and therefore
  the trace class condition \eqref{eq:TraClCond} of Proposition
  \ref{prop:TraClCri}.  The trace norm estimate
  \eqref{eq:TraClEstimateGeneral} then yields
  \begin{align}
    \notag \trn{(i\epsilon \partial_t + \P_\epsilon)\U_\epsilon(t)}
    &\leq \sum_{y \in \epsilon\Z^d} \left( \sum_{x \in \epsilon\Z^d}
    \left| k_{t,\epsilon}(x,y) \right|^2 \right)^{1/2} \\ \notag &\leq
    \frac{1}{(2\pi)^d}\sum_{y \in \epsilon\Z^d} \left( \sum_{x \in
      \epsilon\Z^d} \left( \int_{\T^d} \left|
    \tilde{\eta}(t)(x,y,\xi,\epsilon) \right| d\xi \right)^2
    \right)^{1/2} \\ \notag &= \O(\epsilon^{n_1}) \sum_{y \in
      \epsilon\Z^d} \left( \sum_{x \in \epsilon\Z^d}
    \jap{(x,y)}^{-2n_2} \right)^{1/2} \\ \notag &= \O(\epsilon^{n_1})
    \sum_{y \in \Z^d} \left( \sum_{x \in \Z^d} \jap{(\epsilon
      x,\epsilon y)}^{-2n_2} \right)^{1/2}\\ \label{eq:applyTildeEta2}
    &= \O(\epsilon^{n_1 - n_2}) \sum_{y \in \Z^d} \left( \sum_{x \in
      \Z^d} \jap{(x,y)}^{-2n_2} \right)^{1/2}.
  \end{align}
  These estimates are uniform with respect to $t \in [-T,T]$ since
  $\tilde{\eta}$ is continuous on $(-\tilde{T},\tilde{T})$ and
  therefore bounded on $[-T,T]$. For any $N \in \N$ the lhs of
  \eqref{eq:applyTildeEta2} is of order $\O(\epsilon^N)$ since we may
  choose $n_1 = N + n_2$ in \eqref{eq:applyTildeEta2}.  By
  \eqref{eq:CondForFourierAppr} this completes the proof of Theorem
  \ref{theo:PseudoFourIntApprox}, modulo Claim
  \ref{claim:AppropKernelFIO}.
\end{proof}

In the proof of Claim \ref{claim:AppropKernelFIO} we shall derive
conditions, ultimately seen sufficient, on the kernel function $\mu$,
which will turn out to be transport equations for the coefficients of
an asymptotic expansion of $\mu$. Transport equations are well
studied, see for example \cite{foll95pde, DiSj, evans1998partial}. We
need to prepare a special case where the initial condition and the
inhomogeneity have compact support.

\begin{Lem}
  \label{lem:transpEq}
  Let $F: \R \times \R^{d} \rightarrow \R^d$ and $g: \R \times \R^d
  \rightarrow \C$ be smooth and $L > 0$. Then there is some
  $T > 0$ such that for any smooth functions $u_0 : \R^d \rightarrow
  \C$ and $I : \R \times \R^d \rightarrow \C$ where $u_0$ and
  $I(t,\cdot)$ have compact support in $(-L,L)^d$, uniformly for $t
  \in (-T,T)$, that is
  \begin{align}
    \label{eq:suppPropTransp}
     \overline{\bigcup_{t \in (-T,T)} \supp I(t,\cdot)} \subset
     (-L,L)^d,
  \end{align}
  the initial value problem
  \begin{align}
    \label{eq:transpLem}
    (\partial_t + F(t,x) \cdot \nabla_x + g(t,x))u(t,x) = I(t,x),
    \qquad u(0,\cdot) = u_0
  \end{align}
  has a solution $u \in \Cont^\infty((-T,T) \times \R^d)$ where
  $u(t,\cdot)$ has compact support in $(-L,L)^d$, uniformly for $t \in
  (-T,T)$.
\end{Lem}

\begin{proof}[Proof (sketch of the standard arguments)]
  We first consider the homogeneous problem
  \begin{align}
    \label{eq:transpLemLin}
    (\partial_t + F(t,x) \cdot \nabla_x + g(t,x))u_{hom}(t,x) = 0,
    \qquad u_{hom}(0,\cdot) = 1.
  \end{align}
  The initial condition in \eqref{eq:transpLemLin} is given on the
  hypersurface $S = \left\{0\right\} \times \R^d$ which is
  non-characteristic for \eqref{eq:transpLemLin}. Since $S$ is a
  $\Cont^\infty$-hypersurface and all coefficients in
  \eqref{eq:transpLemLin} are smooth, \eqref{eq:transpLemLin} has a
  smooth solution $u_{hom}$ on some sufficiently small neighbourhood
  $\Omega$ of $S$. A similar statement for real-valued
  $\Cont^1$-coefficients is proven, for example, in
  \cite{foll95pde}. The solution $u_{hom}$ is constructed by solving
  the differential equation \eqref{eq:transpLemLin} along all integral
  curves $\tau \mapsto \gamma_{x_0}(\tau)$ of the transport vector
  field $(1,F(t,x))$ with $\gamma_{x_0}(0) = (0,x_0)$, each passing
  through precisely one point $(0,x_0)$ of the hypersurface $S$. By
  variation of constants, a solution $u$ of \eqref{eq:transpLem} on
  $\Omega$ can then be constructed
  from
  \begin{align}
    u(\gamma_{x_0}(\tau)) = \left( u_0(x_0) + \int_0^\tau
    \frac{I(\gamma_{x_0}(\tilde{\tau}))}{u_{hom}(\gamma_{x_0}(\tilde{\tau}))}
    d\tilde{\tau} \right) u_{hom}(\gamma_{x_0}(\tau)).
  \end{align}
  By compactness and smoothness of the local flow, this defines a
  solution $u \in \Cont^\infty((-T,T) \times \R^d)$ for some $T > 0$
  with the claimed properties.
\end{proof}

\begin{proof}[Sketch of the proof of Claim \ref{claim:AppropKernelFIO}]
  Note that in the following we shall make the Ansatz that $\mu$ has an
  asymptotic expansion, resulting in conditions on its
  coefficients. We will however leave it to the reader to check that a
  function $\mu$ constructed in accordance with these conditions will
  actually fulfil the statement in Claim \ref{claim:AppropKernelFIO}.

  We will keep the notation of the proof of Theorem
  \ref{theo:PseudoFourIntApprox} where the Claim
  \ref{claim:AppropKernelFIO} was stated.
  
  By construction and due to Lemma \ref{lem:HamJacPeriod} the solution
  $\phi$ of the Hamilton-Jacobi equation \eqref{eq:HamJacEq} satisfies
  the assumptions on the phase function $\phi$ in Proposition
  \ref{prop:OpUnitConj} (see the text after
  \eqref{eq:CondForFourierAppr} where $\phi$ was introduced as
  solution on $(-T',T') \times (-L,L)^d \times \R^d$). Thus
  Proposition \ref{prop:OpUnitConj} may be applied with $q = a$ to
  define a family $(\tilde{a}_{t,\xi})_{t \in (-T',T'), \xi \in \R^d}$
  of symbols $\tilde{a}_{t,\xi} \in \Sy^0(m,\epsilon_0)(\R^d \times
  \T^d)$, which are $2\pi\Z^d$-periodic with respect to the parameter
  $\xi$ and satisfy 
     \begin{align}
      e^{i\phi(t,\cdot,\xi)/\epsilon} \Op_{\epsilon,1/2} a
      e^{-i\phi(t,\cdot,\xi)/\epsilon} = \Op_{\epsilon,1/2}
      \tilde{a}_{t,\xi} \qquad (t \in (-T,T),\, \xi \in \R^d).
    \end{align} 
  Inserting this into the last line of \eqref{eq:tildeEtaT} and
  applying the chain and product rule, we
  formally obtain
  \begin{align}
        \label{eq:etaTilde}
        \tilde{\eta}(t)(x,y,\xi,\epsilon) = \left( (\partial_t
        \phi)(t,x,\xi) + i\epsilon\partial_t + \Op_{\epsilon,1/2}
        \tilde{a}_{t,\xi} \right) \mu(t,x,y,\xi;\epsilon).
  \end{align}
  Due to Proposition \ref{prop:OpUnitConj} we know that
  $\tilde{a}_{t,\xi}$ can be asymptotically expanded, uniformly in $t$
  and $\xi$, i.e. we can write
  \begin{align}
    \label{eq:tildea_asymp}
    \tilde{a}_{t,\xi}(x,\eta;\epsilon) \sim \sum_{j=0}^\infty \epsilon^j \tilde{a}_{t,\xi,j}(x,\eta) 
  \end{align}
  with $\epsilon$-independent symbols $\tilde{a}_{t,\xi,k} \in
  \Sy^0(m,\epsilon_0)(\R^d \times \T^d)$ which fulfil the symbol class
  property uniformly in $t$ and $\xi$ as specified in
  \eqref{eq:unifSymb} and \eqref{eq:unifAsympExp}. In addition, by
  \eqref{eq:LeadOrderTerm} the leading order term is given by
  \begin{align}
    \label{eq:tildea_lead}
    \tilde{a}_{t,\xi,0} (x,\eta) = a_0(x,\eta + \nabla_x
    \phi(t,x,\xi)).
  \end{align}
  Each coefficient in \eqref{eq:tildea_asymp} in turn can be formally
  represented by its Taylor expansion with respect to $\eta$,
  \begin{align}
     \label{eq:tildea_Taylor}
     \tilde{a}_{t,\xi,k}(x,\eta) = \sum_{\alpha \in \N^d}
     \frac{1}{\alpha!}  (\partial_\eta^\alpha
     \tilde{a}_{t,\xi,k})(x,0) \eta^\alpha.
  \end{align}
  Combining the expansions \eqref{eq:tildea_asymp} and
  \eqref{eq:tildea_Taylor}, we then formally get
  \begin{align}
        \label{eq:OpExpandedTaylor}
        \Op_{\epsilon,1/2} \tilde{a}_{t,\xi} = \sum_{k=0}^\infty
        \sum_{\alpha \in \N^d} \frac{\epsilon^k}{\alpha!}
        \Op_{\epsilon,1/2} ((\partial_\eta^\alpha
        \tilde{a}_{t,\xi,k})(x,0) \eta^\alpha).
  \end{align}
  Using \eqref{eq:OpExpandedTaylor} and making the Ansatz
  \begin{align}
    \label{eq:muAsymp}
        \mu(t,x,y,\xi;\epsilon) \sim \sum_{j=0}^\infty \epsilon^j
        \mu_j(t,x,y,\xi),
      \end{align}
      we identify the coefficients of $\epsilon^k$ in
      \eqref{eq:etaTilde} and, in order to satisfy
      \eqref{eq:tildeEtaC0SN}, set them equal to zero for any $k \in \N$. The equation
      associated with $\epsilon^0$ is
      \begin{align}
        \label{eq:asymptConstr1}
        0 &= ((\partial_t\phi)(t,x,\xi) + \tilde{a}_{t,\xi,0}(x,0))
        \mu_0(t,x,y,\xi)
      \end{align}
      and the equations associated with $\epsilon^k$ for $k \in \N^*$
      can be identified and rearranged as
      \begin{align}
        \notag 0 &= (\partial_t \phi)(t,x,\xi)
        \epsilon^k\mu_k(t,x,y,\xi) + i\epsilon^k
        (\partial_t\mu_{k-1})(t,x,y,\xi) \\ \notag &\phantom{=} +
        \sum_{l=0}^k \sum_{j=0}^{k-l} \epsilon^j \sum_{|\alpha|=k-l-j}
        \frac{1}{\alpha!} \Op_{\epsilon,1/2} (\partial_\eta^\alpha
        \tilde{a}_{t,\xi,j}(x,0) \eta^\alpha)
        \epsilon^l\mu_l(t,x,y,\xi) \\ \notag &= \epsilon^k \left(
        (\partial_t \phi)(t,x,\xi) + \tilde{a}_{t,\xi,0}(x,0) \right)
        \mu_k(t,x,y,\xi) + i\epsilon^k
        (\partial_t\mu_{k-1})(t,x,y,\xi) \\ \notag &\phantom{=} +
        \tilde{a}_{t,\xi,1}(x,0) \epsilon^{k}\mu_{k-1}(t,x,y,\xi) +
        \sum_{|\alpha|=1} \Op_{\epsilon,1/2} (\partial_\eta^\alpha
        \tilde{a}_{t,\xi,0}(x,0) \eta^\alpha)
        \epsilon^{k-1}\mu_{k-1}(t,x,y,\xi) \\ \label{eq:asymptConstr2}
        &\phantom{=} + \sum_{l=0}^{k-2} \sum_{j=0}^{k-l} \epsilon^j
        \sum_{|\alpha|=k-l-j} \frac{1}{\alpha!} \Op_{\epsilon,1/2}
        (\partial_\eta^\alpha \tilde{a}_{t,\xi,j}(x,0) \eta^\alpha)
        \epsilon^l\mu_l(t,x,y,\xi), \tag{$\theequation_k$}
      \end{align}
      where the last sum is understood to equal $0$ for $k = 1$.

      We recall that for a symbol $p \in \Sy^0(m,\epsilon_0)(\R^d
     \times \R^d)$ of the form $p(x,\eta) = g(x)\eta_j$ we have
     $\Op_{\epsilon,1/2}p = i\epsilon \left(g \partial_j +
     \frac{(\partial_jg)}{2} \right)$. This means that the sum in
     \eqref{eq:asymptConstr2} running through $|\alpha | = 1$ can be
     written as
     \begin{align}
        \notag \sum_{|\alpha|=1} \Op_{\epsilon,1/2}
        (\partial_\eta^\alpha \tilde{a}_{t,\xi,0}(x,0) \eta^\alpha) &=
        \sum_{|\alpha|=1} i\epsilon \left( \partial_\eta^\alpha
        \tilde{a}_{t,\xi,0}(x,0) \partial_x^\alpha + \frac{1}{2}
        \partial_x^\alpha \left( \partial_\eta^\alpha
        \tilde{a}_{t,\xi,0}(x,0) \right) \right) \\ &= i\epsilon
        \left( F(t,x,\xi) \cdot \nabla_x + \frac{1}{2} \div_x
        F(t,x,\xi) \right)
     \end{align}
     where
     \begin{align}
       F(t,x,\xi) := (\nabla_\eta \tilde{a}_{t,\xi,0})(x,0) \qquad (t
       \in (-T',T'),\, x,\xi \in \R^d).
     \end{align}
     Writing $\tilde{a}_{t,\xi,0}$ in terms of $a_0$ (see
     \eqref{eq:tildea_lead}), we see in particular that $F$ is the
     Hamiltonian vector field of $a_0$ projected onto position space,
     i.e.
     \begin{align}
       F(t,x,\xi) &= (\nabla_\xi a_0)(x,\nabla_x
       \phi(t,x,\xi)),
     \end{align}
     where $\nabla_\xi a_0$ denotes the derivative of $a_0$ with
     respect to its momentum variable. Thus, combined with
     \eqref{eq:tildea_lead}, the system \eqref{eq:asymptConstr1} and
     \eqref{eq:asymptConstr2} is equivalent to
     \begin{align}
       \label{eq:transp_eq1}
       0 &= ((\partial_t\phi)(t,x,\xi) + a_0(x,\nabla_x
       \phi(t,x,\xi))) \mu_0(t,x,y,\xi), \\ \notag 0 &= \left(
       (\partial_t \phi)(t,x,\xi) + a_0(x,\nabla_x \phi(t,x,\xi))
       \right) \mu_k(t,x,y,\xi) \\ \notag &\phantom{=} + \left(
       i\partial_t + iF (t,x,\xi) \cdot \nabla_x + \frac{i}{2} \div_x
       F(t,x,\xi) + \tilde{a}_{t,\xi,1}(x,0) \right)
       \mu_{k-1}(t,x,y,\xi)
       \\ \tag{$\theequation_k$}\label{eq:transp_eq2} &\phantom{=} +
       I_k(\mu_0, \dots, \mu_{k-2})(t,x,y,\xi)
     \end{align}
     with functions $I_k(\mu_0, \dots, \mu_{k-2})$ not depending on
     $\mu_{k-1}$ and $I_1 = 0$.

     We shall define the functions $\mu_k$ inductively and then see
     that \eqref{eq:transp_eq1} and all \eqref{eq:transp_eq2} are
     satisfied in a neighbourhood of $t = 0$ and any $x,y,\xi \in
     \R^d$. Note that \eqref{eq:FIOKernelInit} combined with the
     asymptotic expansion for the symbol $c$ given in Theorem
     \ref{theo:FuncCalcNonDiscr} (\ref{item:FuncCalcNonDiscr3})
     provides initial conditions given by
     \begin{align}
       \label{eq:TranspAppMuIni}
       \mu_k(0,x,y,\xi) &= \chi(x)\chi(y) c_k((x+y)/2,\xi).
     \end{align}
     In particular, the supports of $\mu_k(0,\cdot,y,\xi)$ are
     contained in $(-L,L)^d$, uniformly in $y,\xi$.

     Note that, since $\phi$ fulfils the Hamilton-Jacobi equation in
     the domain $(-T',T') \times (-L,L)^d \times \R^d \times \R^d$,
     after dividing by $i$ the equations \eqref{eq:transp_eq2} will
     take the form
     \begin{align}
       \notag 0 &= \left( \partial_t + F (t,x,\xi) \cdot \nabla_x +
       \frac{1}{2} \div_x F(t,x,\xi) - i \tilde{a}_{t,\xi,1}(x,0)
       \right)
       \mu_{k-1}(t,x,y,\xi)\\ \stepcounter{equation}\tag{$\theequation_{k}$} \label{eq:transp_eqV3}
       &\phantom{=} - i I_{k}(\mu_0, \dots, \mu_{k-2})(t,x,y,\xi)
     \end{align}
     in this domain. These equations are transport equations (where
     $y$ and $\xi$ act as parameters) which are treated in Lemma
     \ref{lem:transpEq} (with $g(t,x) = \frac{1}{2} \div_x F(t,x,\xi)
     - i \tilde{a}_{t,\xi,1}(x,0)$ in \eqref{eq:transpLem}) and which
     can be solved inductively. Note that $F$ and $g$ do not depend on
     the equation number $k$, so due to Lemma \ref{lem:transpEq} we
     find some $\tilde{T} \in (0,T')$ and solutions $\mu_k \in
     \Cont^\infty((-\tilde{T},\tilde{T}) \times \R^d)$ where each
     $\mu_k(t,\cdot,y,\xi)$ has compact support in $(-L,L)^d$. 
     
     The
     applicability of Lemma \ref{lem:transpEq} can be checked
     inductively: For $k=1$ we have that $\mu_0(0,\cdot,y,\xi)$ has
     compact support in $(-L,L)^d$ and $I_1 = 0$. Then
     $\mu_0(t,\cdot,y,\xi)$ will have compact support in $(-L,L)^d$,
     uniformly for $t \in (-\tilde{T},\tilde{T})$. Assumed that $k \in
     \N$ is chosen such that $\mu_0(t,\cdot,y,\xi)$,
     $\mu_1(t,\cdot,y,\xi)$, \dots, $\mu_{k-1}(t,\cdot,y,\xi)$ have
     compact support in $(-L,L)^d$, uniformly for $t \in
     (-\tilde{T},\tilde{T})$, then the inhomogeneity
     $I_{k+1}(t,\cdot,y,\xi)$ of the $(k+1)$-th equation will also
     have compact support in $(-L,L)^d$. So, applying Lemma
     \ref{lem:transpEq} to the $(k+1)$-th equation, the compact
     support of $\mu_{k}(t,\cdot,y,\xi)$ will be contained in
     $(-L,L)^d$ as well, uniformly for $t \in (-\tilde{T},\tilde{T})$.

     Note that the solutions $\mu_k$ of \eqref{eq:TranspAppMuIni} and
     \eqref{eq:transp_eqV3} satisfy the system \eqref{eq:transp_eq1},
     \eqref{eq:transp_eq2} in the domain $(-\tilde{T},\tilde{T})
     \times (-L,L)^d \times \R^d \times \R^d$ since $\phi$ fulfils
     the Hamilton-Jacobi equation there. In the domain
     $(-\tilde{T},\tilde{T}) \times (\R^d \setminus (-L,L)^d) \times
     \R^d \times \R^d$, the functions $\mu_k$ trivially fulfil the
     system \eqref{eq:transp_eq1}, \eqref{eq:transp_eq2} since all
     contributions vanish due to the support property. Furthermore,
     all $\mu_k$ are periodic in $\xi$ by construction.

     Given the solutions $\mu_k$ of \eqref{eq:TranspAppMuIni},
     \eqref{eq:transp_eqV3}, we define $\mu$ via Borel summation in
     \eqref{eq:muAsymp}. We now leave it to the reader to verify that
     the function $\tilde{\eta}$ defined by \eqref{eq:tildeEtaT}
     actually fulfils \eqref{eq:tildeEtaC0SN} (essentially by
     reversing the arguments given above). Using that $\mu$ has
     compact support with respect to $x$, one first verifies
     $\tilde{\eta}(t) \in \Sy^{n_1}(m_0^{-n_2},\epsilon_0)(\R^d \times
     \R^d \times \R^d)$ for any $t \in (-\tilde{T},\tilde{T})$ and any
     $n_1, n_2 \in \N$. Furthermore, it can be seen from the
     representation \eqref{eq:etaTilde} that $\tilde{\eta}(t)$ is
     periodic with respect to $\xi$, since $\partial_t \phi$ is due to
     Lemma \ref{lem:HamJacPeriod}, the family of symbols
     $\tilde{a}_{t,\xi}$ is due to Proposition \ref{prop:OpUnitConj}
     and $\mu$ is by construction. It then follows from our uniformity
     statements with respect to $t,y,\xi$ that $t \mapsto
     \tilde{\eta}(t)$ is even smooth.
\end{proof}

\subsection{Proof of Proposition \ref{prop:f1f3asymp}}
\label{subsec:MainProof3}

We shall prove Proposition \ref{prop:f1f3asymp} by following the
strategy given in \cite[chapter 10]{DiSj}. We assume Hypothesis
\ref{hyp:WRef} and consider the functions $f_1$ and $f_3$ from
\eqref{eq:sumf1f2f3}, which have support in a neighbourhood of
$\alpha$ and $\beta$, respectively, only consisting of non-critical
values of $a_0$. Proposition \ref{prop:f1f3asymp} then follows from 
identifying the leading order terms of both the lhs and the rhs of
\eqref{eq:f1asymp} and \eqref{eq:f3asymp} with
\begin{align}
  \label{eq:StepToWeyl}
  \int_\alpha^\infty I_1(\lambda;f_1, \psi, \epsilon) d\lambda \quad
  \mbox{ and } \quad \int_{-\infty}^\beta I_1(\lambda;f_3, \psi,
  \epsilon) d\lambda, \quad \mbox{ respectively.}
\end{align}

Here, for $\lambda \in \R$ and a function $f \in \Cont_0^\infty(J)$
compactly supported in the set of non-critical values of $a_0$,
\begin{align}
  \label{eq:I1Def}
  I_1(\lambda;f, \psi, \epsilon) &:= \frac{1}{2\pi \epsilon} \tr
  \left( \int_\R e^{it(\P_\epsilon - \lambda)/\epsilon} \psi(t)
  f(\P_\epsilon) dt \right)
\end{align}
 with a cut-off function $\psi \in \Cont_0^\infty(\R)$ with $\psi(0) =
 1$ and supported in a small neighbourhood of $0$. Note that
 $f(\P_\epsilon)$ is of trace class for small $\epsilon$
 (Proposition \ref{prop:fPTraceClass}). Therefore the operator in
 \eqref{eq:I1Def} to which the trace is applied is of trace class
 since it is the limit of trace class operators with respect to the
 trace norm. So $I_1$ is well-defined.

In order to identify the rhs of \eqref{eq:f1asymp} and
\eqref{eq:f3asymp} in leading order with the expressions in
\eqref{eq:StepToWeyl}, we need two preparatory steps. Using Theorem
\ref{theo:PseudoFourIntApprox} we first show in Lemma
\ref{lem:I1I2eps} that $I_1(\lambda;f, \psi,\epsilon)$ can be
approximated sufficiently precisely by
\begin{align}
  \label{eq:I2Def}
  I_2(\lambda;f, \psi, \epsilon) &:= \int_\R \int_{\T^d}
  \int_{\R^d} g_{f,\lambda}(t,x,\xi;\epsilon) dxd\xi dt.
\end{align}
Here we set
\begin{align}
   \label{eq:gLambda}
  g_{f,\lambda}(t,x,\xi;\epsilon) &:= \frac{1}{(2\pi\epsilon)^{d+1}}
  e^{i\varphi_\lambda(t,x,\xi)/\epsilon} \psi(t)
  \mu_f(t,x,x,\xi;\epsilon),\\
   \label{eq:PhiLambda}
  \varphi_\lambda(t,x,\xi) &:= x\xi - \phi(t,x,\xi) - \lambda t = - \phi_{\T}(t,x,\xi) - \lambda t, 
\end{align}
where $\phi$ denotes the solution of the Hamilton-Jacobi equation
\eqref{eq:HamJacEq} with $H = a_0$, satisfying the periodicity property in
Lemma \ref{lem:HamJacPeriod} and $\mu_f$ denotes the kernel function
corresponding to the operators $(\U_\epsilon^{(f)}(t))_{t \in (-T,T)}$
for some $T>0$ as constructed in Theorem
\ref{theo:PseudoFourIntApprox}.

Note that $g_{f,\lambda}$ depends on
$f$ since $\mu_f$ does and is $2\pi$-periodic with respect
to $\xi$ since $\phi_\T$ and $\mu_f$ are (by construction and assumption). So the integral in \eqref{eq:I2Def} is
well-defined. Applying the method of stationary phase to
$I_2(\lambda;f, \psi,\epsilon)$, we shall then identify the integral of the resulting leading order term as the 
principal term on the rhs
of \eqref{eq:f1asymp} for $f = f_1$ and \eqref{eq:f3asymp} for $f =
f_3$, respectively.

\begin{Lem}
  \label{lem:I1I2eps}
  Assume Hypothesis \ref{hyp:WRef}. For $f \in \Cont_0^\infty(J)$
  compactly supported in the set of non-critical values of $a_0$ and
  any $\psi \in \Cont_0^\infty (\R)$ supported in a sufficiently small
  neighbourhood of $t=0$, we have
  \begin{align}
    I_1(\lambda;f, \psi, \epsilon) = I_2(\lambda; f, \psi, \epsilon) +
    \O(\epsilon^\infty) \qquad (\epsilon \downarrow 0),
  \end{align}
  uniformly in $\lambda \in \R$.
\end{Lem}

\begin{proof}
  Fixing $f$, we will shorten the
  notation by writing
  \begin{align}
    I_1(\lambda;\psi,\epsilon) &:= I_1(\lambda;f,\psi,\epsilon), &
    I_2(\lambda;\psi,\epsilon) &:= I_2(\lambda;f,\psi,\epsilon), &
    \U_\epsilon &:= \U_\epsilon^{(f)}, & \mu := \mu_f.
  \end{align}

  We define
  \begin{align}
    I_3(\lambda;\psi, \epsilon) &:= \frac{1}{2\pi \epsilon} \tr \left(
    \int_\R e^{-it\lambda/\epsilon} \psi(t) \U_\epsilon(t) dt \right).
  \end{align}
  From \eqref{eq:PseudoFIOApprox} we then conclude for $\psi \in
  \Cont_0^\infty((-T,T))$
  \begin{align}
    \notag \sup_\lambda |I_1(\lambda;\psi,\epsilon) -
    I_3(\lambda;\psi,\epsilon)| &\leq \frac{1}{2\pi\epsilon}
    \sup_\lambda \trn{\int_\R e^{-it\lambda/\epsilon} \psi(t) \left(
      e^{it\P_\epsilon/\epsilon} f(\P_\epsilon) - \U_\epsilon(t)
      \right) dt} \\ \notag &\leq \frac{1}{2\pi\epsilon}
    \left(\sup_{|t| < T} \trn{e^{it\P_\epsilon/\epsilon}
      f(\P_\epsilon) - \U_\epsilon(t)}\right) \int_\R |\psi(t)| dt
    \\ &= \O(\epsilon^\infty),
  \end{align}
  uniformly in $\lambda \in \R$.

  Applying the trace formula \eqref{eq:TraClEstimateGeneral} to $I_3$,
  we have
  \begin{align}
    \label{eq:I3trace}
    I_3(\lambda;\psi, \epsilon) =  \frac{1}{2\pi\epsilon} \frac{1}{(2\pi)^d} \int_\R
    \int_{\T^d} \sum_{x \in \epsilon\Z^d}
    e^{i\varphi_\lambda(t,x,\xi)/\epsilon} \psi(t)
    \mu(t,x,x,\xi;\epsilon) d\xi dt.
  \end{align}
  We apply Proposition \ref{prop:PoisApp} to transform the sum in
  \eqref{eq:I3trace} for appropriate $\psi$ into an integral to get
  \begin{align}
    I_3(\lambda;\psi,\epsilon) = I_2(\lambda;\psi,\epsilon) + \O(\epsilon^\infty),
  \end{align}
  uniformly in $\lambda \in \R$, which concludes the proof.  We
  briefly specify how Proposition \ref{prop:PoisApp} is applied. For
  this purpose, we write
  \begin{align}
    I_3(\lambda;\psi,\epsilon) =
    \frac{1}{2\pi\epsilon}\frac{1}{(2\pi)^d} \int_\R
    \int_{\T^d} e^{-i\lambda t/\epsilon} \sum_{x \in \epsilon\Z^d}
    e^{i\tilde{\varphi}_{t,\xi}(x)/\epsilon}
    \tilde{a}_{t,\xi}(x;\epsilon) d\xi dt
  \end{align}
  where 
  \begin{align}
    \tilde{\varphi}_{t,\xi}(x) &:= x\xi - \phi(t,x,\xi) = - \phi_\T(t,x,\xi),
    \\ \tilde{a}_{t,\xi}(x;\epsilon) &:= \psi(t)
    \mu(t,x,x,\xi;\epsilon).
  \end{align}
  By construction both, the symbol $\tilde{a}_{t,\xi}$ has support in some
  compact $\tilde{K} \subset \R^d$, uniformly for $t \in (-T,T)$ and
  $\xi \in [-\pi,\pi]^d$. Since $\phi(0,x,\xi) = x\xi$, we can find
  some $\tilde{T} \in (0,T)$ such that
  \begin{align}
    \label{eq:IntSumTrafoAppCond}
    \sup_{\substack{t \in (-\tilde{T},\tilde{T}) \\ \xi \in
        [-\pi,\pi]^d}}\sup_{\substack{j \in \{1,\dots,d\} \\ x \in
        \tilde{K}}} |\partial_{x_j} \tilde{\varphi}_{t,\xi}(x)| <
    2\pi.
  \end{align}
  Therefore for any $t \in (-\tilde{T},\tilde{T})$ and $\xi \in
  [-\pi,\pi]^d$ the condition \eqref{eq:PoisAppCond} is fulfilled for
  $K$ and $\phi$ in \eqref{eq:PoisAppCond} chosen as $K = \tilde{K}$
  and $\phi = \tilde{\varphi}_{t,\xi}$. We now choose
  $\psi$ to have compact support in $(-\tilde{T},\tilde{T})$. By
  Proposition \ref{prop:PoisApp} with $a$ in \eqref{eq:PoisApp} and
  \eqref{eq:PoisAppBound} chosen as $\tilde{a}_{t,\xi}$, we then have
  the error estimate
  \begin{align}
    \label{eq:IntSumTrafoApp}
    \abs{\epsilon^d \sum_{x \in \epsilon\Z^d}
      e^{i\tilde{\varphi}_{t,\xi}(x)/\epsilon}
      \tilde{a}_{t,\xi}(x;\epsilon) -
      \int_{\R^d}e^{i\tilde{\varphi}_{t,\xi}(x)/\epsilon}
      \tilde{a}_{t,\xi}(x;\epsilon) dx } \leq \epsilon^{2k}
    \sum_{\tilde{\xi} \in 2\pi \Z^d \setminus \left\{0\right\}}
    \int_{\tilde{K}} \abs{ (\mathbf{W}_\epsilon^k
      \tilde{a}_{t,\xi})(x,\tilde{\xi})} dx
  \end{align}
  for any $k \geq d$, $t \in \R$ and $\xi \in [-\pi,\pi]$ with the
  operator $\mathbf{W}_\epsilon$, here depending on the parameters $t$
  and $\xi$, defined by \eqref{eq:PoisAppDefW}. It follows now from
  the estimate \eqref{eq:IntSumTrafoAppCond} and the explicit formula
  \eqref{eq:PoisAppDefW} that the rhs. of \eqref{eq:IntSumTrafoApp} is
  of order $\O(\epsilon^{2k})$, uniformly for $t \in
  (-\tilde{T},\tilde{T}), \xi \in [-\pi,\pi]^d$. Therefore the
  difference $I_3(\lambda;\psi,\epsilon) - I_2(\lambda;\psi,\epsilon)$
  is of order $\O(\epsilon^\infty)$, uniformly for $\lambda
  \in \R$.
\end{proof}

In the following, we seek to apply the method of stationary phase to the integral $I_2$
defined in \eqref{eq:I2Def}. For this purpose we first study the critical
points of $\varphi_\lambda$ in Lemma \ref{lem:CritManif}. 

\begin{Lem}
  \label{lem:CritManif}
  Let $\phi$ be the solution of the Hamilton-Jacobi equation
  \eqref{eq:HamJacEq} for some Hamiltonian $H \in \Cont^\infty(\R^d
  \times \T^d)$ in the domain $(-T,T) \times (-L,L)^d \times
  \R^d$. For $\lambda \in \R$ denote by $\mathcal{M}_\lambda \subset
  \R_t \times \R_x^d \times \T_\xi^d$ the set of critical points of
  the function $\varphi_\lambda$ defined in \eqref{eq:PhiLambda}. Then
  for any $U \subset \R$ consisting of non-critical values
  of $H$ and with $H^{-1}(U)$ being compact in $\R^d \times \T^d$,
  there is some $t_0 > 0$ such that for any $\lambda \in U$
  \begin{align}
    \label{eq:critSet_t0}
  \mathcal{M}_\lambda \cap ((-t_0, t_0) \times \R^d \times \T^d) =
  \left\{ (0,x,\xi)\; |\; H(x,\xi) = \lambda \right\}.
  \end{align}
\end{Lem}

\begin{proof}
We have $(t,x,\xi) \in \mathcal{M}_\lambda$ if and only if
\begin{align}
  \label{eq:critPoint1}
  0 &= \partial_t \varphi_\lambda(t,x,\xi) = - \partial_t\phi(t,x,\xi)
  - \lambda = H(x,\xi) - \lambda\\
  \label{eq:critPoint2}
  0 &= \nabla_x \varphi_\lambda(t,x,\xi) = \xi - \nabla_x \phi(t,x,\xi) \\
  \label{eq:critPoint3}
  0 &= \nabla_\xi \varphi_\lambda(t,x,\xi) = x - \nabla_\xi \phi(t,x,\xi),
\end{align}
where in \eqref{eq:critPoint1} we used the fact that $\phi$ fulfils
the Hamilton-Jacobi equation \eqref{eq:HamJacEq} combined with
\eqref{eq:critPoint2}. Due to the initial condition in
\eqref{eq:HamJacEq} these equations are satisfied if $t=0$ and
$H(x,\xi) = \lambda$. This proves ``$\supset$'' in
\eqref{eq:critSet_t0}.

We shall prove ``$\subset$'' in \eqref{eq:critSet_t0} by
contradiction. Assume that there is a sequence $\lambda_n$ in $U$ and
a sequence of points $(t_n,x_n,\xi_n) \in \mathcal{M}_{\lambda_n}$
where $0 \neq t_n \rightarrow 0$ for $n \rightarrow \infty$. Since due
to \eqref{eq:critPoint1} the points $(x_n, \xi_n)$ are in the set
$H^{-1}(U)$, which by assumption is compact, we may suppose the
sequence $(x_n, \xi_n)$ to converge to some $(x^*,\xi^*) \in
H^{-1}(U)$.  Using Taylor approximation for $\nabla_{(x,\xi)}\phi$
with respect to the time variable at $t = 0$, we have
\begin{align}
  \label{eq:TaylGradPhi}
  (\nabla_{(x,\xi)}\phi)(t_n,x_n,\xi_n) =
  (\nabla_{(x,\xi)}\phi)(0,x_n,\xi_n) + t_n (\partial_t
  \nabla_{(x,\xi)} \phi)(0,x_n,\xi_n) + \mathcal{O}(t_n^2),
\end{align}
where the remainder $\mathcal{O}(t_n^2)$ is uniform in $x_n$ and
$\xi_n$.  Here, since $(0,x_n,\xi_n)$ and $(t_n,x_n,\xi_n)$ are
critical points of $\varphi_{\lambda_n}$, we have due to
\eqref{eq:critPoint2} and \eqref{eq:critPoint3}
\begin{align}
  \label{eq:TaylGradPhiLead}
  (\nabla_{(x,\xi)}\phi)(t_n,x_n,\xi_n) =
  (\nabla_{(x,\xi)}\phi)(0,x_n,\xi_n).
\end{align}
Since $\phi$ fulfils the Hamilton-Jacobi equation
\eqref{eq:HamJacEq}, the first order coefficient in the expansion
\eqref{eq:TaylGradPhi} is of the form
\begin{align}
  \label{eq:dtCritPoint}
  (\partial_t \nabla_{(x,\xi)} \phi)(0,x_n,\xi_n) = - \nabla_{(x,\xi)}
  \left( H(x_n, \nabla_x\phi(0,x_n,\xi_n) )\right) = -
  \nabla_{(x,\xi)} H (x_n,\xi_n).
\end{align}
Combining the equations \eqref{eq:TaylGradPhi},
\eqref{eq:TaylGradPhiLead} and \eqref{eq:dtCritPoint}, we get
\begin{align}
  \nabla_{(x,\xi)} H (x_n,\xi_n) = \O(t_n).
\end{align}
Taking the limit $n \rightarrow \infty$ here gives $\nabla_{(x,\xi)} H
(x^*,\xi^*) = 0$. But this contradicts $H(x^*,\xi^*) \in U$ since by
assumption $U$ contains only non-critical values of $H$.
\end{proof}

By Lemma \ref{lem:CritManif} the critical points of $\varphi_\lambda$
near $t=0$ form the $(2d-1)$-dimensional manifold $\{ 0 \} \times
H^{-1}(\lambda)$ if $\lambda$ is a non-critical value of the
underlying Hamiltonian. In particular, the critical points are
isolated with respect to $t$ but not with respect to $x$ and
$\xi$. Thus the stationary phase argument, which typically presupposes
a critical point of the phase function to be isolated, is not directly
applicable. We shall therefore follow the strategy in \cite[Chapter 10]{DiSj} and use Fubini's theorem 
to write $I_2$ from
\eqref{eq:I2Def} as an iterated integral where, using the Liouville
form as introduced in \eqref{eq:LiouvilleFormDef}, one integral is
taken over the level set $H^{-1}(\lambda)$ with $H = a_0$ and the
other one with respect to $dt$ and $dH$. On the domain of the
$dtdH$-integral, which is a $2$-dimensional manifold, the phase
function $\varphi_\lambda$ has an isolated critical point for $t$ near
$0$. So, for the $dtdH$-integral, the method of stationary phase
applies. Here $\lambda$ and the coordinates of $H^{-1}(\lambda)$ act
as parameters. We need to obtain remainder estimates that are uniform
with respect to these parameters. We shall therefore prepare a
parameter dependent version of the method of stationary phase in
Corollary \ref{cor:statPhase}. It is based on the local version given
in Theorem \ref{theo:statPhaseHoer}, which is taken from 
\cite[Theorem 7.7.5., Theorem 7.7.6.]{hoermander2015analysis}.

\begin{theo}
  \label{theo:statPhaseHoer}
  Let $(x,y) \mapsto \varphi(x,y)$ be a real-valued smooth function in
  a neighbourhood of $(x_0,y_0)$ in $\R^{n} \times \R^m$. Assume that
  $D_x\varphi (x_0,y_0) = 0$ and that $D^2_{x}\varphi(x_0,y_0)$ is
  non-singular with signature $\sigma$. Denote by $x(y)$ the solution
  of the equation $D_x\varphi(x,y) = 0$ with $x(y_0)=x_0$ given by the
  implicit function theorem near $y = y_0$. Then there exist
  differential operators $L_{j}$ of order $2j$ acting on
  $\Cont_0^\infty(\R^n \times \R^m)$ such that for some constant $C >
  0$, some compact neighbourhood $K$ of $(x_0,y_0)$ and any $u \in
  \Cont_0^\infty(K)$ (i.e. $u \in \Cont_0^\infty(\R^n \times \R^m)$
  with $\supp u \subset K$), $\omega > 0$ and $k \in \N^*$
  \begin{align}
    \label{eq:statPhaseExp}
    \left| \int_{\R^n} u(x,y) e^{i\omega\varphi(x,y)} dx -
    A_{y} \omega^{-n/2} e^{i\omega\varphi(x(y),y)} \sum_{j =
      0}^{k-1} \omega^{-j} L_{j} u(x(y),y) \right| \leq C
    \omega^{-k} \sum_{|\alpha| \leq 2k} \sup_x \left|
    \partial_x^\alpha u(x,y) \right|
  \end{align}
  where $A_{y} = \frac{(2\pi)^{n/2} e^{i\pi \sigma/4}}{|\det D^2_{x}
    \varphi(x(y),y)|^{1/2}}$. Here, $L_{j}u$ evaluated at $(x(y),y)$
  is given by
  \begin{align}
    \label{eq:statPhaseLj}
     \left. L_{j} u \right|_{(x(y),y)} = \sum_{\nu - \mu = j}
     \sum_{2\nu \geq 3\mu \geq 0} (2^\nu i^j \mu! \nu!)^{-1}
     \left. \skp{\left. D^2_x \varphi \right|_{(x(y),y)}^{-1}
       \nabla_x}{\nabla_x}^\nu (g^\mu u) \right|_{(x(y),y)}
  \end{align}
  where
  \begin{align}
    g(x,y) = \varphi(x,y) - \varphi(x(y),y) - \frac{1}{2} \skp{
      \left. D^2_x \varphi \right|_{(x(y),y)} (x - x(y))}{x - x(y)}.
  \end{align}
  In particular, $L_{0} u (x(y),y) = u(x(y),y)$ and the coefficients
  of $L_{1}$ at $(x(y),y)$ are rational functions in
  $\left. \partial^\alpha_x \varphi \right|_{(x(y),y)}$ for $\alpha \in \N^n$
  with $2 \leq |\alpha| \leq 4$, homogeneous of degree $-1$ and with
  denominator $\left(\det\left(\left. D^2_x\varphi
  \right|_{(x(y),y)}\right)\right)^3$.
\end{theo}

We note that the statement on the coefficients of $L_1$ in Theorem
\ref{theo:statPhaseHoer} can be verified by a straightforward
calculation using \eqref{eq:statPhaseLj}, Cramer's rule and the
observation that $\left. \partial^\alpha_x g \right|_{(x(y),y)} = 0$
for $|\alpha| \leq 2$.

The estimate \eqref{eq:statPhaseExp} is fulfilled by functions $u$
with support in some small compact neighbourhood of the point
$(x_0,y_0)$ with $x_0$ being a critical non-degenerate point of the
phase function $\varphi(\cdot,y_0)$. As explained before Theorem
\ref{theo:statPhaseHoer} we need to globalise this estimate with
respect to the parameter $y$. Using a standard compactness argument,
which combines nicely with local estimates, one can extend
\eqref{eq:statPhaseExp} to functions $u$ with support in some small
neighbourhood of a compact set of points $(x(y),y)$ where $y \mapsto
x(y)$ parametrises critical non-degenerate points of the family
$(\varphi(\cdot,y))_y$ of phase functions. In Corollary
\ref{cor:statPhase}, this is deduced from Theorem
\ref{theo:statPhaseHoer} for the special setting ($n = 2$, $k=2$)
needed for the proof of Proposition \ref{lem:I2Asymp}.

\begin{cor}
  \label{cor:statPhase}
  Let $y \mapsto x(y) \in \R^2$ be continuous on some open set $V
  \subset \R^m$ and let $U$ be a neighbourhood of $x(V)$. Let $\varphi
  : U \times V \rightarrow \R$ be a real-valued smooth function such
  that, for any $y \in V$, $\left. D_x\varphi \right|_{(x(y),y)} = 0$
  and $\left. D^2_{x}\varphi \right|_{(x(y),y)}$ is non-singular with
  constant signature $\sigma$. Let $\Omega \subset V$ be compact. Then
  there are some constant $C > 0$ and some compact neighbourhood $K$
  of $\{ (x(y),y) \; | \; y \in \Omega \}$ such that for any $u \in
  \Cont_0^\infty(K)$, $y \in \Omega$ and $\epsilon > 0$
  \begin{align}
    \label{eq:statPhaseSpecEst}
    \left| \int_{\R^2} u(x,y) e^{i\varphi(x,y)/\epsilon} dx - \epsilon
    A_{y} e^{i\varphi(x(y),y)/\epsilon} u(x(y),y) \right| \leq C
    \epsilon^{2} \sum_{|\alpha| \leq 4} \sup_x \left|
    \partial_x^\alpha u(x,y) \right|
  \end{align}
  where $A_{y} = \frac{2\pi e^{i\pi \sigma/4}}{|\det D^2_{x}
    \varphi(x(y),y)|^{1/2}}$.
\end{cor}

\begin{proof}
  Using Theorem \ref{theo:statPhaseHoer}, we shall first derive
  \eqref{eq:statPhaseSpecEst} for functions $u$ with support in some
  small compact neighbourhood $K(y_0)$ of $(x_0,y_0)$ with fixed $y_0
  \in V$ and $x_0 = x(y_0)$. By compactness we shall then cover the
  compact set $\{ (x(y),y) \; | \; y \in \Omega \}$ by finitely many
  sets of the family $(K(y))_{y \in V}$.
  
  Let $y_0 \in V$ and $x_0 = x(y_0)$. Since by assumption $(x_0, y_0)$
  is a critical non-degenerate point of $\varphi(\cdot,y_0)$ and
  $x(\cdot)$ is continuous, $x(\cdot)$ coincides with the solution
  $x(\cdot)$ in Theorem \ref{theo:statPhaseHoer} near $y = y_0$ by the
  uniqueness part of the implicit function theorem. So by Theorem
  \ref{theo:statPhaseHoer} applied with $n=2$, $k=2$ and $\omega =
  \epsilon^{-1}$, there are some constant $C'(y_0) > 0$ and some
  compact neighbourhood $K(y_0)$ of $(x(y_0),y_0)$ such that for any
  $u \in \Cont_0^\infty(K(y_0))$, $y \in V$ and $\epsilon > 0$
  \begin{align}
        \notag &\left| \int_{\R^2} u(x,y) e^{i\varphi(x,y)/\epsilon}
        dx - \epsilon A_{y} e^{i\varphi(x(y),y)/\epsilon} u(x(y),y)
        \right| \\ &\qquad \leq C'(y_0) \epsilon^2 \sum_{|\alpha| \leq 4}
        \sup_x \left| \partial_x^\alpha u(x,y) \right| + \epsilon^2
        \left| A_{y} e^{i\varphi(x(y),y)/\epsilon} L_{1} u(x(y),y)
        \right|.
  \end{align}
  Here, due to the smoothness of $\varphi$ and due to the form of
  $L_{1}$ given in Theorem \ref{theo:statPhaseHoer}, the last term can
  be bounded by the derivatives of $u$ of order of at most $4$,
  uniformly for $(x,y) \in K(y_0)$. Therefore for some constant
  $C(y_0) > C'(y_0)$ and any $u \in \Cont_0^\infty(K(y_0))$, $y \in V$
  and $\epsilon > 0$, we have
  \begin{align}
    \label{eq:statPhaseSpecEst2}
    \left| \int_{\R^2} u(x,y) e^{i\varphi(x,y)/\epsilon} dx - \epsilon
    A_{y} e^{i\varphi(x(y),y)/\epsilon} u(x(y),y) \right| \leq C(y_0)
    \epsilon^{2} \sum_{|\alpha| \leq 4} \sup_x \left|
    \partial_x^\alpha u(x,y) \right|.
  \end{align}

  By a compactness argument we shall now use the local estimates
  \eqref{eq:statPhaseSpecEst2} to gain the global estimate
  \eqref{eq:statPhaseSpecEst}. For any $y \in V$ let $U(y)$ be an open
  neighbourhood of $(x(y),y)$ whose closure is contained in the
  interior of $K(y)$. Since by assumption $\Omega$ is compact and
  $x(\cdot)$ is continuous, the set $\{ (x(y),y) \; | \; y \in \Omega
  \}$ is compact. So we can choose $y_1, \dots, y_n \in V$ such that
  $\{ (x(y),y) \; | \; y \in \Omega \} \subset \bigcup_i U(y_i)$. Let
  $\chi_i \in \Cont_0^\infty(K(y_i))$ with $\sum_i \chi_i = 1$ on $K
  := \overline{\bigcup_i U(y_i)}$ and let $u \in
  \Cont_0^\infty(K)$. Then $\chi_i u \in \Cont_0^\infty(K(y_i))$ and
  \eqref{eq:statPhaseSpecEst2} is fulfilled for $u$ chosen as $\chi_i
  u$ and with $y_0 = y_i$. Summing over $i$ then yields
  \eqref{eq:statPhaseSpecEst}.
\end{proof}

Having prepared this parameter dependent version of the method of
stationary phase, we come back to the analysis of the integral $I_2$
from \eqref{eq:I2Def}. We denote by $dvol$ the symplectic volume form
on $\R^d \times \T^d$. Given a differential form $\tau$, we will
denote by $|\tau|$ the associated density. The integral $I_2$ can then
be written as
\begin{align}
  \label{eq:I2volForm}
  I_2(\lambda;f,\psi,\epsilon) = \int_{\R \times \R^d \times \T^d}
  g_{f,\lambda} |dt \land dvol| \; ,
\end{align}
where the integral is well-defined without introducing an orientation
for the manifold $\R \times \R^d \times \T^d$. On the rhs of
\eqref{eq:I2volForm} we suppressed the dependence on $\epsilon$.

We further denote by $L$ the Liouville form with respect to the level
sets of $a_0$ and the symplectic volume form $dvol$. $L$ is
invariantly defined at any non-critical point of $a_0$ as the
contraction of $dvol$ with the vector field $\partial/\partial a_0$,
i.e.
\begin{align}
  \label{eq:LiouvilleFormDef}
  L = i_{\partial/\partial a_0} dvol = \frac{\partial}{\partial a_0} \contr dvol, 
\end{align}
where $i$ denotes the interior derivative and $\contr$ is the standard
symbol for the contraction. The Liouville form can be used to represent
the symplectic volume form as
\begin{align}
  \label{eq:SymplecLiou}
  dvol = da_0 \land L.
\end{align}
We shall now prove that, applying Corollary \ref{cor:statPhase} to the
$dtda_0$-integral with $\lambda$ and the coordinates of
$a_0^{-1}(\lambda)$ acting as parameters, the integral $I_2$ has the
expansion \eqref{eq:I2Asymp}.
\begin{prop}
  \label{lem:I2Asymp}
  Assume Hypothesis \ref{hyp:WRef}. Let $f \in \Cont_0^\infty(J')$
  where $J' \subset J$ is a compact interval only consisting of
  non-critical values of $a_0$. Then for any $\psi \in
  \Cont_0^\infty(\R)$ with $\psi(0) = 1$ and supported in a
  sufficiently small neighbourhood of $0$, the integral
  $I_2(\lambda;f,\psi,\epsilon)$ given in \eqref{eq:I2volForm} can be
  expanded as
  \begin{align}
    \label{eq:I2Asymp}
    I_2(\lambda; f, \psi,\epsilon) &= \frac{1}{(2\pi\epsilon)^d}
    \left( f(\lambda) \int_{a_0 = \lambda} |L| + \O(\epsilon) \right),
    \qquad \mbox{uniformly for } \lambda \in J'.
  \end{align}
\end{prop}

Note that, as already remarked above, the integral $\int_{a_0 =
  \lambda} |L|$ in \eqref{eq:I2Asymp} is well-defined as a positive
number without introducing an orientation for the level set
$a_0^{-1}(\lambda)$.

\begin{proof}
  Fixing $f$, we will shorten the
  notation by writing
  \begin{align}
    I_2(\lambda;\psi,\epsilon) &:= I_2(\lambda;f,\psi,\epsilon), &
    g_{\lambda} &:= g_{f,\lambda} , & \mu := \mu_f.
  \end{align}
  with $g_{f,\lambda}$, $\mu_f$ and $T > 0$ introduced below
  \eqref{eq:I2Def}.

  Let $J''$ be a neighbourhood of $J'$ such that the closure of $J''$
  is contained in $J$ and consists only of non-critical values of
  $a_0$. By the regular value theorem and Hypothesis \ref{hyp:WRef},
  for any $\lambda \in J''$, the level set $a_0^{-1}(\lambda)$ is a
  smooth compact submanifold of $\R^d \times \T^d$ of dimension
  $2d-1$. We consider $\psi$ with support in $(-t_0,t_0)$ where $t_0
  \in (0,T)$ is determined according to \eqref{eq:critSet_t0} with $U
  = \overline{J''}$ and the Hamiltonian chosen as $H = a_0$.

  We choose families of $\epsilon$-independent charts $(U_i,\sigma_i)$
  and $(U_i,\tilde{\sigma}_i)$ of $\R^d \times \T^d$ with $U_i \subset
  a_0^{-1}(J'')$ such that $(U_i)_i$ is an open covering of a tubular
  neighbourhood of $a_0^{-1}(J')$. Due to Hypothesis \ref{hyp:WRef}
  and the assumption on $J'$ we may assume this family to be
  finite. We shall indicate in the course of the proof how small the
  domains $U_i$ have to be chosen. We further assume that $\sigma_i =
  (x,\xi)$ represents natural coordinates and $\tilde{\sigma}_i= (a_0,
  \omega_1, \dots, \omega_{2d-1}) = (a_0, \omega)$ is a submanifold
  chart locally flattening the level sets of $a_0$, i.e.
  \begin{align}
    \label{eq:chartProperty}
    \tilde{\sigma}_i (a_0^{-1}(\lambda) \cap U_i) \subset \{ \lambda
    \} \times \R^{2d-1} \qquad (\lambda \in J'').
  \end{align}
  When using coordinates $x,\xi, a_0, \omega$, we have as usual
  suppressed the index $i$ labelling the correspondence to the local
  charts $\sigma_i$ and $\tilde\sigma_i$. These coordinates induce
  differential forms $dx, d\xi, da_0$ and $d\omega_j$ on
  $U_i$. Writing $d\omega = d\omega_1 \land \dots \land
  d\omega_{2d-1}$, the symplectic volume form $dvol$ on $U_i$ can then
  be represented as
  \begin{align}
    \label{eq:VolTransf} dvol = dx \land d\xi &= \det (D(\sigma_i \circ
    \tilde{\sigma}_i^{-1})) (a_0,\omega) (da_0
    \land d\omega).
  \end{align}
  Choose $\chi_i \in \Cont_0^\infty(U_i)$ such that
  \begin{align}
    \label{eq:PartUniLevelSet}
   \chi := \sum_i \chi_i = 1 \quad \mbox{ on some tubular
     neighbourhood of } a_0^{-1}(J').
  \end{align}
  
  For $\lambda \in J''$, we can now decompose
  \begin{align}
    I_2(\lambda;\psi,\epsilon) = I_{stat}(\lambda;\psi,\epsilon) +
    I_{nonst}(\lambda;\psi,\epsilon)
  \end{align}
  where
  \begin{align}
    I_{nonst}(\lambda;\psi,\epsilon) &:= \int_{\R \times \R^d \times
      \T^d} g_{\lambda} (1 - \chi) \, |dt \land dvol|
  \end{align}
  and
  \begin{align}
    \label{eq:IStatDef}
    I_{stat}(\lambda;\psi,\epsilon) &:= \sum_{i} A_{i}(\lambda) \qquad
    \mbox{ with } \qquad A_{i}(\lambda) := \int_{\R \times \R^d \times
      \T^d} g_{\lambda} \chi_i \, |dt \land dvol|.
  \end{align}
  Using \eqref{eq:VolTransf} and the definition of the integral of
  forms on manifolds, we see that
  \begin{align}
    \label{eq:IstatAi}
    A_{i}(\lambda) = \int_{\R^{2d-1}} B_{i}(\lambda,\omega) d\omega
    \qquad \mbox{ with } \qquad B_{i}(\lambda,\omega) &:= \int_{\R}
    \int_{\R} C_{\lambda,\omega}^i(t,a_0) dtda_0
  \end{align}
  and
  \begin{align}
    \label{eq:IstatTransfIntegrand}C_{\lambda,\omega}^i(t,a_0) &:= g_{\lambda}(t,
    \tilde{\sigma}_i^{-1}(a_0,\omega);\epsilon) (\chi_i \circ
    \tilde{\sigma}_i^{-1})(a_0,\omega) | \det D(\sigma_i \circ
    \tilde{\sigma}_i^{-1})|(a_0,\omega).
  \end{align}
  In the definitions for $A_i$, $B_i$ and $C_{\lambda,\omega}^i$, we
  suppressed the dependence on $\epsilon$. In \eqref{eq:IstatAi} and
  \eqref{eq:IstatTransfIntegrand}, with the usual abuse of notation,
  we consider $a_0$ and $\omega$ as elements of $\R$ and $\R^{2d-1}$,
  respectively, $dtda_0$ as the Lebesgue measure on $\R^{2}$ and
  $d\omega$ as the Lebesgue measure on $\R^{2d-1}$. Furthermore, in
  \eqref{eq:IstatTransfIntegrand} we have interpreted the function
  $\chi_i \circ \tilde{\sigma}_i^{-1}$ as an element of
  $\Cont_0^\infty(\R^{2d})$.

  We shall apply the parameter dependent version of the method of
  stationary phase, given in Corollary \ref{cor:statPhase}, to each
  integral $B_{i}(\lambda,\omega)$. For this purpose, we identify the
  integration variable $x$ and the parameter $y$ from Corollary
  \ref{cor:statPhase} as
  \begin{align}
    \label{eq:AppStatPhaseNot}
     x &= (t,a_0), \qquad y= (\lambda,\omega).
  \end{align}
  We emphasise that $x$ in \eqref{eq:AppStatPhaseNot} has a meaning
  different from $x$ introduced around \eqref{eq:chartProperty}. By
  the definition \eqref{eq:gLambda} of $g_{\lambda}$ and using the
  notation \eqref{eq:AppStatPhaseNot}, we see that the mapping $(x,y)
  \mapsto C_{y}^i(x)$ given in \eqref{eq:IstatTransfIntegrand} is
  smooth and that the phase function of the oscillating integral
  $B_i(\lambda,\omega)$ is given by
  \begin{align}
    \label{eq:varphiTransf}
    \varphi(x,y) := \varphi_{\lambda}(t,(\sigma_i \circ
    \tilde{\sigma}_i^{-1})(a_0,\omega)) \qquad (t \in (-t_0,t_0),
    (a_0,\omega) \in \im \tilde{\sigma}_i, \lambda \in J''),
  \end{align}
  which is real-valued. On the lhs of \eqref{eq:varphiTransf} we
  suppressed the label $i$. Let further
  \begin{align}
    \label{eq:xycritPoints}
    y \mapsto x(y) := (0,\lambda) \qquad \mbox{for } y =
    (\lambda,\omega) \in \im \tilde\sigma_i, \, \lambda \in J''.
  \end{align}
  Then, due to Lemma \ref{lem:CritManif} and the property
  \eqref{eq:chartProperty}, $x(\cdot)$ parametrises the critical
  points of $\varphi(\cdot,y)$, i.e. one has $\left. D_x \varphi
  \right|_{(x(y),y)} = 0$. Using the definition \eqref{eq:PhiLambda}
  for $\varphi_{\lambda}$ and the Hamilton-Jacobi equation
  \eqref{eq:HamJacEq} for $\phi$ with Hamiltonian $H = a_0$, we find
  \begin{align}
    \label{eq:varphiCritP} \varphi_{\lambda}(0,x,\xi) &= 0,\\ \label{eq:varphiCritP2} \partial_t
    \varphi_{\lambda}(0,x,\xi) &= -\partial_t \phi(0,x,\xi) - \lambda
    = a_0(x,\xi) - \lambda,
  \end{align}
  where $x$ now means the coordinates introduced around
  \eqref{eq:chartProperty}. This implies that the Hessian of
  $\varphi(\cdot,(\lambda,\omega))$ at the critical point
  $(0,\lambda)$ for given parameters $\lambda, \omega$ is of the form
  \begin{align}
    \label{eq:HessianForm}
    K(\lambda,\omega) := \left. D^2
    \varphi(\cdot,(\lambda,\omega)) \right|_{(0,\lambda)} =
    \begin{pmatrix}
      * & 1 \\ 1 & 0
    \end{pmatrix}.
  \end{align}
  So the critical points $(0,\lambda)$ are non-degenerate (since $\det
  K(\lambda,\omega) = -1$).

  We recall that $C_{\lambda,\omega}^i$ depends on $\epsilon$. But,
  defining the compact set $\Omega := \supp (\chi_i \circ
  \tilde\sigma_i^{-1})$, without loss of generality we may assume the
  mapping $(t,a_0,\lambda,\omega) \mapsto C_{\lambda,\omega}^i(t,a_0)$
  to have support in a sufficiently small $\epsilon$-independent
  neighbourhood of
  \begin{align}
    \{ (x(y),y) \; | \; y \in \Omega \} = \{
    (0,\lambda,\lambda,\omega) \; | \; (\lambda,\omega) \in \Omega \}
  \end{align}
  by choosing the domain $U_i$ (and thus the
  support of $\chi_i$ in \eqref{eq:PartUniLevelSet}) and the support
  of $\psi$ (which appears as a factor in the definition
  \eqref{eq:gLambda} of $g_{\lambda}$) sufficiently small. Therefore
  the conditions for applying Corollary \ref{cor:statPhase} are
  satisfied. Thus
  \begin{align}
    \label{eq:statPhaseAppl1}
    \int_{\R}\int_{\R} C^i_{\lambda,\omega}(t,a_0) dtda_0 = \epsilon A
    C^i_{\lambda,\omega} (0,\lambda) + \O(\epsilon^2), \qquad
    \mbox{uniformly for } \lambda \in J'', \omega \in \R^{2d-1},
  \end{align}
  where the constant $A$ is given by
  \begin{align}
    \label{eq:K0StatPhase}
    A = \frac{2\pi e^{i\pi/4 \sign K(\lambda,\omega)}}{|\det
      K(\lambda,\omega)|^{1/2}} = 2\pi.
  \end{align}
  Here we used that $\det K(\lambda,\omega) = -1$ and thus $\sign
  K(\lambda,\omega) = 0$. By the definition of $C_{\lambda,
    \omega}^i$ in \eqref{eq:IstatTransfIntegrand} and $g_{\lambda}$ in
  \eqref{eq:gLambda}, the amplitude in \eqref{eq:statPhaseAppl1}
  evaluated at the critical point $(0,\lambda)$ is given by
  \begin{align}
    \label{eq:StatPhaseApplIntegrand}
    C^i_{\lambda,\omega}(0,\lambda) &= \frac{1}{(2\pi\epsilon)^{d+1}}
    \left( e^{i\varphi_{\lambda}(0,\sigma_i)/\epsilon} \psi(0)
    \mu(0,x,x,\xi;\epsilon) \chi_i \middle)
    \right|_{\tilde\sigma_i^{-1}(\lambda,\omega)} | \det D(\sigma_i
    \circ \tilde{\sigma}_i^{-1})|(\lambda,\omega).
  \end{align}

  To finish the proof, we shall use \eqref{eq:statPhaseAppl1} to find
  an expansion for $I_2$. From \eqref{eq:FIOKernelInit} we get
  \begin{align}
    \mu(0,x,x,\xi;\epsilon) = c(x,\xi;\epsilon) \quad \mbox{ on }
    a_0^{-1}(J'').
  \end{align}
  Due to \eqref{eq:FuncCalcAsympcj}, we have $c \sim \sum_j \epsilon^j
  c_j$ where the leading order term $c_0$ is given by
  \begin{align}
    c_0 = f \circ a_0.
  \end{align}
  Thus, using the remainder estimates in Theorem
  \ref{theo:FuncCalcNonDiscr} (\ref{item:FuncCalcNonDiscr3}) for the
  asymptotic expansion of $c$,
  \begin{align}
    \label{eq:MuAsymp}
    \mu(0,x,x,\xi;\epsilon) = f(\lambda) + \O(\epsilon), \quad \mbox{
      uniformly on } a_0^{-1}(\lambda) \mbox{ and for } \lambda \in
    J''.
  \end{align}
  Inserting \eqref{eq:varphiCritP}, \eqref{eq:MuAsymp} and the
  assumption $\psi(0) = 1$ into \eqref{eq:StatPhaseApplIntegrand}, we
  find
  \begin{align}
    \notag C^i_{\lambda,\omega}(0,\lambda) =
    \frac{1}{(2\pi\epsilon)^{d+1}} (f(\lambda) + \O(\epsilon)) (\chi_i
    \circ \tilde{\sigma}_i^{-1})(\lambda,\omega) | \det D(\sigma_i
    \circ \tilde{\sigma}_i^{-1})|(\lambda,\omega),
    &\\\label{eq:IntegrandCOmega} \mbox{ uniformly for } \lambda \in
    J'', \omega \in \R^{2d-1}.&
  \end{align}
  Combining the definition of $B_i(\lambda,\omega)$ in
  \eqref{eq:IstatAi} with \eqref{eq:statPhaseAppl1} and
  \eqref{eq:IntegrandCOmega}, we get
  \begin{align}
    \notag B_i(\lambda,\omega) = \frac{1}{(2\pi\epsilon)^{d}}
    (f(\lambda) + \O(\epsilon)) (\chi_i \circ
    \tilde{\sigma}_i^{-1})(\lambda,\omega) | \det D(\sigma_i \circ
    \tilde{\sigma}_i^{-1})|(\lambda,\omega),
    &\\ \label{eq:BiEvaluated} \mbox{ uniformly for } \lambda \in J'',
    \omega \in \R^{2d-1}.
  \end{align}
    From \eqref{eq:SymplecLiou} and \eqref{eq:VolTransf} we see that
  \begin{align}
    \label{eq:LiouvOmega} L &= \left(\det (D(\sigma_i \circ
    \tilde{\sigma}_i^{-1})) \circ \tilde{\sigma}_i\right) d\omega.
  \end{align}
  To evaluate the integral of $B_i$ in \eqref{eq:IstatAi}, observe
  that, using \eqref{eq:LiouvOmega} and the definition of the integral
  of forms,
  \begin{align}
    \label{eq:LiouvOmegaInt}
    \int_{\R^{2d-1}}(\chi_i \circ
    \tilde{\sigma}_i^{-1})(\lambda,\omega) | \det D(\sigma_i \circ
    \tilde{\sigma}_i^{-1})|(\lambda,\omega) d\omega= \int_{a_0 =
      \lambda} \chi_i |L|,
  \end{align}
  where $d\omega$ denotes the Lebesgue measure on $\R^{2d-1}$. Thus,
  combining \eqref{eq:IstatAi}, \eqref{eq:BiEvaluated} and
  \eqref{eq:LiouvOmegaInt},
  \begin{align}
    A_i(\lambda) = \frac{1}{(2\pi\epsilon)^d} (f(\lambda) + \O(\epsilon))
    \int_{a_0 = \lambda} \chi_i \, |L|, \qquad \mbox{uniformly for }
    \lambda \in J'.
  \end{align}
  Using \eqref{eq:PartUniLevelSet} to sum over $i$, the integral
  $I_{stat}$ from \eqref{eq:IStatDef} can therefore be expanded as
  \begin{align}
    I_{stat}(\lambda; \psi,\epsilon) &= \frac{1}{(2\pi \epsilon)^d}
    \left( f(\lambda) \int_{a_0 = \lambda} |L| + \O(\epsilon) \right),
    \qquad \mbox{uniformly for } \lambda \in J'.
  \end{align}
  The asymptotics \eqref{eq:I2Asymp} now follows since
  $I_{nonst}(\lambda; \psi,\epsilon)$ is of order
  $\O(\epsilon^\infty)$ by standard arguments of non-stationary phase.
\end{proof}

For the sake of the reader we remark that the argument below
\eqref{eq:xycritPoints} verifying that the critical points
$(0,\lambda)$ are non-degenerate provides a direct way to prove the
inclusion ``$\subset$'' in \eqref{eq:critSet_t0}, avoiding the less
explicit contradiction argument given in the proof of Lemma
\ref{lem:CritManif}. It essentially takes advantage of the coordinates
introduced around \eqref{eq:chartProperty} to compute the Hessian at
the critical points.

In the next Lemma, we shall identify the rhs of \eqref{eq:I2Asymp} as
a suitable approximation of $I_1$ not only for $\lambda$ in a bounded
interval but for any $\lambda \in \R$. In addition, since we want this
approximation to be integrable with respect to $\lambda$, we need the
error term to have sufficient decay in $\lambda$.

\begin{Lem}
  \label{Lem:I1Liouv}
  Assume Hypothesis \ref{hyp:WRef}.  For $f \in \Cont_0^\infty(J)$
  compactly supported in the set of non-critical values of $a_0$ and
  any $\psi \in \Cont_0^\infty(\R)$ with $\psi(0) = 1$ and supported
  in a sufficiently small neighbourhood of $0$, the expression
  $I_1(\lambda;f,\psi,\epsilon)$ defined in \eqref{eq:I1Def} can be
  expanded as
  \begin{align}
    \label{eq:I1Asymp}
    I_1(\lambda; f, \psi,\epsilon) &= \frac{1}{(2\pi\epsilon)^d}
    \left( f(\lambda) \int_{a_0 = \lambda} |L| + \langle \lambda
    \rangle^{-N} \O\left(\epsilon \right) \right),
  \end{align}
  for any $N \in \N$ and uniformly for $\lambda \in \R$.
\end{Lem}

Note that in \eqref{eq:I1Asymp}, for $\lambda \notin \supp f$, we
interpret the leading order term as $0$, i.e.
\begin{align}
  f(\lambda) \int_{a_0 = \lambda} |L| = 0 \quad \mbox{ for } \lambda
  \notin \supp f.
\end{align}

\begin{proof}
  Combining Lemma \ref{lem:I1I2eps} and Proposition \ref{lem:I2Asymp},
  we have
  \begin{align}
    \label{eq:I1inside}
    I_1(\lambda;f,\psi,\epsilon) = \frac{1}{(2\pi\epsilon)^d} \left(
    f(\lambda) \int_{a_0 = \lambda} |L| + \O (\epsilon ) \right), \quad \mbox{ uniformly
  for } \lambda \in J',
  \end{align}
  where $J' \subset J$ is a compact interval only consisting of
  non-critical values of $a_0$ and containing a neighbourhood of
  $\supp f$.

  To get an asymptotic expansion for $\lambda \notin J'$, we 
 give another representation of $I_1$. Since
the integral in \eqref{eq:I1Def} converges with respect to the trace
norm, trace and integration in \eqref{eq:I1Def} can be
interchanged. By functional calculus, we therefore get
\begin{align}
    \notag I_1(\lambda;f,\psi,\epsilon) &= \frac{1}{2\pi \epsilon}
    \sum_{\lambda_j \in \supp f} \int_\R e^{it(\lambda_j -
      \lambda)/\epsilon} \psi(t) f(\lambda_j) dt
    \\ \label{eq:I1RepreFourier} &= \frac{1}{\epsilon \sqrt{2\pi}}\sum_{\lambda_j \in \supp f}
    f(\lambda_j) (\F_\epsilon \psi)(\lambda - \lambda_j),
\end{align}
where the sum is taken over the $\epsilon$-dependent eigenvalues
$\lambda_j$ of $\P_\epsilon$ and $\F_\epsilon \psi$ is the
$\epsilon$-scaled Fourier transform of $\psi$ defined by
\begin{align}
  \label{eq:EpsFourPsi}
  (\F_\epsilon \psi)(\lambda) := \frac{1}{\sqrt{2\pi}} \int_\R
  e^{-it\lambda/\epsilon} \psi(t) dt \qquad (\lambda \in \R).
\end{align}
  By integration by
  parts in the definition of $\F_\epsilon \psi$ given in
  \eqref{eq:EpsFourPsi}, we see that for any $N \in \N$ the
  $\epsilon$-scaled Fourier transform satisfies
  \begin{align}
    \label{eq:FourOrder}
    (\F_\epsilon \psi)(\lambda) = \langle \lambda \rangle^{-N}\O \left(
    \epsilon^N  \right), \quad
    \mbox{uniformly for } |\lambda| \geq C > 0.
  \end{align}
  Therefore, using compactness of $J'$, we obtain uniformly for
  $\lambda_j \in \supp f \subset\subset J'$ and $\lambda \in \R
  \setminus J'$,
  \begin{align}
    \label{eq:FourierI1}
    (\F_\epsilon \psi)(\lambda - \lambda_j) = \langle \lambda
    - \lambda_j \rangle^{-N} \O \left( \epsilon^N \right) = \langle
    \lambda \rangle ^{-N} \O \left( \epsilon^N \right).
  \end{align}
  Inserting \eqref{eq:FourierI1} into \eqref{eq:I1RepreFourier} and
  using the rough Weyl estimate \eqref{eq:NeigAsympRoughAdd}   
  from Corollary \ref{cor:WeylRough}, we
  get
  \begin{align}
    \label{eq:I1outside}
    I_1(\lambda;f,\psi,\epsilon) = \langle \lambda \rangle ^{-N} \O
    \left( \epsilon^{N-d-1} \right), \quad \mbox{ uniformly for }
    \lambda \in \R \setminus J'.
  \end{align}
   The statement \eqref{eq:I1Asymp} now follows by combining
   \eqref{eq:I1inside} for $\lambda \in J'$ and \eqref{eq:I1outside}
   for $\lambda \notin J'$ and using again compactness of $J'$.
\end{proof}

Integrating \eqref{eq:I1Asymp} with $f = f_1$ and $N \geq 2$ and using
\eqref{eq:SymplecLiou}, we have
\begin{align}
  \notag \int_\alpha^\infty I_1(\lambda; f_1, \psi,\epsilon) d\lambda
  &= \frac{1}{(2\pi\epsilon)^d} \left( \int_\alpha^\infty \left(
  f_1(\lambda) \int_{a_0 = \lambda} |L| \right) d\lambda +
  \O(\epsilon) \right) \\ \label{eq:IntI1PhaseVol} &= \frac{1}{(2\pi
    \epsilon)^d} \left( \int_{\substack{x \in \R^d,\, \xi \in \T^d
      \\ \alpha \leq a_0(x,\xi)}} f_1(a_0(x,\xi)) d\xi dx +
  \O(\epsilon) \right).
\end{align}
Thus we identified the rhs of \eqref{eq:f1asymp} with the first
expression in \eqref{eq:StepToWeyl}. Analogously, we can identify the
rhs of \eqref{eq:f3asymp} with the second expression in
\eqref{eq:StepToWeyl}.

It remains to identify the leading order terms on the lhs of
\eqref{eq:f1asymp} and \eqref{eq:f3asymp} with the expressions given
in \eqref{eq:StepToWeyl}. As a first step, we need the following
Lemma, which bounds the number of eigenvalues of $\P_\epsilon$ in an
interval of length $\epsilon$. It is a statement on the absence of
clustering of eigenvalues, uniformly in $\epsilon$. The proof needs
the construction of a semi-classical approximation of the time
evolution.

\begin{Lem}
  \label{lem:epsSubinterval}
  Assume Hypothesis \ref{hyp:WRef}.  For $\epsilon > 0$ sufficiently
  small let $J_\epsilon$ be a subinterval of $J$ such that the length
  $|J_\epsilon|$ of $J_\epsilon$ is of order $\O(\epsilon)$. In
  addition, we assume that there is a set covering all $J_\epsilon$
  which is compactly contained in $J$ and in the set of non-critical
  values of $a_0$. Then the number of eigenvalues of $\P_\epsilon$ in
  $J_\epsilon$ is of order $\O(\epsilon^{1-d})$.
\end{Lem}

\begin{proof}
  Without loss of generality, we shall assume that
  \begin{align}
    \label{eq:JepsLow}
    \frac{|J_\epsilon|}{\epsilon} \geq C > 0.
  \end{align}
  In fact, for any smaller interval the claimed estimate holds a
  fortiori.
  
  Let $f \in \Cont_0^\infty(J)$ be non-negative with $f = 1$ near
  $\bigcup_\epsilon J_\epsilon$ and with support compactly contained
  in the set of non-critical values of $a_0$. Applying Lemma
  \ref{lem:I1I2eps} and Proposition \ref{lem:I2Asymp} to $f$ and using
  \eqref{eq:I1RepreFourier}, we get
    \begin{align}
      \label{eq:NumEigSubinterv}
      \frac{1}{\epsilon \sqrt{2\pi}} \sum_{\lambda_j} f(\lambda_j)
      (\F_\epsilon \psi)(\lambda - \lambda_j) =
      \frac{1}{(2\pi\epsilon)^d} \left( f(\lambda) \int_{a_0 =
        \lambda} |L| + \O(\epsilon) \right),\quad \mbox{uniformly for
      } \lambda \in \R.
    \end{align}
    Integrating \eqref{eq:NumEigSubinterv} over $J_\epsilon$ yields
    \begin{align}
      \label{eq:NumEigSubinterv2}
      \frac{1}{\epsilon} \sum_{\lambda_j} f(\lambda_j)
      \int_{J_\epsilon} (\F_\epsilon \psi)(\lambda - \lambda_j)
      d\lambda = \O(\epsilon^{1-d}).
    \end{align}
    We claim that $\psi$ may be chosen such that $\F_\epsilon \psi$ is
    non-negative and $(\F_1 \psi)(0) > 0$. In fact we may choose $\psi
    = \frac{1}{\sqrt{2\pi}} g \ast \tilde{g}$ for some real-valued $g
    \in \Cont_0(\R)$ where $\tilde{g}(t) := g(-t)$. Then $\F_1 \psi =
    |\F_1 g|^2 \geq 0$ and thus $(\F_1\psi)(0) > 0$ by choosing $g$ to
    be non-negative anywhere and positive somewhere. Since $\psi(0) =
    \frac{1}{\sqrt{2\pi}} \int (\F_1\psi) (\lambda) d\lambda$, we may
    arrange that $\psi(0) = 1$. Choosing $g$ with small support
    guarantees a small support of $\psi$. It is straightforward to
    check that $\F_\epsilon \psi$ has the stated properties by using
    the scaling property
    \begin{align}
      \label{eq:epsFourScaling}
      (\F_\epsilon \psi)(\lambda) = (\F_1 \psi) \left(
      \frac{\lambda}{\epsilon} \right) \qquad (\lambda \in \R).
    \end{align}
    With this choice of $\psi$ and using $f = 1$ on each $J_\epsilon$,
    we obtain
    \begin{align}
      \label{eq:NumEigSubinterv3}
      \mbox{lhs} \eqref{eq:NumEigSubinterv2} \geq \frac{1}{\epsilon}
      \sum_{\lambda_j \in J_\epsilon} 1 \cdot \left( \int_{J_\epsilon}
      (\F_\epsilon \psi)(\lambda - \lambda_j) d\lambda \right).
    \end{align}
    Using the scaling property \eqref{eq:epsFourScaling}, we get for
    $\lambda_{j} \in J_\epsilon$
    \begin{align}
      \label{eq:NumEigSubinterv4}
      \frac{1}{\epsilon} \int_{J_\epsilon} (\F_\epsilon \psi)(\lambda
      - \lambda_j) d\lambda = \int_{J_{\epsilon,j}} (\F_1
      \psi)(\lambda) d\lambda \geq C' > 0, \quad \mbox{where }
      J_{\epsilon,j} := (J_\epsilon - \lambda_j)/\epsilon.
    \end{align}
    We shall show that $C'$ can be chosen independently of $\epsilon$
    and $j$: Assumption \eqref{eq:JepsLow} gives $|J_{\epsilon,j}|
    \geq C$. Furthermore, due to $|J_\epsilon| = \O(\epsilon)$, there
    is a compact set $K \subset \R$ with $J_{\epsilon,j} \subset K$
    for all $\epsilon > 0$ and $\lambda_j \in J_\epsilon$. Since $0
    \in J_{\epsilon,j}$ for any $\lambda_j \in J_\epsilon$, we have
    $\delta \cdot J_{\epsilon,j} \subset J_{\epsilon,j}$ for any
    $\delta \in (0,1)$. For $\delta$ sufficiently small, the minimum
    $M$ of $\F_1 \psi$ on $\delta \cdot K$ is positive, since
    $(\F_1 \psi)(0) > 0$. Thus, using non-negativity of
    $\F_1 \psi$,
    \begin{align}
      \label{eq:NumEigSubinterv5}
      \int_{J_{\epsilon,j}} (\F_1 \psi)(\lambda) d\lambda
      \geq \int_{\delta \cdot J_{\epsilon,j}} (\F_1
      \psi)(\lambda) d\lambda \geq \delta C M.
    \end{align}
    The estimate in \eqref{eq:NumEigSubinterv4} follows from
    \eqref{eq:NumEigSubinterv5} with $C' = \delta CM$ .

    Combining \eqref{eq:NumEigSubinterv2}, \eqref{eq:NumEigSubinterv3}
    and \eqref{eq:NumEigSubinterv4} implies the statement in Lemma
    \ref{lem:epsSubinterval}.
\end{proof}

The following lemma finally identifies in leading order the lhs of
\eqref{eq:f1asymp} and \eqref{eq:f3asymp} with an integral over $I_1$
(depending on $f_1$ and $f_3$, respectively). Combined with the
asymptotic relation in \eqref{eq:IntI1PhaseVol}, this proves the two
estimates \eqref{eq:f1asymp} and \eqref{eq:f3asymp} in Proposition
\ref{prop:f1f3asymp}.

\begin{Lem}
  Assume Hypothesis \ref{hyp:WRef}. For any $\psi \in
  \Cont_0^\infty(\R)$ with $\psi(0) = 1$, the expressions
  $I_1(\lambda;f_1,\psi,\epsilon)$ and
  $I_1(\lambda;f_3,\psi,\epsilon)$ defined by \eqref{eq:I1Def} satisfy
  \begin{align}
    \label{eq:intI1}
      \sum_{\lambda_j \geq \alpha} f_1(\lambda_j) &=
      \int_\alpha^\infty I_1(\lambda;f_1,\psi,\epsilon) d\lambda +
      \O(\epsilon^{1-d}),\\ \label{eq:intI2} \sum_{\lambda_j \leq
        \beta} f_3(\lambda_j) &= \int_{-\infty}^\beta
      I_1(\lambda;f_3,\psi,\epsilon) d\lambda + \O(\epsilon^{1-d}),
  \end{align}
  where we sum over eigenvalues $\lambda_j$ of $\P_\epsilon$.
\end{Lem}

\begin{proof}
  We shall prove only \eqref{eq:intI1}. The statement \eqref{eq:intI2}
  follows by analogous arguments.

  Using the representation \eqref{eq:I1RepreFourier} for $I_1$, we may
  write the lhs of \eqref{eq:intI1} as
  \begin{align}
    \label{eq:intI1step}
    \int_\alpha^\infty I_1(\lambda;f_1,\psi,\epsilon) d\lambda =
    \frac{1}{\epsilon\sqrt{2\pi}} \sum_{\lambda_j \in \supp f_1}
    f_1(\lambda_j) \int_\alpha^\infty (\F_\epsilon \psi)(\lambda -
    \lambda_j) d\lambda.
  \end{align}
  Using the scaling property \eqref{eq:epsFourScaling} for the
  $\epsilon$-scaled Fourier transform, the integral on the rhs of
  \eqref{eq:intI1step} takes the form
  \begin{align}  \label{eq:FourierPsiStep}
     \int_\alpha^\infty (\F_\epsilon \psi)(\lambda - \lambda_j)
    d\lambda &= \int_\alpha^\infty (\F_1 \psi) \left( \frac{\lambda -
      \lambda_j}{\epsilon} \right) d\lambda
   =
    \epsilon\, \int_{\frac{\alpha-\lambda_j}{\epsilon}}^\infty  (\F_1
    \psi) (\lambda) d\lambda.
  \end{align}
  Due to the Fourier inversion theorem, we have
  \begin{align}
    \frac{1}{\sqrt{2\pi}} \int_{-\infty}^\infty (\F_1 \psi)(\lambda)
    d\lambda = \psi(0) = 1.
  \end{align}
  Therefore 
  \begin{align}\label{eq:defRepsilonj}
    R_\epsilon^j :=  \I_{[\alpha,\infty)}(\lambda_j) - \frac{1}{\sqrt{2\pi}} \int_{\frac{\alpha-\lambda_j}{\epsilon}}^\infty (\F_1
    \psi) (\lambda) d\lambda  = 
    \begin{cases}
      - \frac{1}{\sqrt{2\pi}} \int_{\frac{\alpha-\lambda_j}{\epsilon}}^\infty (\F_1
      \psi) (\lambda) d\lambda & \mbox{if } \alpha > \lambda_j
      \\  \frac{1}{\sqrt{2\pi}} \int_{-\infty}^{\frac{\alpha-\lambda_j}{\epsilon}}
      (\F_1 \psi)(\lambda) d\lambda & \mbox{otherwise.}
    \end{cases}
  \end{align}
    Using \eqref{eq:FourierPsiStep} and \eqref{eq:defRepsilonj} we may write 
  \begin{equation} \label{eq:sumf1umformung}
  \sum_{\lambda_j \geq \alpha} f_1(\lambda_j) = \sum_{\lambda_j \in \supp f_1} f_1(\lambda_j) \I_{[\alpha,\infty)}(\lambda_j) =
   \sum_{\lambda_j \in \supp f_1} f_1(\lambda_j) \left(R_\epsilon^j  + 
 \frac{1}{\epsilon\sqrt{2\pi}} \int_\alpha^\infty (\F_\epsilon \psi)(\lambda -
    \lambda_j) d\lambda \right).
    \end{equation}

 We now claim 
  \begin{align}
    \label{eq:countSumRemain}
    \sum_{\lambda_j \in \supp f_1} f_1(\lambda_j) R_\epsilon^j = \O
    (\epsilon^{1-d}).
  \end{align}
Then, inserting \eqref{eq:countSumRemain} and \eqref{eq:intI1step} into \eqref{eq:sumf1umformung},
gives \eqref{eq:intI1}.
   
   It remains to prove \eqref{eq:countSumRemain}. 
   
  Since $\F_1\psi$ is a Schwartz function, \eqref{eq:defRepsilonj} gives
  \begin{align}
    \label{eq:FouPsiCasesOrder}
    R_\epsilon^j = \O\left( \left\langle \frac{\alpha -
      \lambda_j}{\epsilon} \right\rangle^{-N} \right) \quad \mbox{for
      any } N \in \N, \mbox{ uniformly for } \lambda_j \in \supp f_1.
  \end{align}
   Now let $E_\epsilon^m$ be the
  subset of all eigenvalues $\lambda_j \in \supp f_1$ that are
  contained in the interval $[\alpha + m\epsilon, \alpha +
    (m+1)\epsilon)$, for $m \in \Z$. The function $\lambda \mapsto
    \left\langle \frac{\alpha - \lambda}{\epsilon} \right\rangle^{-2}$
    defined on this interval and arising as a bound in
    \eqref{eq:FouPsiCasesOrder} for $N = 2$ takes its supremum at the
    boundary, i.e.
    \begin{align}
      \sup_{\lambda \in [\alpha + m\epsilon, \alpha + (m+1)\epsilon)}
        \left\langle \frac{\alpha - \lambda}{\epsilon}
        \right\rangle^{-2} = \max \left\{ \left\langle m
        \right\rangle^{-2}, \left\langle m+1 \right\rangle^{-2}
        \right\}.
    \end{align}
    As a consequence,
    \begin{align}
      \label{eq:boundMax}
      \max_{\lambda_j \in E_\epsilon^m} |R_\epsilon^j| = \langle m
      \rangle^{-2} \O (1), \qquad \mbox{uniformly for } m \in \Z.
    \end{align}
    Due to Lemma \ref{lem:epsSubinterval} the number of eigenvalues
    $\lambda_j$ in $E_\epsilon^m$ is of order $\O(\epsilon^{1-d})$,
    uniformly for $m \in \Z$. Therefore, using \eqref{eq:boundMax},
    boundedness of $f_1$ and the fact that $(E_\epsilon^m)_m$ is a
    decomposition of the set of eigenvalues $\lambda_j \in \supp f_1$,
  \begin{align}
    \notag \sum_{\lambda_j \in \supp f_1} |f_1(\lambda_j)
    R_\epsilon^{j}| &= \sum_{m \in \Z} \left( \sum_{\lambda_j \in
      E_\epsilon^m} |f_1(\lambda_j) R_\epsilon^{j}| \right)
    \\ \notag &= \O(\epsilon^{1-d}) \sum_{m \in \Z} \langle m
    \rangle^{-2} \\ &= \O(\epsilon^{1-d}),
  \end{align}
  which proves \eqref{eq:countSumRemain}.
\end{proof}

\section*{Acknowledgement}
We thank Bernard Helffer for helpful discussions on the subject of this paper and earlier relevant work.

\begin{appendix}

\section{Pseudo-differential operators in the discrete setting}
\label{sec:AppA}

Pseudo-differential operators in a discrete setting have already
been introduced in previous works like \cite{KR09} and \cite{KR18} as
a tool to study difference operators on a lattice. As explained there,
difference operators are induced by symbols periodic with respect to
the momentum variable. In this section, we shall recall basic definition and properties. 

In particular, 
we shall discuss the intertwining property between the standard $t$-quantisation  $\Op_{\epsilon,t} a $  and the discrete 
$t$-quantisation
$\Op_{\epsilon,t}^{\T} a $ given by restriction to the lattice (see
Proposition \ref{prop:RestrForm}), a change of quantisation formula from $s$- to $t$-quantisation, the symbolic calculus for 
our operators and a discrete  version of the Calderon-Vaillancourt theorem. 

We conclude this section with a result on the 
effect of conjugation of a discrete pseudodifferential operator by a rapidly oscillating multiplication operator and give its 
principal symbol. Conceptually, this is a result of Egorov type, but we do not prove the full Egorov theorem for the class of 
operators considered here. This very special type  of result is sufficient for our application to the time parametrix. 
We shall prove periodicity and uniform control on all parameters in a more elementary way, using a result on the quantisation 
of symbols $a(x,y,\xi)$ depending on 2 space variables and 1 momentum variable. This is all we need here.

We remark that throughout this work, by slight abuse of notation, we
shall identify any mapping from the $d$-dimensional torus $\T^d = \R^d
/ 2\pi\Z^d$ also as a $(2\pi\Z^d)$-periodic mapping from $\R^d$. Thus,
whenever referring to standard literature on pseudo-differential
operators like \cite{DiSj} or \cite{Mart}, we might consider the
 spaces of symbols introduced in our work as subsets of the 
spaces of symbols treated there.

Let $N, d \in \N^*$ and $\epsilon_0 \in (0,1]$.  A function $m : \R^N
  \times \T^d \rightarrow (0,\infty)$ is called an order function if
  there are constants $C > 0$, $M \in \N$ such that
\begin{align}
  \label{eq:orderFuncDef}
  m(x,\xi) \leq C \jap{ x - y }^M m(y,\mu) \qquad (x,y \in \R^N,\,
  \xi,\mu \in \T^d)
\end{align}
where $\jap{ x } := \sqrt{ 1 + \abs{x}^2 }$.  For $k \in \R$ we then
define the symbol class $\Sy^k(m,\epsilon_0)(\R^{N} \times \T^d)$ as
the space of functions $a : \R^N \times \T^d \times (0,\epsilon_0]
  \rightarrow \C$ with $a( \cdot, \cdot; \epsilon) \in
  \Cont^\infty(\R^N \times \T^d)$ for $\epsilon \in (0, \epsilon_0]$
    that for some constants $C_{\alpha} > 0$ ($\alpha \in \N^{N\times d}$)
    satisfy
\begin{align}
    \abs{ \partial_{x,\xi}^\alpha a(x,\xi;\epsilon) } \leq C_{\alpha}
    \epsilon^k m(x,\xi) \qquad (x \in \R^N, \xi \in \T^d,\, \epsilon \in
    (0,\epsilon_0]).
\end{align}
The space $\Sy^k(m,\epsilon_0)(\R^{N} \times \T^d)$ can be equipped with the
Fréchet seminorms
  \begin{align}
    \label{eq:FrechetSNSym}
    \norm{ a }_{\alpha} := \sup_{\substack{x \in \R^N, \xi \in \T^d \\ \epsilon \in
        (0,\epsilon_0]}} \frac{ \abs{\partial_{x,\xi}^\alpha a(x,\xi;\epsilon)}
      }{\epsilon^k m(x,\xi)} \qquad (\alpha \in \N^{N\times d}).
  \end{align}
For $\epsilon \in (0,\epsilon_0]$, we adapt the Schwartz space
  $\S(\R^N)$ to a discrete version $\s(\epsilon\Z^N)$ by defining it
  as the space of functions $u : \epsilon \Z^N \rightarrow \C$ with
\begin{align}
  \label{eq:s_seminorm}
  \norm{ u }_{\epsilon,\alpha} := \sup_{x\in \epsilon\Z^N}
  \abs{x^\alpha u(x)} < \infty \qquad \mbox{for all $\alpha \in
    \N^N$.}
\end{align}
It is known (see \cite{DiSj}) that for $t \in [0,1]$ and a symbol $a
\in \Sy^k(m,\epsilon_0)(\R^d \times \T^d)$ the standard
pseudo-differential operator $\Op_{\epsilon,t} a : u \mapsto
\left(\Op_{\epsilon,t} a\right) u$ where
\begin{align}
  \label{eq:OpaNonDiscrDef}
    \left( \Op_{\epsilon,t} a \right) u(x) :=
    \frac{1}{(2\pi\epsilon)^d} \int_{\R^{2d}} e^{i(y-x)\xi/\epsilon}
    a(tx + (1-t)y, \xi; \epsilon) u(y) dyd\xi \qquad (x \in \R^d)
\end{align}
is well-defined and continuous as a mapping $\S(\R^d) \rightarrow
\S(\R^d)$. Here, recall that we consider $a(x,\xi)$ as a function on $\R^{2d}$ periodic
with respect to $\xi \in \R^d$. 

For $u \in \s(\epsilon\Z^d)$, we now
define the function $ \left(\Op_{\epsilon,t}^\T a\right)u :
\epsilon\Z^d \rightarrow \C$ by
\begin{align}
  \label{eq:OpTa_u_Def}
  \left(\Op_{\epsilon,t}^\T a\right)u(x) := \frac{1}{(2\pi)^d} \sum_{y
    \in \epsilon \Z^d} \int_{\T^d} e^{i(y-x)\xi/\epsilon} a(tx
  + (1-t)y, \xi; \epsilon) u(y) d\xi \qquad (x \in \epsilon\Z^d).
\end{align}
It is clear that $\left(\Op_{\epsilon,t}^\T a\right)u(x)$ is
well-defined for fixed $x$: The sum on the rhs of
\eqref{eq:OpTa_u_Def} converges absolutely since the symbol $a$ is
bounded by some polynomial and $u \in \s(\epsilon\Z^d)$. In fact,
$\left(\Op_{\epsilon,t}^\T a\right)u$ is even a function in
$\s(\epsilon\Z^d)$. We check this by relating the standard
non-discrete pseudo-differential operators in
\eqref{eq:OpaNonDiscrDef} to their discrete version in
\eqref{eq:OpTa_u_Def}. The following proposition states that the
action \eqref{eq:OpTa_u_Def} is essentially the restriction of $
\left( \Op_{\epsilon,t} a \right) u$ to the lattice
$\epsilon\Z^d$. Defining the restriction map
  \begin{align}
    \label{eq:RestrMapDef}
  r_\epsilon &: \S(\R^d) \rightarrow \s(\epsilon\Z^d), \quad
  (r_\epsilon u)(x) = u(x) \qquad (u \in \S(\R^d),\, x
  \in \epsilon\Z^d)
\end{align}
we have
\begin{prop}
  \label{prop:RestrForm}
  Let $a \in \Sy^k(m,\epsilon_0)(\R^d \times \T^d)$. Then for
  $\epsilon \in (0,\epsilon_0]$ and $t \in [0,1]$
  \begin{align}
    \label{eq:RestrForm}
    \left( r_\epsilon \circ \Op_{\epsilon,t}a \right)u(x) =
    \left( \Op_{\epsilon,t}^\T a \right) (r_\epsilon u)(x)
    \qquad (u \in \S(\R^d),\, x \in \epsilon\Z^d).
  \end{align}
\end{prop}
We remark that a version of Proposition \ref{prop:RestrForm} has been
proven in \cite[Proposition A.2]{KR18}. There, the more general case
of operators $\widetilde{\Op}_\epsilon {a}$ induced by a
symbol ${a}(x,y,\xi;\epsilon) \in \Sy_\delta^k(m)(\R^{2d} \times \T^d)$ is
treated. These operators act by
\begin{align}
  \label{eq:OpaNonDiscrDefgen}
    \left( \widetilde{\Op}_{\epsilon} a \right) u(x) :=
    \frac{1}{(2\pi\epsilon)^d} \int_{\R^{2d}} e^{i(y-x)\xi/\epsilon}
    a(x,y, \xi; \epsilon) u(y) dyd\xi \qquad (x \in \R^d).
\end{align}
Setting ${a}_t(x,y,\xi;\epsilon) := a((1-t)x +
ty,\xi;\epsilon)$ for $a \in \Sy^k(m)(\R^d \times \T^d)$, one has
${a}_t \in \Sy_0^k(m)(\R^{2d} \times \T^d)$ and 
$\widetilde{\Op}_{\epsilon}
{a}_t = \Op_{\epsilon,t} a$. 

Vice versa, for a general symbol $a(x,y,\xi)$ in $\Sy_0^k(m)(\R^{2d} \times \T^d)$  there is a symbol 
$a_t \in \Sy^k(m)(\R^d \times \T^d)$
such that
\begin{align}
    \label{eq:genquant}
  \widetilde{\Op}_{\epsilon}
{a} = \Op_{\epsilon,t} a_t,  
  \end{align}
where the principal symbol of $a_t(x,\xi)$ is given by $a(x,x,\xi)$,
see \cite[Proposition A.5]{KR18}.

Thus, the $t$-quantisation \eqref{eq:OpaNonDiscrDef} may be
considered as a special case of the general quantisation \eqref{eq:OpaNonDiscrDefgen} (compare
\cite[Remark A.3]{KR18} but keep in mind that compared to \cite{KR18}
we used a different convention in the definitions
\eqref{eq:OpaNonDiscrDef} and \eqref{eq:OpTa_u_Def}, with $t$ replaced
by $1-t$). 

We also note that \cite[Proposition A.2]{KR18} is
restricted to functions $u$ with compact support but easily extends to
$u \in \S(\R^d)$ using continuity (compare \cite[Remark
  A.4]{KR18}). Lastly, we remark that Proposition \ref{prop:RestrForm}
has also been proven in \cite[Appendix A]{KR09} for the
$(t=1)$-quantisation.

For completeness sake, we remark that for this more general quantisation there is also a discrete version
$\widetilde{\Op}_\epsilon^{\T}\tilde{a}$  such that the analog of the restriction formula
\eqref{eq:RestrForm} and the formula \eqref{eq:genquant} on the $t$-quantisation hold in this case. 
However, we shall not formally need this result.

 For the sake of the reader we recall the proof
of Proposition \ref{prop:RestrForm}.

\begin{proof}
   Using the $\epsilon$-scaled Fourier transform
  \begin{align}
    \label{eq:FourTransfEpsScaled}
    \F_\epsilon u(x) = \sqrt{2\pi}^{-d} \int_{\R^d} e^{-ix\xi/\epsilon}
    u(\xi) d\xi,
  \end{align}
  we can write for $u \in \S(\R^d)$
  \begin{align}
    (\Op_{\epsilon,t} a) u(x) = (\epsilon \sqrt{2\pi})^{-d}
    \int_{\R^d} (\F_\epsilon a(tx + (1-t)y, \cdot;\epsilon))(x-y) u(y)
    dy.
  \end{align}
  Since for any $2\pi\Z^d$-periodic function $g \in
  \Cont^\infty(\R^d)$ the Fourier transform is given by
  \begin{align}
    \F_\epsilon g = \left( \frac{\epsilon}{\sqrt{2\pi}} \right)^d
    \sum_{z \in \epsilon\Z^d} \delta_z c_z, \quad \mbox{ where } \quad
    c_z := \int_{\T^d} e^{-iz\mu / \epsilon} g(\mu) d\mu,
  \end{align}
  we formally get for any $x \in \R^d$
  \begin{align}
    \notag (\Op_{\epsilon,t} a) u(x) &= \frac{1}{(2\pi)^d} \sum_{z \in
      \epsilon\Z^d} \int_{\T^d} \int_{\R^d} e^{-iz\mu/\epsilon} a(tx +
    (1-t)y, \mu; \epsilon) \delta_z(x-y) u(y) dy d\mu \\
    \label{eq:OpWithKernel}
    &= \sum_{y \in G_x} K(x,y) u(y)
  \end{align}
  with $G_{x} = x + \epsilon\Z^d$ and the pointwise defined  kernel
  \begin{align}
  \label{kernel}
    K(x,y) = \frac{1}{(2\pi)^d} \int_{\T^d} e^{i(y-x)\mu/\epsilon}
    a(tx + (1-t)y, \mu; \epsilon) d\mu.
  \end{align}
  Restricting to $x \in \epsilon\Z^d$ we have $G_{x} =
  \epsilon\Z^d$. So, applying the restriction map $r_\epsilon$
  to \eqref{eq:OpWithKernel}, we conclude
  \begin{align}
    \notag \left( r_\epsilon \circ \Op_{\epsilon,t}a \right)u(x) &=
    \sum_{y \in  \epsilon\Z^d} K(x,y) u(y) = \left( \Op_{\epsilon,t}^\T a
    \right) ( r_\epsilon u)(x).
  \end{align}
\end{proof}
Since $\Op_{\epsilon,t}a$ maps $\S(\R^d)$ into $\S(\R^d)$
continuously, it is a direct consequence of Proposition
\ref{prop:RestrForm} that $\Op_{\epsilon,t}^\T a$ 
maps $\s(\epsilon\Z^d)$ into
$\s(\epsilon\Z^d)$ continuously, where $\s(\epsilon \Z^d)$ is equipped
with the Fréchet topology induced by the seminorms $\norm{ \cdot
}_{\epsilon,\alpha}$ of \eqref{eq:s_seminorm}.

In order to extend $\Op_{\epsilon,t}^\T a$ to a
continuous operator on $\s'(\epsilon\Z^d)$, we define the bilinear form
\begin{align}
  \label{eq:dualps}
  \dualp{u}{v} := \sum_{x \in \epsilon\Z^d} u(x) v(x) \qquad ( u, v
  \in \s(\epsilon\Z^d))
\end{align}
and, by abuse of notation, extend \eqref{eq:dualps} to a dual pairing
between $\s(\epsilon\Z^d)$ and $\s'(\epsilon\Z^d)$.  Identifying an
element $u \in \s(\epsilon\Z^d)$ with the distribution
$\dualp{u}{\cdot} \in \s'(\epsilon\Z^d)$, the space $\s(\epsilon\Z^d)$
may be continuously embedded into $\s'(\epsilon\Z^d)$, where
$\s'(\epsilon\Z^d)$ is endowed with the weak$^\ast$-topology. For $u,
v \in \s(\epsilon\Z^d)$ we have
\begin{align}
  \label{eq:OpExtSDualPrep}
  \notag \dualp{ \left( \Op_{\epsilon,t}^\T a \right) u }{
    v } &= \frac{1}{(2\pi)^d}\sum_{x \in \epsilon\Z^d}
  \sum_{y \in \epsilon\Z^d} \int_{\T^d} e^{i(y-x)\xi/\epsilon}
  a(tx + (1-t)y,\xi;\epsilon) u(y) d\xi v(x) \\ \notag &=
  \frac{1}{(2\pi)^d}\sum_{y \in \epsilon\Z^d} u(y)  \sum_{x
      \in \epsilon\Z^d} \int_{\T^d} e^{i(y-x)\xi/\epsilon}
  a((1-t)y + tx, \xi;\epsilon) v(x) d\xi \\ \notag &=
  \frac{1}{(2\pi)^d}\sum_{y \in \epsilon\Z^d} u(y)  \sum_{x
      \in \epsilon\Z^d} \int_{\T^d} e^{i(x-y)\xi/\epsilon}
  a((1-t)y + tx, -\xi;\epsilon) v(x) d\xi
  \\&= \dualp{ u
  }{  \left( \Op_{\epsilon,1-t}^\T a' \right) v 
  }
\end{align}
with $a'(x,\xi;\epsilon) := a(x,-\xi;\epsilon)$. Thus the restriction
of the adjoint operator to $\s(\epsilon\Z^d)$ is itself a continuous
mapping. We may therefore extend $\Op_{\epsilon,t}^\T a$ to a
continuous operator $\s'(\epsilon\Z^d) \rightarrow \s'(\epsilon\Z^d)$
by defining
\begin{align}
  \label{eq:OpSDual}
  \dualp{ \left( \Op_{\epsilon,t}^\T a \right) u' }{ v } := \dualp{ u'
  }{ \left( \Op_{\epsilon,1-t}^\T a' \right) v } \qquad (u' \in
  \s'(\epsilon\Z^d),\, v \in \s(\epsilon\Z^d)).
\end{align}

As usual, for symbols  $a_j \in \Sy^{k_j}(m,\epsilon_0)(\R^{d}\times \T^d)$, $a\in \Sy^{k_0}(m,\epsilon_0)(\R^{d}\times \T^d)$ where 
the sequence $\bigl(k_j\bigr)_{j \in \N}$ is increasing with $k_j \rightarrow \infty$, 
we write 
  \begin{align}
    a(x,\xi;\epsilon) \sim \sum_{j = 0}^\infty a_j(x,\xi;\epsilon) \;
    \mbox{ if and only if } \; \left( a - \sum_{j=0}^{M} a_j \right) \in
    \Sy^{k_{M+1}}(m, \epsilon_0)(\R^{d} \times \T^d) \mbox{ for any
    } M \in \N.
  \end{align}
The formal
sum  $\sum_{j = 0}^\infty a_j$ is called asymptotic expansion of 
the symbol $a$.

The following proposition states that a pseudodifferential operator
can be represented with any quantisation parameter $t$. Given the
symbol $a_s$ of an operator in $s$-quantisation, the symbol $a_t$ for the
same operator in $t$-quantisation can be computed by formula
\eqref{eq:SwitchForm}. In particular, considering the asymptotic
expansion \eqref{eq:SwitchAsympExp} of $a_t$, the leading order term
equals $a_s$.

\begin{prop}
  \label{prop:SwitchQuant}
  Let $s \in [0,1]$ and $a_s \in \Sy^k(m,\epsilon_0)(\R^d \times
  \T^d)$. Then for any $t \in [0,1]$ there is a unique $a_t \in
  \Sy^k(m,\epsilon_0)(\R^d \times \T^d)$ such that
  $\Op_{\epsilon,t}^\T a_t = \Op_{\epsilon,s}^\T a_s$ for $\epsilon
  \in (0,\epsilon_0]$. Moreover, the mapping $a_s \mapsto a_t$ is
    continuous. Formally, $a_t$ is given by
  \begin{align}
    \label{eq:SwitchForm}
    a_t(x,\xi;\epsilon) = \frac{1}{(2\pi)^d} \sum_{y \in \epsilon\Z^d}
    \int_{\T^d} e^{i(\xi-\mu)y/\epsilon} a_s(x+(s-t)y,\mu;\epsilon)
    d\mu \qquad (x,\xi \in \R^d,\, \epsilon \in (0,\epsilon_0]).
  \end{align}
  Furthermore,
  \begin{align}
    \label{eq:SwitchAsympExp}
    a_t(x,\xi;\epsilon) \sim \sum_{j=0}^\infty \epsilon^j
    a_{t,j}(x,\xi;\epsilon) \qquad \mbox{where} \quad a_{t,j}(x,\xi;\epsilon) :=
    \sum_{\substack{\alpha \in \N_0^d \\ |\alpha| = j}}
    \frac{i^j}{\alpha!} \left.\partial_\mu^\alpha \partial_y^\alpha a_s(x +
    (s-t)y,\mu;\epsilon) \right|_{\substack{y=0 \\ \mu = \xi}}.
  \end{align}
  Writing
  \begin{align}
    \label{eq:SwitchAsympExpRemain}
    R_N(a_s)(x,\xi;\epsilon) := a_t(x,\xi;\epsilon) - \sum_{j=0}^{N-1}
    \epsilon^j a_{t,j}(x,\xi;\epsilon),
  \end{align}
  we have $R_N(a_s) \in S^{k+N}(m,\epsilon_0)(\R^d \times \T^d)$ and
  the Fréchet seminorms of $R_N$ only depend linearly on finitely many
  $\norm{a_s}_{\alpha}$ with $\abs{\alpha} \geq N$. 
\end{prop}

\begin{proof}
  Considering the $t$-quantisation as a special case of the general
  quantisation as described below Proposition \ref{prop:RestrForm},
  Proposition \ref{prop:SwitchQuant} may be seen as a special case of
  \cite[Proposition A.5]{KR18}.
\end{proof}

\eqref{eq:SwitchForm} can be understood in a distributional sense or
as an iterated integral. The analytically nontrivial estimate is the
estimate on the remainder \eqref{eq:SwitchAsympExpRemain}. We use it
for estimates uniform with respect to a parameter.

The next proposition determines the symbol of the composition of
pseudodifferential operators in the discrete setting. The symbol has
an asymptotic expansion that can be derived from the derivatives of
the symbols of the operators involved. The proposition follows
analogous results from the standard theory (see \cite[Proposition 7.7
  and Theorem 7.9]{DiSj}) but we make the statements on the remainder
estimate more precise. A special case in the discrete setting has
already been proved in \cite[Corollary A.5]{KR09}.

\begin{prop}
  \label{prop:SharpProd}
  Let $t \in [0,1]$
  and $a_j \in \Sy^0(m_j,\epsilon_0)(\R^d \times \T^d)$ for
  $j \in \left\{1, 2\right\}$. Define
  \begin{align}
    (a_1 \comp{t} a_2)(x,\xi;\epsilon) := \left. e^{i\epsilon \left(
          \nabla_\eta \cdot \nabla_v^T - \nabla_u \cdot \nabla_\xi^T
        \right) } a_1(tx + (1-t)u,\eta;\epsilon) a_2((1-t)x +
      tv,\xi;\epsilon) \right|_{\substack{u = v = x \\ \eta = \xi}} 
  \end{align}
  for $x \in \R^d, \xi \in \T^d$, $\epsilon \in (0,\epsilon_0]$. Then $a_1
  \comp{t} a_2 \in \Sy^0(m_1m_2, \epsilon_0)(\R^d \times \T^d)$ and
  \begin{multline}
    \label{eq:SharpProdAsExp}
    (a_1 \comp{t}  a_2)(x,\xi;\epsilon) \\
     \sim \sum_{k=0}^\infty
    \left.\frac{1}{k!} \left( i\epsilon \right)^k \left(
        \nabla_\eta \cdot \nabla_v^T - \nabla_u \cdot \nabla_\xi^T
      \right)^k a_1(tx + (1-t)u,\eta;\epsilon) a_2((1-t)x + tv,\xi;\epsilon)
    \right|_{\substack{u=v=x \\ \eta = \xi}}.
  \end{multline}
  The remainder 
  \begin{align}
    \label{eq:SharpProdAsExpRemain}
    & R_N(a_1,a_2)(x,\xi;\epsilon) := (a_1 \comp{t}
    a_2)(x,\xi;\epsilon) \\
    & \quad - \sum_{k=0}^{N-1} \left.\frac{1}{k!} \left( i\epsilon
      \right)^k \left( \nabla_\eta \cdot \nabla_v^T - \nabla_u \cdot
        \nabla_\xi^T \right)^k a_1(tx + (1-t)u,\eta;\epsilon)
      a_2((1-t)x + tv,\xi;\epsilon) \right|_{\substack{u=v=x \\ \eta =
        \xi}} \nonumber
  \end{align}
  is an element of $\Sy^N(m_1m_2, \epsilon_0)(\R^d \times \T^d)$ and
  its Fréchet seminorms only depend linearly on finitely many Fréchet
  seminorms of the symbols $a_1$ and $a_2$. Furthermore,
  \begin{align}
    \left( \Op_{\epsilon,t}^\T a_1 \right) \circ \left(
      \Op_{\epsilon,t}^\T a_2 \right) = \Op_{\epsilon,t}^\T (a_1
    \comp{t} a_2) \qquad (\epsilon \in (0,\epsilon_0]).
  \end{align}
  We define $\compw := \comp{1/2}$.
\end{prop}

\begin{proof}
  For the special case $t = 1$, this statement is proved in
  \cite[Corollary A.5]{KR09}. The general case can be proved
  analogously or by applying a change of quantisation to the special
  case.
\end{proof}

 In particular, we use the estimate on the remainder
 \eqref{eq:SharpProdAsExpRemain} for estimates uniform with respect to a parameter.
 We remark that in the usual non-discrete setting similar statements on the remainder hold and are 
 used in the proof of Proposition \ref{prop:discrPseudoInvShort}.

 The next proposition is a discrete version of the Theorem of
 Calder{\'o}n-Vaillancourt stating that pseudodifferential operators
 induced by a bounded symbol are bounded (see \cite[Theorem
   7.11]{DiSj}). Considering the $t$-quantisation as a special case of
 the general quantisation \eqref{eq:OpaNonDiscrDefgen}
 described below Proposition
 \ref{prop:RestrForm}, Proposition \ref{prop:discrCaldVail} is a
 special case of \cite[Corollary A.6]{KR18}.
 
\begin{prop}
  \label{prop:discrCaldVail}
  Let $t \in [0,1]$ and $a \in \Sy^0(1,\epsilon_0)(\R^d \times  \T^d)$. 
  Then for any $\epsilon \in (0,\epsilon_0]$ the operator
    $\Op_{\epsilon,t}^\T a$ can be extended to a bounded operator
    $\Op_{\epsilon,t}^\T a : \l^2(\epsilon\Z^d) \rightarrow
    \l^2(\epsilon\Z^d)$.

  Moreover, there exists a constant $M>0$ depending only on (upper bounds for) a finite number of
  Fr{\'e}chet seminorms of the symbol $a$ such that 
  \begin{align}
    \label{eq:discrCaldVailEstim}
    \norm{ \Op_{\epsilon,t}^\T a } \leq M
    \qquad (\epsilon \in (0,\epsilon_0],\, t \in [0,1]).
  \end{align}
\end{prop}

In the next proposition, we analyse the symbol of an operator
conjugated with an oscillating term $e^{i\phi/\epsilon}$ with focus on
its asymptotic expansion. The proposition and its proof resemble
\cite[Proposition A.7]{KR18}, where conjugation with an amplitude
$e^{\phi/\epsilon}$ is treated. We use Proposition
\ref{prop:OpUnitConj} as a tool to approximate the time evolution of
the parametrix in Section \ref{subsec:MainProof2}.

\begin{prop}
  \label{prop:OpUnitConj}
  Let $q \in \Sy^0(m,\epsilon_0)(\R^d \times \T^d)$ be a symbol with
  asymptotic expansion $q \sim \sum_{j=0}^\infty \epsilon^j q_j$. Let
  $T > 0$ and let $\phi \in \Cont^\infty((-T,T) \times \R^d \times
  \R^d, \R)$ be such that the map
  \begin{align}
      \label{eq:gradPhiAss}
      (t,x,\eta) \mapsto \nabla_x\phi (t,x,\eta) - \eta
    \end{align}
    is $2\pi\Z^d$-periodic with respect to $\eta$ with all derivatives
    being bounded. Fix $s \in [0,1]$. \\ Then there is a family
    $(\tilde{q}_{t,\eta})_{t \in (-T,T),\, \eta \in \R^d}$ of symbols
    $\tilde{q}_{t,\eta} \in \Sy^0(m,\epsilon_0)(\R^d \times \T^d)$,
    $2\pi\Z^d$-periodic in the parameter $\eta$, such that
    \begin{align}
      \label{eq:ConjOpSymbol}
      e^{i\phi(t,\cdot,\eta)/\epsilon} \Op_{\epsilon,s}(q)
      e^{-i\phi(t,\cdot,\eta)/\epsilon} = \Op_{\epsilon,s}
      (\tilde{q}_{t,\eta}) \qquad (t \in (-T,T),\, \eta \in \R^d)
    \end{align}
    and satisfying for any $\alpha \in \N^{1+3d}$
    \begin{align}
        \label{eq:unifSymb}
        \sup_{t,\eta,x,\xi,\epsilon}
        \frac{|\partial^\alpha_{t,\eta,x,\xi}
          \tilde{q}_{t,\eta}(x,\xi;\epsilon) |}{m(x,\xi)} < \infty
    \end{align}
    Moreover, for some sequence $(\tilde{q}_{t,\eta;j})_{j \in \N}$ of
    $\epsilon$-independent symbols in $\Sy^0(m,\epsilon_0)(\R^d \times
    \T^d)$
    \begin{align}
      \label{eq:unifAsympExp}
      \sup_{t,\eta,x,\xi,\epsilon} \frac{\left|\partial^\alpha_{t,\eta,x,\xi} \left( \tilde{q}_{t,\eta}
        - \sum_{j=0}^{N-1} \epsilon^j \tilde{q}_{t,\eta;j} \right)(x,\xi;\epsilon) \right|}{\epsilon^N
        m(x,\xi)} < \infty
    \end{align}
      holds for any $N \in \N^*$ with leading order term given by
    \begin{align}
      \label{eq:LeadOrderTerm}
      \tilde{q}_{t,\eta;0} (x,\xi) = q_0(x,\xi + \nabla_x \phi(t,x,\eta)).
    \end{align}
    An analog of \eqref{eq:unifSymb} holds for all $\tilde{q}_{t,\eta;j}$.
\end{prop}

\begin{proof}
  The operator $ e^{i\phi(t,\cdot,\eta)/\epsilon} \Op_{\epsilon,s}(q)
  e^{-i\phi(t,\cdot,\eta)/\epsilon}$ is characterised by its
  distributional kernel using \eqref{eq:OpWithKernel}. This in turn is completely described by its 
  pointwise defined kernel $K(x,y)$ for $x \in \R^d$, $ y \in \epsilon\Z^d + x$ introduced in \eqref{kernel}, i.e.
  
  \begin{align}
    K(x,y) &:= (2\pi)^{-d} \int_{\T^d} e^{i((y-x)\xi +
      \phi(t,x,\eta) - \phi(t,y,\eta))/\epsilon} q(sx +
    (1-s)y,\xi;\epsilon) d\xi \notag \\ \label{eq:IntKernConjOp} &=
    (2\pi)^{-d} \int_{\T^d} e^{i(y-x)(\xi -
      \Phi(t,\eta,x,y))/\epsilon} q(sx + (1-s)y,\xi;\epsilon) d\xi
  \end{align}
  where
  \begin{align}
    \label{eq:PHI_conjProof}
    \Phi(t,\eta,x,y) := \int_0^1 (\nabla_x \phi)(t,(1-\tau)y + \tau x,
    \eta) d\tau.
  \end{align}
  Substituting $\tilde{\xi} := \xi - \Phi(t,\eta,x,y)$ and using
  that the integrand in \eqref{eq:IntKernConjOp} is $2\pi\Z^d$-periodic with respect to $\xi$
  (note that $e^{i(y-x)\xi/\epsilon}$ is $2\pi \Z^d$-periodic since $y-x \in \epsilon\Z^d$), we get
  \begin{align}
    \notag \mbox{rhs} \eqref{eq:IntKernConjOp} &= (2\pi)^{-d}
    \int_{\T^d - \Phi(t,\eta,x,y)} e^{i(y-x)\tilde{\xi}/\epsilon} q(sx
    + (1-s)y,\tilde{\xi} + \Phi(t,\eta,x,y);\epsilon) d\tilde{\xi}
    \\ \label{eq:IntKernConjOp2} &= (2\pi)^{-d} \int_{\T^d}
    e^{i(y-x)\tilde{\xi}/\epsilon} q(sx + (1-s)y,\tilde{\xi} +
    \Phi(t,\eta,x,y);\epsilon) d\tilde{\xi}.
  \end{align}
  
  Thus the operator on the lhs of \eqref{eq:ConjOpSymbol} might be seen as a special case of the 
  general quantisation of a symbol $a(x,y,\xi)$, which by \eqref{eq:genquant} can be expressed as 
  a pseudodifferential operator in s-quantisation. It remains to check that we actually work in 
  the appropriate symbol spaces. More precisely, we proceed as follows.
  
  Since by assumption all derivatives of \eqref{eq:gradPhiAss} are
  bounded, it follows that all derivatives of the non-oscillating factor
  of the integrand in rhs\eqref{eq:IntKernConjOp2}, namely the
  derivatives of the function
  \begin{align}
    \label{eq:ExtSymbolConjOp}
    (t,x,y,\xi,\eta) \mapsto q(sx + (1-s)y, \xi + \Phi(t,\eta,x,y);
    \epsilon),
  \end{align}
  are of order $\mathcal{O}(m(sx + (1-s)y, \xi))$,
  uniformly in $\epsilon$. 
  
  By \cite[Proposition A.5]{KR18} (see our equation \eqref{eq:genquant} above) we then find
  $\tilde{q}_{t,\eta} \in \Sy^0(m,\epsilon_0)(\R^d \times \T^d)$ which
  is the symbol of the operator associated to the kernel
  \eqref{eq:IntKernConjOp2} in $s$-quantisation, i.e.,
  $\tilde{q}_{t,\eta}$ satisfies \eqref{eq:ConjOpSymbol}.

  The uniformity assertions in \eqref{eq:unifSymb} and
  \eqref{eq:unifAsympExp} for $\tilde{q}_{t,\eta}$ and its expansion
  terms $\tilde{q}_{t,\eta;j}$ with respect to $t$ and $\eta$ follow
  from the $\Cont^\infty$-assumptions on $\phi$ combined with the
  continuity statement in \cite[Proposition A.5]{KR18} for the mapping
  $a \mapsto a_s$ in the relevant Fréchet topology of symbols. Note
  that periodicity in the parameter $\eta$ follows from the
  periodicity of \eqref{eq:gradPhiAss} and $q$ used in the
  representation formula \eqref{eq:IntKernConjOp2}.  
  Evaluating
  \eqref{eq:ExtSymbolConjOp}  at $x = y$ gives the leading order
  term \eqref{eq:LeadOrderTerm}. 
\end{proof}

 \section{Poisson summation and application}
\label{sec:appPoisson}

We recall the well-known Poisson's
summation formula (see e.g. \cite[Section 7.2]{hoermander2015analysis}.

\begin{prop}
  \label{prop:Pois}
  Let $u \in \S(\R^d)$ and $a > 0$. Then
  \begin{align}
    \sum_{x \in a\Z^d} u(x) = a^{-d} (2\pi)^{d/2} \sum_{\xi \in
      \frac{2\pi}{a} \Z^d} (\F u)(\xi),
  \end{align}
  where $(\F u)(\xi) = (2\pi)^{-d/2} \int_{\R^d} e^{-ix\xi} u(x) dx$.
\end{prop}

This formula is the main tool to prove the following Proposition \ref{prop:PoisApp}, which gives a sufficient condition on a
phase function to approximate the associated oscillating sum by an
integral with small remainder. We note that a similar approximation
(for non-oscillating sums) has been given in \cite[Lemma 4.1]{KR18}. 
We use Proposition \ref{prop:PoisApp} in the proofs of
Theorem \ref{theo:TraceAsympExp} and Lemma \ref{lem:I1I2eps} to
transform sums into integrals interpretable as standard phase space
volumes.

\begin{prop}
  \label{prop:PoisApp}
  Let $\epsilon_0 \in (0,1]$ and $a \in \Sy^0(1,\epsilon_0)(\R^d)$
  with support in some compact $K \subset \R^d$, uniformly in
  $\epsilon \in (0,\epsilon_0]$. Let $\varphi \in \Cont^\infty(\R^d)$ be
  real-valued with 
  \begin{align}
    \label{eq:PoisAppCond}
    \sup_{\substack{j \in \left\{1, \dots, d\right\} \\ x \in K
      }}\abs{\partial_j \varphi(x)} < 2\pi
  \end{align}
  Then
  \begin{align}
    \label{eq:PoisApp}
    \abs{\epsilon^d \sum_{x \in \epsilon\Z^d} e^{i\varphi(x)/\epsilon}
    a(x;\epsilon) - \int_{\R^d}e^{i\varphi(x)/\epsilon} a(x;\epsilon)
    dx } = \O(\epsilon^\infty) \qquad (\epsilon \downarrow 0).
  \end{align}
  More precisely, for any $k \in \N$ with $k \geq d$ an error bound in
  \eqref{eq:PoisApp} is given by
  \begin{align}
    \label{eq:PoisAppBound}
    \epsilon^{2k} \sum_{\xi \in 2\pi \Z^d \setminus
      \left\{0\right\}} \int_{K} \abs{ (\mathbf{W}_\epsilon^k
      a)(x,\xi)} dx  \qquad (\epsilon \in (0,\epsilon_0]),
  \end{align}
  where the operator $\mathbf{W}_\epsilon$ acts via
  \begin{align}
    \label{eq:PoisAppDefW}
    \left( \mathbf{W}_\epsilon a \right)(x,\xi) = \nabla_{x}^2 \left(
      \frac{a(x;\epsilon)}{w_\epsilon(x,\xi)} \right) \qquad \mbox{
      with } \quad w_\epsilon(x,\xi) = \sum_{j=1}^d \left(
      i\epsilon \partial_j^2\varphi(x) - (\partial_j\varphi(x) -
      \xi_j)^2 \right).
  \end{align}
\end{prop}

\begin{proof}
  By Proposition \ref{prop:Pois}
  \begin{align}
    \epsilon^d \sum_{x \in \epsilon\Z^d} e^{i\varphi(x)/\epsilon}
    a(x;\epsilon) - \int_{\R^d} e^{i\varphi(x)/\epsilon} a(x;\epsilon)
    dx = \sum_{\xi \in 2\pi\Z^d \setminus \left\{0\right\}}
      \int_{\R^d}e^{i(\varphi(x) - x\xi)/\epsilon} a(x;\epsilon) dx.
  \end{align}
  Due to condition \eqref{eq:PoisAppCond}, $1/w_\epsilon$ is
  bounded on $K \times (2\pi\Z^d \setminus \left\{ 0 \right\})$,
  uniformly in $\epsilon \in (0,\epsilon_0]$. Using the identity
  \begin{align}
    \frac{\epsilon^2 \nabla_x^2}{w_\epsilon(x,\xi)} e^{i(\varphi(x) -
      x\xi)/ \epsilon} = e^{i(\varphi(x) - x\xi)/ \epsilon} \qquad (x
    \in K, \, \xi \in 2\pi\Z^d \setminus \left\{ 0 \right\} ),
  \end{align}
  integration by parts yields
  \begin{align}
    \abs{\sum_{\xi \in 2\pi\Z^d \setminus \left\{0\right\}}
      \int_{\R^d}e^{i(\varphi(x) - x\xi)/\epsilon} a(x;\epsilon) dx}
    &= \epsilon^{2k} \abs{\sum_{\xi \in 2\pi\Z^d \setminus
        \left\{0\right\}} \int_{K}e^{i(\varphi(x) - x\xi)/\epsilon}
      (\mathbf{W}_\epsilon^k a)(x,\xi) dx} \nonumber\\
    &\leq \epsilon^{2k} \sum_{\xi \in 2\pi \Z^d \setminus
      \left\{0\right\}} \int_{K} \abs{ (\mathbf{W}_\epsilon^k
      a)(x,\xi)} dx,
  \end{align}
  where the last expression is finite and of order $\O(\epsilon^{2k})$
  for large $k \geq d$.
\end{proof}

\end{appendix}


\nocite{*}




\begin{thebibliography}{H{\"o}r85b}


\bibitem[Ag79]{a} Shmuel  Agmon.
\newblock  {\em Some new results in spectral and scattering theory of differential operators on $\R^n$.}
\newblock { S{é}m. Goulaouic-Schwartz} 1978-1979, Exp. IUI, 1--11


\bibitem[AK92]{ak} Shmuel  Agmon,  Markus Klein.
\newblock  {\em Analyticity properties in scattering and spectral theory for Schrödinger operators with long-range radial potentials.}
\newblock Duke Mathematical Journal 68(2): 337--399,  1992.

\bibitem[AE08]{AmEsch09}
Herbert Amann and Joachim Escher.
\newblock {\em Analysis III}.
\newblock Birkh{\"a}user, Basel, Boston, Berlin, 2008.


\bibitem[BCR24]{bcr}
L.N.A.  Botchway,   M. Chatzakou  and M. Ruzhansky. \newblock {\em Semi-classical Pseudo-differential Operators on $h \Z^n$ and
Applications} 
\newblock  Journal of Fourier Analysis and Applications  30:41, (2024)


\bibitem[BKR20]{bkr}
L.N.A. Botchway, P.G. Kibiti and  M. Ruzhansky. 
\newblock {\em Difference operators and pseudo-differential operator on $\Z^n$ .} 
\newblock  J. Funct. Anal. 278(11), 108473 (2020)




\bibitem[BdH15]{bov}  
Anton Bovier, Frank den Hollander.
\newblock {\em Metastability. A Potential-Theoretic approach.}
\newblock  Grundlehren der mathematischen Wissenschaften 351, Springer, 2015


\bibitem[BEGK01]{begk1} 
Anton Bovier, Michael Eckhoff, Veronique Gayrard, Markus Klein. 
\newblock {\em Metastability in
stochastic dynamics of disordered mean-field models}.
\newblock Probability Theory and Related Fields 119, p. 99-161, (2001)

\bibitem[BEGK02]{begk2} 
Anton Bovier, Michael Eckhoff, Veronique Gayrard, Markus Klein. 
{\em Metastability and
low lying spectra in reversible Markov chains}.
\newblock Communications in Mathematical Physics 228, p. 219-255, (2002)


\bibitem[CdV73]{cdv1} 
Y. Colin de Verdi{è}re.
\newblock {\em Spectre du laplacien et longeur des g{é}od{é}siques p{é}riodiques.}
\newblock  Composition Mathematica, 27: 83--106, 1973



\bibitem[CdV85]{cdv2} Yves Colin de Verdière.
\newblock {\em Ergodicité et fonctions propres du Laplacien.}
\newblock   Séminaire Bony-Sjöstrand-Meyer 1984–85 no XIII. Commun. Math. Phys.,  102: 497–-502 , 1985



\bibitem[Cha74]{chaz74}
Jacques Chazarain.
\newblock {\em Formules de Poisson pour les vari{\'e}t{\'e}s riemanniens.}
\newblock  Inventiones mathematicae, 24: 65 -- 82, 1974


\bibitem[Cha80]{Chaz1980}
Jacques Chazarain.
\newblock {\em Spectre d'un hamiltonien quantique et m{\'e}canique classique.}
\newblock Communications in Partial Differential Equations,
  5(6):595--644, 1980.

\bibitem[DG75]{DuGu75Spectrum}
Johannes~Jisse Duistermaat and Victor~William Guillemin.
\newblock {\em The spectrum of positive elliptic operators and periodic
  bicharacteristics.}
\newblock  Inventiones mathematicae, 29:39--79, 1975.

\bibitem[DH72]{10.1007/BF02392165}
Johannes~Jisse Duistermaat and Lars H{\"o}rmander.
\newblock {\em Fourier integral operators. II}.
\newblock { Acta Mathematica}, 128:183 -- 269, 1972.

\bibitem[Die08]{dieu2008}
Jean Dieudonn{\'e}.
\newblock {\em Foundations of Modern Analysis}.
\newblock Pure and Applied Mathematics. Read Books, 2008.



\bibitem[DS99]{DiSj}
Mouez Dimassi and Johannes Sj{\"o}strand.
\newblock {\em Spectral Asymptotics in the Semi-Classical Limit}.
\newblock London Mathematical Society Lecture Note Series Number 268.
  Cambridge University Press, New York, 1999.

\bibitem[Dui96]{duistermaat1995fourier}
Johannes~Jisse Duistermaat.
\newblock {\em Fourier Integral Operators}.
\newblock Progress in Mathematics. Birkh{\"a}user Boston, 1996.

\bibitem[Eva98]{evans1998partial}
Lawrence~Craig Evans.
\newblock {\em Partial Differential Equations}.
\newblock Graduate studies in mathematics. American Mathematical Society, 1998.

\bibitem[Fol95]{foll95pde}
Gerald~Budge Folland.
\newblock {\em Introduction to Partial Differential Equations}.
\newblock Princeton University Press vol. 102, 2nd edition, 1995.

\bibitem[GdG13]{gdg1}
Giacomo di Ges{ù}.
\newblock {\em Semiclassical spectral analysis of discrete Witten Laplacians}.
\newblock  Thesis, 2013, \url{https://publishup.uni-potsdam.de/opus4-ubp/frontdoor/index/index/docId/6287}


\bibitem[GdG23]{gdg2}
Giacomo di Ges{ù}.
\newblock {\em Spectral Analysis of Discrete Metastable Diffusions} 
\newblock  {Communications in Mathematical Physics} 402 543--580, 2023. 




\bibitem[GK69]{GohK69}
Israel Gohberg and Mark~Grigor'evi{\v{c}} Krejn.
\newblock {\em Introduction to the Theory of Linear Non-Self-Adjoint
  Operators}.
\newblock AMS, Providence, RI, USA, 1969.

\bibitem[GS94]{Grigis_Sjoestrand_1994}
Alain Grigis and Johannes Sj{\"o}strand.
\newblock {\em Microlocal Analysis for Differential Operators: An
  Introduction}.
\newblock London Mathematical Society Lecture Note Series, 196. Cambridge
  University Press, 1994.

\bibitem[GV64]{GeVi64}
Israel~Moiseevich Gel'fand and Naum~Yakovlevich Vilenkin.
\newblock {\em Generalized Functions: Volume 4. Applications of Harmonic
  Analysis}.
\newblock Academic Press, Inc., New York, London, 1964.

\bibitem[H{\"o}r68a]{hoermander1968LecNotes}
Lars H{\"o}rmander.
\newblock Lecture notes at the nordic summer school of mathematics, 1968.

\bibitem[H{\"o}r68b]{Hoer1968Spectral}
Lars H{\"o}rmander.
\newblock {\em The spectral function of an elliptic operator.}
\newblock { Acta Mathematica}, 121:193--218, 1968.


\bibitem[H{\"o}r76]{hscat}
Lars H{\"o}rmander.
\newblock {\em The existence of wave operators in scattering theory.}
\newblock { Mathematische Zeitschrift}, 146:  69--91, 1976.

\bibitem[H{\"o}r83]{hormander1983analysis}
Lars H{\"o}rmander.
\newblock {\em The Analysis of Linear Partial Differential Operators, II:
  Differential Operators with Constant Coefficients}.
\newblock Grundlehren der mathematischen Wissenschaften, 257. Springer Berlin
  Heidelberg, 1983.

\bibitem[H{\"o}r85a]{hormander1985analysis}
Lars H{\"o}rmander.
\newblock {\em The Analysis of Linear Partial Differential Operators, III:
  Pseudo-differential Operators}.
\newblock Grundlehren der mathematischen Wissenschaften, 274. Springer Berlin
  Heidelberg, 1985.

\bibitem[H{\"o}r85b]{hoermander2009analysis}
Lars H{\"o}rmander.
\newblock {\em The Analysis of Linear Partial Differential Operators IV:
  {Fourier} Integral Operators}.
\newblock Grundlehren der mathematischen Wissenschaften, 275. Springer Berlin
  Heidelberg, 1985.

\bibitem[H{\"o}r90]{hoermander2015analysis}
Lars H{\"o}rmander.
\newblock {\em The Analysis of Linear Partial Differential Operators I:
  Distribution Theory and {Fourier} Analysis}.
\newblock Grundlehren der mathematischen Wissenschaften, 256. Springer Berlin
  Heidelberg, 1990.


\bibitem[HMR]{hmr} Bernard Helffer, Andr{\'e} Martinez, Didier Robert.
\newblock  {\em Ergodicité et limite semi-classique.}
\newblock {Communications in Mathematical Physics}, 109: 313--326,   1987.

\bibitem[HR81]{hero81}
Bernard Helffer and Didier Robert.
\newblock Comportement   semi-classiques du spectre des hamiltoniens quantiques elliptiques.
\newblock {\em Annales Institut Fourier, Grenoble}, 31(3):169--223, 1981.


\bibitem[HR83]{HELFFER1983246}
Bernard Helffer and Didier Robert.
\newblock Calcul fonctionnel par la transformation de {Mellin} et
  op{\'e}rateurs admissibles.
\newblock {\em Journal of Functional Analysis}, 53(3):246--268, 1983.


\bibitem[IK85]{ik}
Hiroshi Isozaki, Hitoshi Kitada.
\newblock  {\em Modified wave operators with time-independent modifiers.}
\newblock  { Journal Faculty of Science, University Tokyo, Sect. IA Math} 32.1: 77--104,  1985.



\bibitem[Ivr98]{ivrii1998}
Victor Ivrii.
\newblock {\em Microlocal Analysis and Precise Spectral Asymptotics}.
\newblock Springer, Berlin, Heidelberg, 1998.




\bibitem[Ivr19]{ivmonster}
Victor Ivrii. 
\newblock  {\em Microlocal Analysis, Sharp Spectral Asymptotics and Applications vol. I -V.} 
\newblock  Springer Monographs in Mathematics,  2019 

\bibitem[Kam23]{Ka23}
Kentaro Kameoka.
\newblock {\em Semiclassical analysis and the {Agmon-Finsler} metric for discrete
  {Schr{\"o}dinger} operators.}
\newblock {Communications on Pure and Applied Analysis}, 22(4):1180--1193,
  2023.


\bibitem[KN24]{nakam}
Kentaro Kameoka and Shu Nakamura.
\newblock {\em Continuum limit of resonances for discrete Schr\"{o}dinger operators.}
\newblock {Journal of Mathematical Physics} 66(7), 2025. 



\bibitem[KMW93]{kmw} 
Markus Klein, André Martinez and Xue Ping Wang.
\newblock {\em On the Born-Oppenheimer approximation of wave operators in molecular scattering theory}. 
\newblock {Communications in Mathematical Physics} 152: 73–-95,  1993.


\bibitem[KLR14]{KRL14}
Markus Klein, Christian L{\'e}onard, and Elke Rosenberger.
\newblock {\em Agmon-type estimates for a class of jump processes.}
\newblock { Mathematische Nachrichten}, 287(17-18):2021--2039, 2014.

\bibitem[KR08]{KR08}
Markus Klein and Elke Rosenberger.
\newblock {\em Agmon-type estimates for a class of difference operators.}
\newblock {Annales Henri Poincar{\'e}}, 9:1177--1215, 01 2008.

\bibitem[KR09]{KR09}
Markus Klein and Elke Rosenberger.
\newblock {\em Harmonic approximation of difference operators.}
\newblock {Journal of Functional Analysis}, 257(11):3409--3453, 2009.

\bibitem[KR11]{KR11}
Markus Klein and Elke Rosenberger.
\newblock {\em Asymptotic eigenfunctions for a class of difference operators.}
\newblock {Asymptotic Analysis}, 73(1-2):1--36, 2011.

\bibitem[KR12]{KR12}
Markus Klein and Elke Rosenberger.
\newblock {\em Tunneling for a class of difference operators.}
\newblock {Annales Henri Poincar{\'e}}, 13(5):1231--1269, 2012.

\bibitem[KR16]{KR16}
Markus Klein and Elke Rosenberger.
\newblock {\em {Agmon} estimates for the difference of exact and approximate
  {Dirichlet} eigenfunctions for difference operators.}
\newblock { Asymptotic Analysis}, 97(1-2):61--89, mar 2016.

\bibitem[KR18]{KR18}
Markus Klein and Elke Rosenberger.
\newblock {\em Tunneling for a class of difference operators: Complete asymptotics.}
\newblock {Annales Henri Poincar{\'e}}, 19:3511--3559, 11 2018.

\bibitem[Mar02]{Mart}
Andr{\'e} Martinez.
\newblock {\em An Introduction to Semiclassical and Microlocal Analysis}.
\newblock Universitext. Springer-Verlag, New York, 2002.

\bibitem[Mat71]{mather1971}
John~Norman Mather.
\newblock {\em On {Nirenberg}'s proof of {Malgrange}'s preparation theorem.}
\newblock { Proceedings of Liverpool
  Singularities --- Symposium I}, pages 116--120, Berlin, Heidelberg, 1971.
  Springer Berlin Heidelberg.

\bibitem[Nak14]{na}
Shu Nakamura.
\newblock {\em Modified wave operators for discrete {Schr{\"o}dinger} operators with
  long-range perturbations.}
\newblock {Journal of Mathematical Physics}, 55(11):112101, 11 2014.


\bibitem[Rob82]{ROBERT198274}
Didier Robert.
\newblock {\em Calcul fonctionnel sur les op{\'e}rateurs admissibles et application.}
\newblock {Journal of Functional Analysis}, 45(1):74--94, 1982.

\bibitem[Rob87]{Robert1987AutourDL}
Didier Robert.
\newblock {\em Autour de l'approximation semi-classique}.
\newblock Progress in mathematics. Birkh{\"a}user, 1987.

\bibitem[RT87]{rt}
Didier Robert and  Hideo Tamura .
\newblock  {\em Semi-classical estimates for resolvents and asymptotics
for total scattering cross-sections.}
\newblock {Annales Henri Poincar{\'e}, section A, tome} 46, no 4: 415--442, 1987.

\bibitem[RS75]{ReSi2}
Michael Reed and Barry Simon.
\newblock {\em Methods of Modern Mathematical Physics. Vol. 2: {Fourier}
  Analysis, Self-Adjointness}.
\newblock Academic Press, Inc., San Diego, California, 1975.

\bibitem[RS78]{ReSi4}
Michael Reed and Barry Simon.
\newblock {\em Methods of Modern Mathematical Physics. Vol. 4: Analysis of
  Operators}.
\newblock Academic Press, Inc., San Diego, California, 1978.

\bibitem[RS79]{ReSi3}
Michael Reed and Barry Simon.
\newblock {\em Methods of Modern Mathematical Physics. Vol. 3: Scattering
  Theory.}
\newblock Academic Press, Inc., San Diego, California, 1979.

\bibitem[RS80]{ReSi}
Michael Reed and Barry Simon.
\newblock {\em Methods of Modern Mathematical Physics. Vol. 1: Functional
  Analysis.}
\newblock Academic Press, Inc., San Diego, California, 1980.

\bibitem[R06]{thesis} 
Elke Rosenberger. 
\newblock {\em Asymptotic Spectral Analyis and Tunnelling for a class
of Difference Operators.} Thesis 2006, \url{http://nbn-resolving.de/urn:nbn:de:kobv:517-opus-7393}

\bibitem[RT09]{RuzhTuru2009}
Michael Ruzhansky and Ville Turunen.
\newblock {\em Pseudo-Differential Operators and Symmetries.}
\newblock Birkh{\"a}user, Basel, 2009.

\bibitem[RT10]{RuzhTuru2010}
Michael Ruzhansky and Ville Turunen.
\newblock {\em Quantization of Pseudo-differential Operators on the Torus.}
\newblock {Journal of Fourier Analysis and Applications} (2010) 16: 943–982

\bibitem[Shn74]{sh}
Alexander Shnirelman. 
\newblock {\em Ergodic properties of eigenfunctions. }
\newblock {Russian mathematical surveys},  29: 181–-182,  1974
 
 
\bibitem[Shu01]{shu} Mikhail A. Shubin.
\newblock  {\em Pseudodifferential Operators and Spectral Theory.}
\newblock  Springer, 2nd edition, 2001

\bibitem[TS73]{shut} V.N. Tulovskii and Mikhail A. Shubin.
\newblock  {\em On the asymptotic distribution of eigenvalues of pseudodifferential operators in $\R^n$.} 
\newblock  {Mathematics USSR Sbornik}, 21: 565-- 583, 1973.


\bibitem[Sim05]{SimonTI}
Barry Simon.
\newblock {\em Trace ideals and their applications.}
\newblock { Mathematical Surveys and Monographs, volume 120},
Providence, RI: American Mathematical Society, 2nd ed. edition,
  2005.

\bibitem[Sti58]{Stinespring1958}
William~Forrest Stinespring.
\newblock {\em A sufficient condition for an integral operator to have a trace.}
\newblock {Journal f{\"u}r die reine und angewandte Mathematik},
  200:200--207, 1958.

\bibitem[Tad19]{tad}
Yukihide Tadano.
\newblock {\em Long-range scattering for discrete {Schr{\"o}dinger} operators.}
\newblock { Annales Henri Poincar{\'e}}, 20:1439--1469, 05 2019.

\bibitem[Wey11]{Weyl1911}
Hermann Weyl.
\newblock {\em {\"U}ber die asymptotische {Verteilung} der {Eigenwerte}.}
\newblock {Nachrichten von der Gesellschaft der Wissenschaften zu
  G{\"o}ttingen, Mathematisch-Physikalische Klasse}, 1911:110--117, 1911.

\bibitem[Wey12]{Weyl1912}
Hermann Weyl.
\newblock {\em Das asymptotische {Verteilungsgesetz} der {Eigenwerte} linearer
  partieller {Differentialgleichungen} (mit einer {Anwendung} auf die {Theorie}
  der {Hohlraumstrahlung}).}
\newblock {Mathematische Annalen}, 71:441--479, 1912.

\bibitem[Yos78]{Yos78}
K{\=o}saku Yosida.
\newblock {\em Functional Analysis}.
\newblock Springer-Verlag, Berlin, fifth edition, 1978.
\newblock Grundlehren der Mathematischen Wissenschaften, Band 123.

\end{thebibliography}
\end{document}